\documentclass[11pt]{article}
\usepackage[dvipdfmx]{graphicx}
\usepackage{amsmath}
\usepackage{amssymb}
\usepackage{amsthm}
\usepackage{bm}
\usepackage{url}
\usepackage{float}
\usepackage{mathtools}
\mathtoolsset{showonlyrefs=true}
\usepackage{subcaption}
\usepackage{mathrsfs}
\usepackage[usenames]{color}
\usepackage[table,xcdraw]{xcolor}
\newtheorem{dfn}{Definition}[section]
\newtheorem{thm}[dfn]{Theorem}

\newtheorem{lem}[dfn]{Lemma}
\newtheorem{cor}[dfn]{Corollary}
\newtheorem{rem}[dfn]{Remark}

\newtheorem{proc}[dfn]{Procedure}

\newcommand{\R}{\mathbb{R}}

\newcommand{\Z}{\mathbb{Z}}
\newcommand{\N}{\mathbb{N}}
\newcommand{\C}{\mathbb{C}}
\newcommand{\bx}{\bar{x}}
\newcommand{\ba}{\bar{a}}
\newcommand{\bU}{\bar{U}}

\newcommand{\cB}{\mathcal{B}}
\newcommand{\cD}{\mathcal{D}}
\newcommand{\cE}{\mathcal{E}}
\newcommand{\cG}{\mathcal{G}}
\newcommand{\cI}{\mathcal{I}}
\newcommand{\cL}{\mathcal{L}}
\newcommand{\cN}{\mathcal{N}}
\newcommand{\cP}{\mathcal{P}}
\newcommand{\cQ}{\mathcal{Q}}
\newcommand{\cS}{\mathcal{S}}
\newcommand{\cT}{\mathcal{T}}
\newcommand{\cX}{\mathcal{X}}
\newcommand{\ta}{\tilde{a}}

\newcommand{\bydef}{\,\stackrel{\mbox{\tiny\textnormal{\raisebox{0ex}[0ex][0ex]{def}}}}{=}\,} 

\newcommand{\im}{\mathrm{i}}

\newcommand{\bfj}{{\boldsymbol j }}
\newcommand{\bfi}{{\boldsymbol i }}
\newcommand{\bk}{{\bm{k}}}
\newcommand{\bj}{{\bm{j}}}
\newcommand{\bL}{{\bm{L}}}
\newcommand{\bell}{{\bm{\ell}}}

\newcommand{\bN}{{\bm{N}}}
\newcommand{\bM}{{\bm{M}}}
\newcommand{\bn}{{\bm{n}}}
\newcommand{\bbm}{{\bm{m}}}
\newcommand{\Fm}{{\bm{F_m}}}
\newcommand{\FM}{{\bm{F_M}}}
\newcommand{\FN}{{\bm{F_N}}}
\newcommand{\kk}{\text{\textbf{k}}}

\numberwithin{equation}{section}
\usepackage[top=1in, bottom=1in, left=1in, right=1in]{geometry}
\newcommand{\ak}[1]{{#1}}
\newcommand{\jp}[1]{{#1}}
\newcommand{\doverline}[1]{\overline{\overline{#1}}}
\begin{document}
\title{
A rigorous integrator and global existence for higher-dimensional semilinear parabolic PDEs via semigroup theory}
\author{
Gabriel William Duchesne \thanks{Department of Mathematics and Statistics, McGill University, 805 Sherbrooke West, Montreal, QC H3A 0B9, Canada (\texttt{gabriel.duchesne@mail.mcgill.ca})
}
\and
Jean-Philippe~Lessard \thanks{Department of Mathematics and Statistics, McGill University, 805 Sherbrooke West, Montreal, QC H3A 0B9, Canada (\texttt{jp.lessard@mcgill.ca})
}
\and
Akitoshi~Takayasu \thanks{Institute of Systems and Information Engineering, University of Tsukuba, 1-1-1 Tennodai, Tsukuba, Ibaraki 305-8573, Japan (\texttt{takitoshi@risk.tsukuba.ac.jp})
}
}
\date{\today}
\maketitle

\begin{abstract}
In this paper, we introduce a general constructive method to compute solutions of initial value problems of semilinear parabolic partial differential equations \jp{on hyper-rectangular domains} via semigroup theory and computer-assisted proofs. Once a numerical candidate for the solution is obtained via a finite dimensional projection, Chebyshev series expansions are used to solve the linearized equations about the approximation from which a solution map operator is constructed. Using the solution operator (which exists from semigroup theory), we define an infinite dimensional contraction operator whose unique fixed point together with its rigorous bounds provide the local inclusion of the solution. Applying this technique for multiple time steps leads to constructive proofs of existence of solutions over long time intervals. As applications, we study the 3D/2D Swift-Hohenberg, where we combine our method with explicit constructions of trapping regions to prove global existence of solutions of initial value problems converging asymptotically to nontrivial equilibria. A second application consists of the 2D Ohta-Kawasaki equation, providing a framework for handling derivatives in nonlinear terms.
\end{abstract}



%

\section{Introduction} \label{sec:intro}

Studying global dynamics and global existence (forward in time) of solutions of higher-dimensional semilinear parabolic partial differential equations (PDEs) is a central problem in mathematics. Perhaps the most striking example is the famous millennium prize problem which raises the question of global existence of solutions of initial value problems (IVPs) in the three-dimensional Navier-Stokes equations~\cite{MR2238274}. From the perspective of dynamical systems and as made clear from the pioneering work of Poincaré on the three-body problem \cite{Poincare}, to understand the global dynamics of a given nonlinear differential equation, it is crucial to grasp the existence of asymptotic objects such as equilibria, periodic orbits and connections between them. This is especially true for parabolic PDEs modeling phenomena in material science, spatial ecology and fluids mechanics, where patterns arise as asymptotic limits, often within global attractors, of solutions of initial value problems. For all its importance and popularity, the problem consisting of understanding the global dynamics of a given PDE is a notoriously difficult task since the model is nonlinear and naturally lead to the notion of an infinite dimensional dynamical system. 

In this paper, we introduce a constructive method to solve rigorously initial value problems of \jp{scalar} semilinear parabolic PDEs of the form 
\begin{equation} \label{eq:general_PDE}
u_t = (\lambda_0  + \lambda_1 \Delta  + \lambda_2 \Delta^2)u + \jp{\Delta^p N(u)}
\end{equation}
where $p \in \{0,1\}$ and $N$ is a polynomial satisfying \ak{both $N(0)=0$ and its Fr\'echet derivative $DN(0)=0$}, and we provide a computational framework to show that solutions exist globally forward in time. While in this work we focus on the scalar case, that is $u=u(t,x) \in \R$, extending our method to systems (i.e. when $u=u(t,x) \in \R^n$) should be rather straightforward, the main limitation essentially being the computational cost. Given a dimension $d \in \{1,2,3\}$ and domain sizes $L_1,\dots,L_d$, we assume that the domain of the equation \eqref{eq:general_PDE} is given by the \jp{hyper-}rectangular domain 
\[
\Omega \bydef \prod_{j=1}^d \left[ -\frac{\pi}{L_j},\frac{\pi}{L_j} \right] \subset \R^d
\]
which we supplement with periodic boundary conditions. In other words, the geometry of the domain on which we solve the PDE \eqref{eq:general_PDE} is the $d$-torus $\mathbb{T}^d$. It is worth mentioning that one case of the Navier-Stokes millennium prize  problem considers the model defined on the domain $\mathbb{T}^3$.  The parameters $\lambda_0,\lambda_1,\lambda_2 \in \R$ in \eqref{eq:general_PDE} are chosen so that the equation is parabolic. The assumption of the PDE being semilinear implies that the degree \jp{$p$ of the Laplacian in front of} the nonlinear term $N$ is less than the one of the differential linear operator $\lambda_0  + \lambda_1 \Delta  + \lambda_2 \Delta^2$. While slightly restrictive, \jp{the class of equations \eqref{eq:general_PDE}} includes many known models including the Swift-Hohenberg, Cahn-Hilliard, Ohta-Kawasaki and phase-field-crystal (PFC) equations. Moreover, to simplify the presentation and the computations, we impose the even symmetry $u(t, -x) = u(t, x)$ on the solutions. Note that our approach can readily be adapted to the case of odd boundary conditions $u(t, -x) = -u(t, x)$ or to the general case of periodic boundary conditions without symmetries. 

The general method that we introduce in this paper falls in the category of computer-assisted proofs (CAPs) in nonlinear analysis and dynamical systems.

As exemplified by the early pioneering works on the Feigenbaum conjectures \cite{MR648529} and on the existence of chaos and global attractor in the Lorenz equations \cite{MR1276767,MR1870856,MR1701385}, and by the more recent works on Jones' and Wright’s conjectures in delay equations \cite{MR3912700,MR3779642}, chaos in the 1D Kuramoto-Sivashinsky PDE \cite{MR4113209}, \jp{instability proof in Poiseuille flow \cite{MR2492179}, bifurcating solutions for 3D Rayleigh-Bénard problems \cite{MR2470145}, equilibria in the 3D Navier-Stokes (NS) equations \cite{MR4372115}, solutions of NS on unbounded strips with obstacle \cite{Wunderlich2022_1000150609}}, 3D gyroids patterns in materials \cite{MR3904424}, periodic orbits in NS \cite{BerBreLesVee21}, blowup in Euler equations on the cylinder \cite{CheHou22} and imploding solutions for 3D compressible fluids \cite{BucCaoGom22}, it is undeniable that the role of CAPs is starting to play a major role in the analysis of differential equations, dynamical systems and PDEs. We refer the interested reader to the book \cite{MR3971222} and the survey papers \cite{MR1420838,notices_jb_jp,MR3990999} for more details. 

Due to the interest in understanding global dynamics in semilinear parabolic PDEs and perhaps motivated by the Navier-Stokes millenium problem, it is not surprising that developing CAPs techniques for initial value problems has been a rather active field in the last twenty years. Examples consist of the topological method based on covering relations and self-consistent bounds \cite{MR4113209,MR2049869,MR2788972,MR3167726,MR3408840,MR3338669}, the $C^1$ rigorous integrator of \cite{MR2728184}, the evolution operator and semigroup approach of \cite{MR3639578,MR3683781,MR4444839,MR4388950}, Nakao’s projection method \cite{MR4134381}, the finite element and self-consistent bounds approach \cite{MR4257868}, the fully spectral Fourier-Chebyshev approach \cite{MR4379799}, the Chebyshev interpolation in time method \cite{JBMaxime} and the finite element discretization based approach of \cite{MR3971222,MR4015319}.

To the best of our knowledge, all the above methods (expect \cite{MR3639578} which considers 2D examples) have focused on PDEs defined on one-dimensional domains. Our contribution here is the presentation of a general constructive method to compute solutions of IVPs for higher-dimensional semilinear parabolic equations via semigroup theory, a method which can then be used to prove rigorously global existence (in time) of solutions. Our approach begins by using the periodic boundary conditions on $\Omega$ and expand solutions of \eqref{eq:general_PDE} with Fourier expansions in space leading to an equivalent infinite-dimensional system of ordinary differential equations of the form 
\[
\dot a(t) = f(a(t)),\quad t>0
\]
whose phase space is given by a Banach algebra of Fourier coefficients endowed with analytic norms. For a given time step $\tau>0$, solving an IVP on a time interval $J\bydef(0,\tau]$ with initial condition $a(0) = \varphi$ reduces to find a zero of the map 
\[
F(a(t))=(\dot a(t) - f(a(t)),a(0)-\varphi), \quad t \in J.
\]
Using a finite dimensional projection of the map $F$, we compute a numerical approximation $\ba(t)$ of the IVP using Chebyshev series expansions in time, that is $F(\ba(t)) \approx 0$. We then develop estimates to obtain CAPs for the existence of the solutions of the linearized equations about the approximation $\ba(t)$ from which a solution map operator is obtained. While semigroup theory (cf.\ \cite{pazy1983semigroups}, e.g.) guarantees the existence of the solution operator, our computational and rigorous construction provides an explicit control over this operator, a feature which is new in the context of higher-dimensional PDEs. Using the solution operator, the linear operator $DF(\ba(t))$ is inverted ``by hand" and then used to define the Newton-like operator $T(a(t)) = DF(\ba(t))^{-1}(DF(\ba(t))a(t)-F(a(t)))$. Constructing explicit and computable estimates, and using the time step $\tau$ to control the contraction rate, we derive a sufficient condition in the form of a polynomial inequality to demonstrate rigorously that $T$ is a self-map and a contraction on a ball centered at $\ba(t)$ in an appropriate Banach space of time-dependent Fourier coefficients. The unique fixed point together with its rigorous bounds provide the local inclusion of the solution. Applying this technique for multiple time steps leads to constructive proofs of existence of solutions over long time intervals and, when combined with explicit constructions of trapping regions, can be used to prove global existence. While our computer-assisted approach is an extension of the works \cite{MR3639578,MR4444839,MR4388950,MR4356641} to the case of higher-dimensional semilinear evolution equations, we believe it provides nontrivial contributions. 

First, we believe this is the first time that CAPs techniques are used to prove global existence of solutions of three-dimensional PDEs with initial conditions far from equilibrium. 
Second, 
the evolution operator approach presented here is generalized to PDEs with derivatives in the nonlinear terms, which is made possible thanks to the introduction of geometric weights in the Banach algebra. Third, the smoothing property of the evolution operator controls wrapping effects, allowing the method to be applied for multiple time steps to get existence of solutions over long time intervals. Last by not least, the semigroup approach provides a cost-effective way to generalize our approach to multi-steps (or domain decompositions). More precisely, solving the linearized problems at each step can be made a-priori and independently (i.e. the process is naturally {\em parallelizable}), which implies that the computational cost is additive in time rather than multiplicative. 

The paper is organized as follows. In Section~\ref{sec:Newton}, we introduce the fixed point operator $T$ whose fixed point provides a solution to a given initial value problem. The operator $T$ requires inverting the linearization about the numerical solution, which is done in Section~\ref{sec:solving_IVP}, where we construct the evolution operator via a solution map of the linearized problem. In Section~\ref{sec:time_stepping}, we generalize our method to a multi-step approach by considering a coupled system of the zero-finding problem over multiple time steps. In Section~\ref{sec:global_existence}, we propose a strategy to demonstrate global existence of solutions to IVPs via the mechanism of convergence towards an asymptotically stable equilibrium solution. This approach is exemplified by the application of the Swift-Hohenberg equation, illustrating the efficacy of computer-assisted proofs of global existence for higher dimensional PDEs. Section~\ref{sec:OK} is devoted to presenting the result using our rigorous integrator for IVPs of the 2D Ohta-Kawasaki equation, in particular addressing the handling of derivatives in the nonlinear term. 


\section{Newton-like operator to solve initial value problems} \label{sec:Newton}
In this section, we consider an initial value problem (IVP) associated to \eqref{eq:general_PDE}, that is 
\begin{equation} \label{eq:IVP_PDE}
	\begin{cases}
		u_t = (\lambda_0  + \lambda_1 \Delta  + \lambda_2 \Delta^2)u + \jp{\Delta^p N(u)}, & t>0,\quad x\in \Omega,\\
		u(0,x) = u_0(x),& x \in \Omega,
	\end{cases}
\end{equation}
where $u_0(x)$ is a given initial data. Supplementing the PDE with periodic boundary conditions with the even symmetry $u(t, -x) = u(t, x)$ implies that the unknown solution of the IVP is expanded by the Fourier series in space variables
\begin{equation}\label{eq:FourierSeries}
u(t, x)=\sum_{\bk\in \mathbb{Z}^{d}} a_{\bk }(t) e^{\im (k_1 L_1 x_1 + \cdots + k_d L_d x_d)}, \qquad a_{-\bk} = a_{\bk}	
\end{equation}
where $\bk=(k_1,\dots,k_d)$ and  $x=(x_1,\dots,x_d)$. The symmetry assumption on $u$ implies that the function $u$ is represented by the cosine series 
	\[
	u(t,x) = \sum_{\bk\ge 0} \alpha_{\bk} a_{\bk}(t) \cos(k_1 L_1 x_1)\cdots\cos(k_d L_d x_d),
	\]
	with $\alpha_{\bk}=\alpha_{k_1,\dots,k_d} \bydef 2^{\delta_{k_1,0}}\cdots 2^{\delta_{k_d,0}}$, 
where $\delta_{i,j}$ is the Kronecker delta function. From now on we always assume that $a_{\bk} = a_{|\bk|}$ holds for all $\bk\in\Z^d$ from the cosine symmetry and use the notation of the component-wise absolute value, that is $|\bk |\bydef(|k_1|,\dots,|k_d|)$.
	
\begin{rem}
Boldface characters will be used throughout this paper to denote multi-indices, with $\bk =(k_1,\dots,k_d)$ representing (non-negative) integers $k_j$ for $j=1,\dots,d$. Furthermore, component-wise inequalities are used such that $\bk < \bn$ denotes $k_j<n_j$ for all $j=1,\dots,d$. Similarly, the notations $\bk\le \bn$, $\bk>\bn$, and $\bk\ge \bn$ mean component-wise inequalities.	
\end{rem}

Plugging the Fourier expansion \eqref{eq:FourierSeries} in the nonlinearity \jp{$\Delta^p N(u)$} of the general model \eqref{eq:general_PDE} results in 
\begin{equation} \label{eq:assumption_on_N}
\jp{\Delta^p N(u) = 
\sum_{\bk\in \mathbb{Z}^{d}} \jp{\im^q} (\kk \bL)^q \cN_{\bk}(a) e^{\im (k_1 L_1 x_1 + \cdots + k_d L_d x_d)},}
\end{equation}
where $q=2p \in \{0,2\}$ and 
\[
\kk \bL \bydef \left((k_1 L_1)^2+\dots + (k_d L_d)^2\right)^{1/2},
\]
$a = (a_{\bk })_{\bk\ge 0}$, and $\cN_{\bk}(a)$ is a nonlinear term involving discrete convolutions of $a$ such that $\cN_{\bk}(a) = \cN_{|\bk|}(a)$, $\cN_{\bk}(0)=0$ and $D \cN_{\bk}(0)=0$.


Plugging the Fourier series \eqref{eq:FourierSeries} into \eqref{eq:general_PDE} and assuming \eqref{eq:assumption_on_N} holds, we obtain the following general infinite-dimensional system of ordinary differential equations (ODEs)
\begin{equation} \label{eq:the_ODEs}
\dot{a}_{\bk}(t) = \mu_{\bk} a_{\bk}(t) + \jp{\im^q} (\bk \bL)^q \cN_{\bk}(a(t)),\quad t>0
\end{equation}
where 
\[
\mu_{\bk} = \lambda_0  - \lambda_1 (\bk \bL)^2  + \lambda_2 (\bk \bL)^4
\]
is determined by the linear term $\lambda_0  + \lambda_1 \Delta  + \lambda_2 \Delta^2$ of \eqref{eq:general_PDE}. 
	
Our goal is to determine the time-dependent Fourier coefficients $a(t)\bydef (a_{\bk }(t))_{\bk\ge 0}$ that solve \eqref{eq:the_ODEs} with an initial condition $a(0)=(\varphi_\bk)_{\bk\ge 0}$. The function $a(t)$ will then provide, via the Fourier expansion \eqref{eq:FourierSeries}, the solution of the IVP \eqref{eq:IVP_PDE} with $u_0(x) = \sum_{\bk\ge 0} \alpha_{\bk}  \varphi_{\bk} \cos(k_1 L_1 x_1)\cdots\cos(k_d L_d x_d)$.
	
Before considering the infinite-dimensional problem, we explain how to obtain a numerical approximation of $a(t)$ via a finite-dimensional truncation of the ODEs \eqref{eq:the_ODEs}. To this end, we define a finite set of indices of ``size $\bN=(N_1,\dots,N_d)$'' as $\FN \bydef \{\bk\ge 0:\bk<\bN\}$.
We note that $\FN=F_{N_1}\times\dots \times F_{N_d}$, where $F_{N_j}\bydef\{k_j\ge 0:k_j<N_j\} = \{0,\dots,N_j - 1\}$. 
	
Truncating \eqref{eq:the_ODEs} to the size $\bN$, we fix a step size $\tau>0$ and numerically solve the following truncated $N_1N_2 \cdots N_d$-dimensional IVP
\begin{equation} \label{eq:ODEs_finite}
	\begin{cases}
		\dot{a}_{\bk}(t) = \mu_{\bk}a_{\bk}(t) + \jp{\im^q} (\bk \bL)^q\cN_{\bk}(a(t)), &t \in (0,\tau],\\
		a_{\bk}(0)=\varphi_{\bk},
	\end{cases}\quad\quad (\bk\in\FN).
\end{equation}
The numerical solution of \eqref{eq:ODEs_finite} is denoted by $(\ba^{(\bN)}_{\bk }(t))_{\bk\in\FN}$ and has Fourier expansion
\[
\bar u^{(\bN)}(t,x) = \sum_{\bk \in \FN} \alpha_{\bk} \ba^{(\bN)}(t) \cos(k_1 L_1 x_1)\cdots\cos(k_d L_d x_d).
\]
Moreover,	we assume that $\ba^{(\bm{N})}_{\bk }(t)$ is \jp{ a polynomial given by}
\begin{align}\label{eq:appsol}
\bar{a}^{(\bm{N})}_{\bk }(t)=\bar{a}_{0, \bk }+2 \sum_{\ell=1}^{n-1} \bar{a}_{\ell, \bk } T_{\ell}(t),
\end{align}
where $T_\ell$ is the $\ell$-th order Chebyshev polynomial of the first kind defined on $[0,\tau]$ and $n$ denotes the size of Chebyshev coefficients. To get the coefficients $(\ba_{\ell,\bk})_{0\le \ell<n,\atop\bk\in\FN}$ we use a standard Chebyshev interpolation technique (see Appendix \ref{appA:chebyshev} for more details).
	
Using the numerical solution \eqref{eq:appsol}, we define an approximate solution $\ba_{\bk}(t)$ of \eqref{eq:ODEs_finite} as a natural extension of $(\ba_{\bk }^{(\bm{N})}(t))_{\bk\in\FN}$, that is
\begin{equation} \label{eq:a_bar}
\ba_{\bk}(t) \bydef 
\begin{cases}
\bar{a}^{(\bm{N})}_{\bk}(t),& \bk\in\FN 
\\
0, & \bk\not\in\FN.
\end{cases}
\end{equation}

\jp{
\begin{rem}[Choice of basis functions] 
It is worth noting that by focusing our analysis on PDEs defined over hyper-rectangular domains with periodic boundary conditions, the use of a Fourier expansion in space is a natural choice. Extending this to alternative spatial expansions would require a non-trivial generalization, which lies beyond the scope of this work. On the other hand, we chose to solve the finite-dimensional IVP \eqref{eq:ODEs_finite} using a Chebyshev expansion in time due to the favorable approximation properties of Chebyshev polynomials and our positive experience with them. However, other expansions, such as Taylor series, could also be employed to solve \eqref{eq:ODEs_finite}.
\end{rem}
}

\subsection{Function space and Banach algebra}

Having introduced the finite dimensional IVP \eqref{eq:ODEs_finite} together with the representation of its numerical approximation \eqref{eq:appsol} expressed as a finite linear combination of Chebyshev  polynomials, we are now ready to introduce several function spaces to validate the numerical approximation \eqref{eq:a_bar}. 

Let us define a Banach space of multi-index sequences as
\begin{align}\label{eq:ell-one-space}
	\ell_{\omega}^1\bydef\left\{a=(a_{\bk })_{\bk\ge 0}:a_{\bk}\in \R,~\|a\|_{\omega}\bydef\sum_{\bk\ge 0}|a_{\bk }|\omega_{\bk}<+\infty\right\},
\end{align}
	where the weight $\omega_{\bk}$ for $\bk\ge 0$ is defined by
	\begin{align}\label{eq:weight_omega}
		\omega_{\bk}\bydef\alpha_{\bk}\nu_{\jp{{F}}}^{k_1+\dots+k_d},\quad\nu_{\jp{{F}}}\ge 1.
	\end{align}
	The choice of the weight in \eqref{eq:weight_omega} is to ensure that $\ell_{\omega}^1$ is \textit{Banach algebra} under the discrete convolution, that is
	\begin{align}\label{eq:Banach_algebra}
		\|a\ast b\|_{\omega}\le \|a\|_{\omega}\|b\|_{\omega}
	\end{align}
	holds for all $a, b\in \ell_{\omega}^1$, where the discrete convolution of $a$ and $b$ is defined by
	\begin{align}\label{eq:discrete_convolution}
		\left(a\ast b\right)_{\bk}\bydef\sum_{\bk_1+\bk_2= \bk \atop \bk_1,\bk_2\in \Z^d}a_{|\bk_1|}b_{|\bk_2|}.
	\end{align}
	
Denote the time step $J=(0,\tau]$ for the fixed step size $\tau>0$ and define a function space of the time-dependent sequences as
	$C(J;\ell_{\omega}^1)$,
	which is the Banach space with the norm
\begin{align}\label{eq:Xnorm}
	\|a\|\bydef\sup _{t \in J}\|a(t)\|_{\omega}.
\end{align}

	The operator norm for a bounded linear operator $\cB$ from $\ell_{\omega}^1$ to $\ell_{\omega}^1$ is defined by
	\begin{align}
		\|\cB\|_{B(\ell_{\omega}^1)}\bydef \sup_{\phi\in\ell_{\omega}^1\setminus\{0\}}\frac{\|\cB\phi\|_{\omega}}{\|\phi\|_{\omega}}.
	\end{align}

\jp{
We also introduce subspaces of all multi-index sequences denoted by
\[
	\ell_{\omega,-q}^1\bydef\left\{a=(a_{\bk })_{\bk\ge 0}:a_{\bk}\in \R,~\|a\|_{\omega,-q}\bydef\sum_{\bk\ge 0}|a_{\bk}|\frac{\omega_{\bk}}{|\bk|_\infty^q}<+\infty\right\},\quad |\bk|_\infty\bydef \max(1,\max_{j=1,\dots,d}|k_j|).
\]
Note that these subspaces are also Banach spaces, but are not Banach algebras under the discrete convolution. In addition, $\ell_{\omega}^1=\ell_{\omega,0}^1$ ($q=0$).
}

	
\subsection{The zero-finding problem for the initial value problem}\label{sec:fixed-point-form}

To solve rigorously the infinite-dimensional set of ODEs \eqref{eq:the_ODEs} on the sequence space \eqref{eq:ell-one-space}, let us define three operators as follows. First, define $\cL:D(\cL)\subset \ell_{\omega}^1\to \ell_{\omega}^1$ to be a densely defined closed operator 
acting on a sequence $\phi=(\phi_{\bk})_{\bk\ge0}$ as
\[
\left(\cL \phi\right)_{\bk}\bydef\mu_{\bk}\phi_{\bk} = (\lambda_0  - \lambda_1 (\bk \bL)^2  + \lambda_2 (\bk \bL)^4) \phi_{\bk},
\]
where the domain of the operator $\cL$ is defined by $D(\cL)\bydef\left\{\phi=(\phi_{\bk})_{\bk\ge 0}:\cL \phi\in \ell_{\omega}^1\right\}$.
Second, define \jp{$\cQ:\ell_{\omega}^1\to \ell_{\omega,-q}^1$} to be a multiplication operator acting on a sequence $\phi=(\phi_{\bk})_{\bk\ge0}$ as
\begin{align}\label{eq:def_Q}
	(\cQ \phi)_{\bk} \bydef \jp{\im^q} (\bk \bL)^q\phi_{\bk}.
\end{align}
%
%
Third, define $\mathcal{N}:\ell_{\omega}^1\to \ell_{\omega}^1$ to be a Fr\'echet differentiable nonlinear operator acting on a sequence $\phi=(\phi_{\bk})_{\bk\ge0}$ as
	\[
	\left(\cN(\phi)\right)_{\bk}\bydef \cN_{\bk}(\phi),
	\]
	which depends on the nonlinear term of the target PDE, and satisfies $\cN(0)$ and $D\cN(0)=0$. Hence, for any fixed $\psi \in \ell_{\omega}^1$, there exists a non-decreasing function $g:(0,\infty)\to(0,\infty)$ such that
\begin{equation} \label{eq:N_assumption}
\left\|D\cN(\psi)\phi\right\|_{\omega}\le g(\| \psi \|_{\omega})\|\phi\|_{\omega},\quad\forall\phi\in\ell_{\omega}^1.
\end{equation}
	
Using the above operators, we define a zero-finding problem whose root provide a solution of the ODEs \eqref{eq:the_ODEs} with initial condition $\varphi \in \ell_{\omega}^1$.
Let $X\bydef C(J;\jp{\ell_{\omega}^1})$, $Y\bydef C(J;\jp{\ell_{\omega,-q}^1})$ and $\cD\bydef C\left(J; D(\cL)\right)\cap C^1\left(\jp{(0,\tau)};\ell_{\omega}^1\right)\jp{\cap\,C(J;\ell_{\omega}^1)}$. Note that $X$ is the Banach space with the norm defined in \eqref{eq:Xnorm}.
The map $F$ acting on $a\in \jp{\cD}$ is defined by
\begin{equation}\label{ODE_Cauchy}
	F(a)\bydef\left(\dot{a}-\cL a-\cQ\,\cN(a),~a(0)-\varphi\right)\in Y\times \ell_{\omega}^1.
\end{equation}
We find it convenient to denote the first and second components of \eqref{ODE_Cauchy} respectively by $(F(a))_1=\dot{a}-\cL a-\cQ\,\cN(a)$ and $(F(a))_{2}=a(0)-\varphi$. From the above construction, solving the initial value problem \eqref{eq:IVP_PDE} reduces looking for the solution of the zero-finding problem for the map $F$ given in \eqref{ODE_Cauchy}. 

As it is often the case in nonlinear analysis, the strategy to prove existence (constructively) of a solution of $F=0$ is to turn the problem into an equivalent fixed-point problem, which we solve using a Newton-Kantorovich type theorem. The idea is to look for the fixed point of a Newton-like operator of the form $a \mapsto a - DF(\ba)^{-1} F(a)$, where $DF(\ba)$ is the Fréchet derivative of the map $F$ at the numerical approximation $\ba$ given component-wise by \eqref{eq:a_bar}. This construction is done next.  

\subsection{Exact inverse of the linearization and definition of the Newton-like operator}\label{sec:inverse_DF}
	To use a Newton-like operator to validate the root of $F$, we consider first the Fr\'echet derivative of the map $F$ at $\ba$ defined in \eqref{eq:a_bar}.
For any $h \in \jp{\cD}$, the Fr\'echet derivative is given by
\begin{equation} \label{eq:DFa_bar}
DF(\ba) h = \left(\dot{h} - \cL h-\cQ D\cN(\ba)h,~h(0)\right),\quad DF(\ba):\jp{\cD\subset X}\to Y\times \ell_{\omega}^1,
\end{equation}
where $D\mathcal{N}(\ba)$ is the Fr\'echet derivative of $\cN$ at $\ba$.
		
To construct the inverse of $DF(\ba)$, consider $(p,\phi)\in Y\times \ell_{\omega}^1$ and note that the unique solution of $DF(\ba) h = (p,\phi)$, that is of the linearized problem 
\begin{equation}\label{eq:LinearizedProb}
	\dot{h} - \cL h-\cQ D\cN(\ba)h=p,\quad h(0)=\phi
\end{equation}
%
is given by the variation of constants formula
\begin{equation}\label{eq:variation_of_constant}
h(t) = U(t,0)\phi + \int_{0}^tU(t,s)p(s)ds.
\end{equation}

In the previous formula, $\{U(t, s)\}_{0\le s\le t\le \tau}$ denotes the evolution operator (cf.\ \cite{pazy1983semigroups}) defined by a solution map of the first equation with $p=0$ in \eqref{eq:LinearizedProb}. More precisely, we consider an initial value problem which is represented by
	\begin{align}\label{LinearizedProb}
		\dot{b}(t)=\cL b(t)+\cQ D\cN(\ba(t))b(t),\quad b(s)=\phi, \quad (t>s),
	\end{align}
	where $\phi\in\ell_{\omega}^1$ is any initial data.
	Here $s$ is also a parameter since the evolution operator $U(t,s)$ is a two parameter family of bounded operators.
	We define the evolution operator by $U(t,s)\phi\bydef b(t)$ for $0\le s\le t$, where $b$ is the solution of \eqref{LinearizedProb}. Furthermore, it is known that the evolution operator satisfies the following properties: 
	\begin{enumerate}\renewcommand{\labelenumi}{(\roman{enumi})}
		\item $U(s,s)=\mathrm{Id}$ (identity operator) and $U(t,r)U(r,s)=U(t,s)$ for $0\le s\le r\le t\le \tau$
		\item $(t,s)\mapsto U(t,s)$ is strongly continuous for $0\le s\le t\le \tau$.
	\end{enumerate}
	%
We show how to construct rigorously the evolution operator in Section~\ref{sec:solution_map}. Assuming that this is done, the formula \eqref{eq:variation_of_constant} gives us the inverse of $DF(\ba)$ via the formula
	\begin{align}\label{eq:invDF}
		DF(\ba)^{-1}(p,\phi) \bydef U(t,0)\phi + \int_{0}^tU(t,s)p(s)ds.
	\end{align}
Using the explicit expression for $DF(\ba)^{-1}$, we define the Newton-like operator by
	\begin{align}\label{eq:opT}
		T(a) &\bydef DF(\ba)^{-1}\left(DF(\ba)a-F(a)\right),\quad T:X\to \jp{\cD}\subset X.
	\end{align}
	
	\begin{rem}
		The Newton-like operator defined in \eqref{eq:opT} seems to be defined only on $\jp{\cD\subset X}$ by the formula $a \mapsto a-DF(\ba)^{-1}F(a)$, but it can be defined on $X$ using the following fact: 
		\[
		DF(\ba)a - F(a) =\left(\cQ \cN(a)-\cQ D\cN(\ba)a,~\varphi\right)\in Y\times\ell_{\omega}^1,
		\]
		where the linear part including derivatives is canceled out.
		From the property of $DF(\ba)^{-1}:Y\times\ell_{\omega}^1\to \jp{\cD}$, a fixed-point, say $\tilde{a}$, of $T$ satisfies $\tilde{a}\in \jp{\cD}$. 
		Combining this fact with the continuous property of the evolution operator, the resulting zero of the map defined in \eqref{ODE_Cauchy} also satisfies $\tilde{a}\in C([0,\tau];\ell_{\omega}^1)$.

	\end{rem}
	\begin{rem}\label{rem:nu_val}
	\jp{
		The inverse map $DF(\ba)^{-1}$ restores the smoothness of the solution.  To demonstrate the recovery of smoothness from $Y$ to $X$, letting $\phi=0$ and $p=\cQ \psi\in Y$ ($\psi\in X$) in \eqref{eq:invDF}, one can show the inverse map denoted by
		\[
		\left(DF(\ba)^{-1}\left(\cQ \psi,0\right)\right)(t)  = \int_0^tU(t,s)\cQ \psi(s)ds
		\]
		is bounded in $X$.
		By using the uniform control of the evolution operator given in \eqref{eq:Wh_bound}, which is constructed in Section \ref{sec:solution_map}, it follows that
		\begin{align*}
			\left\|DF(\ba)^{-1}\left(\cQ \psi,0\right)\right\|\le \sup_{t\in J}\int_0^t\|U(t,s)\mathcal{Q}\psi(s)\|_\omega ds
			\le \frac{\tau^{1-\gamma}\bm{{W_q^{\cS_{J}}}}}{1-\gamma} \|\psi\|<+\infty.
		\end{align*}
		Moreover, it is obvious that the function $b$ satisfying \eqref{eq:LinearizedProb} belongs to both $C\left(J; D(\cL)\right)$ and $C^1\left((0,\tau);\ell_{\omega}^1\right)$. This fact yields the smoothing recovery of the inverse map $DF(\ba)^{-1}:Y\times\ell_{\omega}^1\to \jp{\cD}$.
	}
	\end{rem}
	
\jp{	
\begin{rem}
	It is also worth mentioning that there is another natural way of obtaining the operator $T$ directly considering the IVP written by
	\begin{align}
		\begin{cases}
		\dot{a} - (\cL+\cQ D\cN(\ba))a = \cQ(\cN(a)-D\cN(\ba)a)\\
		a(0) = \varphi.
		\end{cases}
	\end{align}
	Using the evolution operator, it follows that
	\begin{align}
		a(t)=(T(a))(t)\bydef U(t,0)\varphi + \int_0^tU(t,s)\cQ(\cN(a(s))-D\cN(\ba(s))a(s))ds,
	\end{align}
	and this formulation directly makes sense on $X$.
\end{rem}}

Now that the Newton-like operator $T$ is defined, the remaining task in validating the local existence of a solution of $F=0$ given in \eqref{ODE_Cauchy} is to show that $T:B_{J}(\ba, \varrho) \to B_{J}(\ba, \varrho)$ is a contraction mapping, for some $\varrho>0$, where
\begin{align}\label{eq:theBall}
	B_{J}(\ba, \varrho) \bydef \left\{a \in X:\|a-\bar{a}\| \leq \varrho\right\}.
\end{align}
Note that $B_{J}(\ba, \varrho) \subset X$ is the ball of radius $\varrho$ in $X$ centered at the numerical approximation $\ba(t) = (\ba_{\bk}(t))_{\bk\ge 0}\in X$ defined in \eqref{eq:a_bar}.

\section{Rigorous integration of initial value problems} \label{sec:solving_IVP}
\jp{
In this section, we provide a numerical method to validate the solution of the zero-finding problem defined in \eqref{ODE_Cauchy}.
The main theorems for this method are based on Theorem \ref{thm:solutionmap} in Section \ref{sec:solution_map} and Theorem \ref{thm:local_inclusion} in Section \ref{sec:local_inclusion}.
The techniques presented here are closely related to the work in \cite{MR4444839}.

First, in Section \ref{sec:solution_map}, we obtain a uniform bound for the evolution operator over the simplex $\cS_{J} \bydef \{(t, s) : 0 < s\,\jp{<} \,t\le \tau\}$, which is presented in Theorem \ref{thm:solutionmap} (and Corollary \ref{cor:solutionmap}). 
More precisely, by setting two initial data for \eqref{LinearizedProb} as $b(s) = \phi$ and ${\cQ} \psi\in\ell_{\omega,\jp{-q}}^1$, we obtain two computable constants $\bm{W^{\cS_{J}}},~\bm{W_q^{\cS_{J}}}>0$ such that
\begin{align}\label{eq:Wh_bound_classic}
	\|b\|=\sup _{(t, s) \in \cS_{J}}\|U(t, s) \phi\|_{\omega} &\le \bm{W^{\cS_{J}}}\|\phi\|_{\omega}, \quad \forall \phi \in\ell_{\omega}^1\\
	\label{eq:Wh_bound}
	\|b\|_{{\cX}}\bydef\sup _{(t, s) \in \cS_{J}}{(t-s)^\gamma}\|U(t, s) {{\cQ}}\psi\|_{\omega} &\le \bm{{W_q^{\cS_{J}}}}\|\psi\|_{\omega}, \quad \forall \psi \in \jp{\ell_{\omega}^1},
\end{align}
where $\gamma\in(0,1)$ is a parameter to be determined in Section \ref{sec:evolution_op}. If $q=0$ in the target ODEs \eqref{eq:the_ODEs}, one can take $\gamma=0$,  resulting in $\bm{W^{\cS_{J}}}=\bm{W_{0}^{\cS_{J}}}$. Note that the bound \eqref{eq:Wh_bound} simply shows that by introducing the weight $(t-s)^\gamma$ in time, the loss of regularity in space is compensated by the evolution operator. See also Remark~\ref{rem:nu_val}.

Second, in Section \ref{sec:local_inclusion}, the above computable constants $\bm{W^{\cS_{J}}}$ and $\bm{{W_q^{\cS_{J}}}}$ are used to bound the action of the linear operator $DF(\ba)^{-1}$ (see Lemma~\ref{lem:DFinv_bounds}). This control is then used to verify that the Newton-like operator \eqref{eq:opT} has a unique fixed point in $B_{J}(\ba, \varrho)$, as defined in \eqref{eq:theBall}. This yields a rigorous validation of the local existence of the solution to the IVP \eqref{eq:IVP_PDE}.
This result is formalized in our second main theorem, Theorem~\ref{thm:local_inclusion}, which provides a hypothesis guaranteeing the local existence of the solution to the zero-finding problem in \eqref{ODE_Cauchy}.
This hypothesis can be numerically verified using interval arithmetic. Therefore, in the present method, we numerically confirm that the hypothesis holds, thereby validating the solution of the zero-finding problem.

}

\subsection{\jp{Explicit uniform bounds of the} evolution operator}\label{sec:solution_map}

\jp{
In this section, we introduce explicit uniform bounds for the evolution operator. We present an explicit computable formula (see Equation~\eqref{eq:W_tau_q_bound}) for the bound $\bm{{W_q^{\cS_{J}}}}$ satisfying \eqref{eq:Wh_bound} (see Theorem~\ref{thm:solutionmap}). Then, in Corollary~\ref{cor:solutionmap}, we also present an explicit formula (see Equation~\eqref{eq:W_tau_bound}) the bound $\bm{W^{\cS_{J}}}$ satisfying \eqref{eq:Wh_bound_classic}.

\begin{rem}
While the theoretical existence of the evolution operator $U(t,s)$ is well-established  (cf.\ \cite[\jp{Section 5.6, Theorem 6.1}]{pazy1983semigroups}), obtaining a rigorous enclosure of its action for a given $\tau$ is difficult in practice. In practical computations, for a specified  $\ba$ and $\tau$, an explicit and rigorous bound on $U(t,s)$ is achieved first by employing computer-assisted proofs to solve a finite dimensional projection of the linearized problem \eqref{LinearizedProb} and then by using theoretical estimates on the tail. This methodology yields the uniform bounds \eqref{eq:Wh_bound_classic} and \eqref{eq:Wh_bound}, used to solve the Cauchy problem. 
\end{rem}
}

Let $\bbm = (m_1,\dots,m_d)\ge 0$ be a size of the finite set $\Fm$, where we suppose that each component $m_i$ is small such that $\bbm<\bN$.
We define a finite-dimensional (Fourier) projection $\varPi^{(\bm{m})}: \ell_{\omega}^1 \rightarrow \ell_{\omega}^1$ acting on $\phi\in\ell_{\omega}^1$ as
\begin{align}\label{eq:Fourier_proj}
	\left(\varPi^{(\bm{m})} \phi\right)_{\bk}\bydef \begin{cases}
		\phi_{\bk }& (\bk\in\Fm) \\
		0 & (\bk\not\in\Fm).
	\end{cases}
\end{align}
We denote $\phi^{(\bm{m})}\bydef\varPi^{(\bm{m})} \phi\in\ell_{\omega}^1$ and $\phi^{(\infty)}\bydef(\mathrm{Id}-\varPi^{(\bm{m})}) \phi\in\ell_{\omega}^1$. Thus, any $\phi\in\ell_{\omega}^1$ is uniquely decomposed by $\phi=\phi^{(\bm{m})}+\phi^{(\infty)}$.

We decompose $b\in X$ in \eqref{LinearizedProb} into a finite part $b^{(\bbm)}$ and an infinite part $b^{(\infty)}$ using the Fourier projection defined in \eqref{eq:Fourier_proj}, we construct the solution $b$. 
{In the followings, two initial conditions are considered, that is $b(s)=\phi$ and $b(s)=\cQ\psi$, but only the case of $b(s)=\cQ\psi$ is presented, because the other case is derived as a corollary to the presented case.}%
For this purpose we consider yet another Cauchy problem with respect to the sequence $c(t)=\left(c_{\bk }(t)\right)_{\bk\ge 0}$.

\begin{align}
&\dot{c}^{(\bm{m})}=\cL c^{(\bm{m})}+\varPi^{(\bm{m})} \cQ D\mathcal{N}(\ba)c^{(\bm{m})}\label{eq:cm}\\
&\dot{c}^{(\infty)}=\cL c^{(\infty)}+(\mathrm{Id}-\varPi^{(\bm{m})})\cQ D\mathcal{N}(\ba)c^{(\infty)}\label{eq:cinf}
\end{align}
with the initial data $c^{(\bbm)}(s)={{\cQ}}\psi^{(\bbm)}$ and $c^{(\infty)}(s)={{\cQ}}\psi^{(\infty)}$.
We remark that one can independently separate these system into the finite part and the infinite part.
So we define the solution operator of \eqref{eq:cm} by extending the finite dimensional fundamental solution matrix $C^{(\bbm)}(t,s)\bydef\left(C^{(\bbm)}_{\bk,\bj}(t,s)\right)_{\bk,\bj\in\Fm}$ satisfying $\sum_{\bj\in\Fm} C^{(\bbm)}_{\bk,\bj}(t,s){\left({\cQ}\psi\right)}_{\bj}= c_{\bk}(t)$.
Similarly, the solution operator of \eqref{eq:cinf} is defined by using the solution map $C^{(\infty)}(t,s){{\cQ}\psi^{(\infty)}}\bydef (c_{\bk}(t))_{\bk\notin\Fm}$ for $(t,s)\in\jp{\cS_{J}}$.
We extend the action of the operator $C^{(\bbm)}(t,s)$ (resp.\ $C^{(\infty)}(t,s)$) on $\ell_{\omega}^1$ by introducing the operator $\bU^{(\bbm)}(t,s)$ (resp.\ $\bU^{(\infty)}(t,s)$) given by
\begin{align}
	\left(\bU^{(\bbm)}(t,s)\phi\right)_{\bk}&\bydef\begin{cases}
		\left(C^{(\bbm)}(t,s)\phi^{(\bbm)}\right)_{\bk} & (\bk\in\Fm)\\
		0 & (\bk\notin \Fm)
	\end{cases}\label{eqn:bUm}\\[1mm]
	\left(\bU^{(\infty)}(t,s)\phi\right)_{\bk}&\bydef\begin{cases}
		0 & (\bk\in\Fm)\\
		\left(C^{(\infty)}(t,s)\phi^{(\infty)}\right)_{\bk} & (\bk\notin \Fm).
	\end{cases}\label{eqn:bUinf}
\end{align}
At this point, the existence of the operator 
$C^{(\infty)}(t,s)$ is not guaranteed, but will be verified via a method of rigorous numerics for the fundamental solution matrix introduced in Sec \ref{sec:fundamental_sol} and via an analytic approach for the evolution operator in Sec \ref{sec:evolution_op}, respectively.
Using the solution operators \eqref{eqn:bUm} and \eqref{eqn:bUinf}, we present the solutions of \eqref{eq:cm} and \eqref{eq:cinf} as
\[
\begin{cases}
c^{(\bm{m})}(t) = \bU^{(\bm{m})}(t,s){{\cQ}}\psi^{(\bm{m})}\\[1mm]
c^{(\infty)}(t)=\bU^{(\infty)}(t,s){{\cQ}}\psi^{(\infty)}.
\end{cases}
\]

Returning to consider the solution $b(t)$, the solutions of \eqref{LinearizedProb} is represented by using the Fourier projection
\begin{align*}
\begin{cases}
\dot{b}^{(\bm{m})}=\cL b^{(\bm{m})}+\varPi^{(\bm{m})} \cQ D\mathcal{N}(\ba)b\\[1mm]
\dot{b}^{(\infty)}=\cL b^{(\infty)}+(\mathrm{Id}-\varPi^{(\bm{m})} )\cQ D\mathcal{N}(\ba)b.
\end{cases}
\end{align*}
Here, from the variation of constants, we have 
\begin{align}
&\begin{cases}
\dot{b}^{(\bm{m})}=\cL b^{(\bm{m})}+\varPi^{(\bm{m})} \cQ D\mathcal{N}(\ba)b\\[1mm]
\dot{b}^{(\infty)}=\cL b^{(\infty)}+(\mathrm{Id}-\varPi^{(\bm{m})} )\cQ D\mathcal{N}(\ba)b
\end{cases}\\[4mm]
&\iff
\begin{cases}
\dot{b}^{(\bm{m})}-\cL b^{(\bm{m})}-\varPi^{(\bm{m})} \cQ D\mathcal{N}(\ba)b^{(\bm{m})}=\varPi^{(\bm{m})} \cQ D\mathcal{N}(\ba)b^{(\infty)}\\[1mm]
\dot{b}^{(\infty)}-\cL b^{(\infty)}-(\mathrm{Id}-\varPi^{(\bm{m})} )\cQ D\mathcal{N}(\ba)b^{(\infty)}=(\mathrm{Id}-\varPi^{(\bm{m})} )\cQ D\mathcal{N}(\ba)b^{(\bm{m})}
\end{cases}\\[4mm]
&\iff
\begin{cases}
\displaystyle b^{(\bm{m})}(t)=\bU^{(\bm{m})}(t,s){{\cQ}}\psi^{(\bm{m})}+\int_s^t\bU^{(\bm{m})}(t,r)\varPi^{(\bm{m})} \cQ D\mathcal{N}(\ba(r))b^{(\infty)}(r)dr\\[4mm]
\displaystyle b^{(\infty)}(t)=\bU^{(\infty)}(t,s){{\cQ}}\psi^{(\infty)}+\int_s^t\bU^{(\infty)}(t,r)(\mathrm{Id}-\varPi^{(\bm{m})} )\cQ D\mathcal{N}(\ba(r))b^{(\bm{m})}(r)dr.
\end{cases}\label{eq:b_form}
\end{align}
By using this formula \eqref{eq:b_form}, the following theorem \jp{provides a sufficient condition (see Equation~\eqref{eq:condition_kappa}) under which we have an explicit formula for the uniform bound of $\bm{\jp{W_q^{\cS_{J}}}}$ satisfying \eqref{eq:Wh_bound}.}

\begin{thm} \label{thm:solutionmap}
	Let $(t, s) \in \jp{\cS_{J}}$ and $\ba$ be fixed.
	Let $\bbm\ge 0$ be the size of the Fourier projection such that $\mu_{\bk}<0$ holds for $\bk\not\in\Fm$.
	For such $\bbm\ge 0$, define $\mu_\ast$ as
	\begin{align}\label{eq:ass_m}
		\mu_\ast \bydef\max_{\bk\not\in\Fm}\left\{\mu_{\bk}\right\}
	\end{align}
	and assume that there exists a constant $\jp{W_{m,q}^{\cS_{J}}}>0$ such that
	\begin{align}\label{eq:Wm_bound}
	\sup_{(t, s) \in \jp{\cS_{J}}}\left\| \bU^{(\bm{m})}(t,s){{\cQ}\psi^{(\bbm)}}\right\|_{\omega}\le \jp{W_{m,q}^{\cS_{J}}}\|\psi^{(\bbm)}\|_{\omega},\quad \forall\psi\in \jp{\ell_{\omega}^1}.
	\end{align}
	Assume also that $\bU^{(\infty)}(t,s)$ satisfies
	\begin{align}\label{eq:Winfts_bound}
	\left\| \bU^{(\infty)}(t,s){{\cQ}\psi^{(\infty)}}\right\|_{\omega}\le W^{(\infty)}_{{q}}(t,s)\|\psi^{(\infty)}\|_{\omega},\quad \forall\psi\in \jp{\ell_{\omega}^1}.
	\end{align}
	Moreover, let us assume the existence of constants \jp{${W_{\infty,q}^{\cS_{J}}}>0$}, \jp{$\overline{W}_{\infty,q}^{\cS_{J}}\ge0$, $\doverline{W}_{\infty,q}^{\cS_{J}}\ge0$}  satisfying
	\begin{align}
			\label{eq:ineqW_infty_sup}
	\sup_{(t,s)\in \jp{\cS_{J}}} {(t-s)^\gamma} W^{(\infty)}_{{q}}(t,s)  & \le \jp{W_{\infty,q}^{\cS_{J}}}\\
			\label{eq:ineqW_infty}
	\sup _{(t,s)\in \jp{\cS_{J}}} {(t-s)^\gamma} \int_{s}^{t} W^{(\infty)}_{{q}}(r, s) d r,\quad\sup _{(t,s)\in \jp{\cS_{J}}} {(t-s)^\gamma} \int_{s}^{t} W^{(\infty)}_{{q}}(t, r) d r &\le \jp{\overline{W}_{\infty,q}^{\cS_{J}}}\\
	\sup _{(t,s)\in \jp{\cS_{J}}} {(t-s)^\gamma} \int_{s}^{t} \int_{s}^{r} W^{(\infty)}_{{q}}(r, \sigma){(\sigma-s)^{-\gamma}} d \sigma d r,\qquad\\\label{eq:ineqbarW_infty}
	 \sup _{(t,s)\in \jp{\cS_{J}}} {(t-s)^\gamma} \int_{s}^{t} W^{(\infty)}_{{q}}(t, r){\int_{s}^r (\sigma-s)^{-\gamma} d\sigma}d r  &\le  \jp{\doverline{W}_{\infty,q}^{\cS_{J}}},
	\end{align}
respectively.
Assume also the existence of two constants \jp{$\cE_{m,\infty}^{J}$} and $\jp{\cE_{\infty,m}^{J}}$ such that
\jp{\begin{align}
\sup_{t \in J}\left\|\varPi^{(\bm{m})} D\mathcal{N}(\ba(t))(\mathrm{Id}-\varPi^{(\bm{m})} )\right\|_{B(\ell_{\omega}^1)} &\le \cE_{m,\infty}^{J}\label{eq:Cm_bound}\\[1mm]
\sup_{t \in J}\left\|(\mathrm{Id}-\varPi^{(\bm{m})} ) D\mathcal{N}(\ba(t))\varPi^{(\bm{m})} \right\|_{B(\ell_{\omega}^1)} &\le \cE_{\infty,m}^{J}\label{eq:Cinf_bound}
\end{align}}
hold for each $t\in J$.
Then, if 
\begin{equation} \label{eq:condition_kappa}
\kappa\bydef 1-\jp{W_{m,q}^{\cS_{J}}}\jp{\doverline{W}_{\infty,q}^{\cS_{J}}} \jp{\cE_{m,\infty}^{J}} \jp{\cE_{\infty,m}^{J}} >0,
\end{equation}
\jp{a uniform bound of the evolution operator $U(t,s)$, that is $\bm{\jp{W_q^{\cS_{J}}}}>0$ in \eqref{eq:Wh_bound}, is obtained by}
\begin{align}\label{eq:W_tau_q_bound}
\bm{\jp{W_q^{\cS_{J}}}} \bydef\left\|\begin{pmatrix}
{\tau^{\gamma}}\jp{W_{m,q}^{\cS_{J}}} & \jp{W_{m,q}^{\cS_{J}}}\jp{\overline{W}_{\infty,q}^{\cS_{J}}} \jp{\cE_{m,\infty}^{J}} \\
\jp{W_{m,q}^{\cS_{J}}}\jp{\overline{W}_{\infty,q}^{\cS_{J}}} \jp{\cE_{\infty,m}^{J}}  & \jp{W_{\infty,q}^{\cS_{J}}}
\end{pmatrix}\right\|_1\kappa^{-1},
\end{align}	
where $\|\cdot\|_1$ denotes the matrix 1-norm.
\end{thm}

\jp{
Note that the bounds $\jp{\overline{W}_{\infty,q}^{\cS_{J}}}$ and $\jp{\doverline{W}_{\infty,q}^{\cS_{J}}}$ become smaller if an appropriate size of the Fourier projection $\bbm\ge 0$ is chosen. In particular, the fact that $\jp{\doverline{W}_{\infty,q}^{\cS_{J}}}$ can be made small by increasing $\bbm$ helps insuring that condition \eqref{eq:condition_kappa} of Theorem \ref{thm:solutionmap} is satisfied. This can be intuitively understood from the following simplified case: if we assume $W^{(\infty)}_{{q}}(t,s)\sim e^{\mu_*(t-s)}$ (where $\mu_*$ is defined in \eqref{eq:ass_m}), then $\jp{\overline{W}_{\infty,q}^{\cS_{J}}}\sim 1/|\mu_*|$ and $\jp{\doverline{W}_{\infty,q}^{\cS_{J}}}\sim 1/\mu_*^2$ hold. Therefore, as the absolute value of $\mu_*$ become large, it is more likely that condition \eqref{eq:condition_kappa} will be satisfied.
}

\begin{rem}
	Theorem \ref{thm:solutionmap} asserts that there exists the evolution operator $\{U(t, s)\}_{0\le s\le t\le \tau}$ on $\ell_{\omega}^1$ and it is uniformly bounded in the sense of \eqref{eq:Wh_bound} over the simplex $\jp{\cS_{J}}$.
	The specific bounds, including the $\jp{W_{m,q}^{\cS_{J}}}$, $W^{(\infty)}_{{q}}(t,s)$, $\jp{W_{\infty,q}^{\cS_{J}}}$, $\jp{\overline{W}_{\infty,q}^{\cS_{J}}}$, and $\jp{\doverline{W}_{\infty,q}^{\cS_{J}}}$ bounds, are presented in the following sections of the paper.
	The $\jp{W_{m,q}^{\cS_{J}}}$ bound is discussed in Section \ref{sec:fundamental_sol} via rigorous numerics for the  fundamental solution matrix of \eqref{eq:cm}. The $W^{(\infty)}_{{q}}(t,s)$ bound is presented in Section \ref{sec:evolution_op} considering the solution map of \eqref{eq:cinf}. The remaining bounds are presented in Section \ref{sec:other_bounds}. Additionally, the $\jp{\cE_{m,\infty}^{J}} $ and $\jp{\cE_{\infty,m}^{J}} $ bounds are determined for each case. These bounds are further introduced in Section \ref{sec:results}.
\end{rem}

\begin{proof}
	Taking the $\ell_{\omega}^1$ norm of \eqref{eq:b_form}, it follows from \eqref{eq:Wm_bound}, \eqref{eq:Winfts_bound}, \eqref{eq:Cm_bound}, and \eqref{eq:Cinf_bound} that
\begin{align}
\left\|b^{(\bm{m})}(t)\right\|_{\omega}&\le \jp{W_{m,q}^{\cS_{J}}}\left\|\psi^{(\bm{m})}\right\|_{\omega}+\jp{W_{m,q}^{\cS_{J}}}\jp{\cE_{m,\infty}^{J}} \int_s^t\left\|b^{(\infty)}(r)\right\|_{\omega} dr\label{eq:bm_omegaq}\\
\left\|b^{(\infty)}(t)\right\|_{\omega}&\le W^{(\infty)}_{{q}}(t,s)\left\|\psi^{(\infty)}\right\|_{\omega}+\jp{\cE_{\infty,m}^{J}} \int_s^tW^{(\infty)}_{{q}}(t,r)\left\|b^{(\bm{m})}(r)\right\|_{\omega} dr.\label{eq:binf_omegaq}
\end{align}
Plugging each estimate in the other one, we obtain
\begin{align}
&{(t-s)^\gamma}\left\|b^{(\bm{m})}(t)\right\|_{\omega}\\
&\le {(t-s)^\gamma}\jp{W_{m,q}^{\cS_{J}}}\left\|\psi^{(\bm{m})}\right\|_{\omega}+\jp{W_{m,q}^{\cS_{J}}}\jp{\cE_{m,\infty}^{J}} {(t-s)^\gamma}\int_s^t\left\|b^{(\infty)}(r)\right\|_{\omega} dr\\
&\le {(t-s)^\gamma}\jp{W_{m,q}^{\cS_{J}}}\left\|\psi^{(\bm{m})}\right\|_{\omega}\\
&\hphantom{\le}\quad+\jp{W_{m,q}^{\cS_{J}}}\jp{\cE_{m,\infty}^{J}} {(t-s)^\gamma}\int_s^t\left(W^{(\infty)}_{{q}}(r,s)\left\|\psi^{(\infty)}\right\|_{\omega}+\jp{\cE_{\infty,m}^{J}} \int_s^rW^{(\infty)}_{{q}}(r,\sigma)\left\|b^{(\bm{m})}(\sigma)\right\|_{\omega} d\sigma\right) dr\\
&\le {(t-s)^\gamma}\jp{W_{m,q}^{\cS_{J}}}\left\|\psi^{(\bm{m})}\right\|_{\omega}+\jp{W_{m,q}^{\cS_{J}}}\jp{\overline{W}_{\infty,q}^{\cS_{J}}} \jp{\cE_{m,\infty}^{J}} \left\|\psi^{(\infty)}\right\|_{\omega}+\jp{W_{m,q}^{\cS_{J}}}\jp{\doverline{W}_{\infty,q}^{\cS_{J}}} \jp{\cE_{m,\infty}^{J}} \jp{\cE_{\infty,m}^{J}} \left\|b^{(\bm{m})}\right\|_{{\cX}}
\end{align}
and
\begin{align}
&{(t-s)^\gamma}\left\|b^{(\infty)}(t)\right\|_{\omega}\\
&\le {(t-s)^\gamma}W^{(\infty)}_{{q}}(t,s)\left\|\psi^{(\infty)}\right\|_{\omega}+\jp{\cE_{\infty,m}^{J}} {(t-s)^\gamma}\int_s^tW^{(\infty)}_{{q}}(t,r)\left\|b^{(\bm{m})}(r)\right\|_{\omega} dr\\
&\le {(t-s)^\gamma}W^{(\infty)}_{{q}}(t,s)\left\|\psi^{(\infty)}\right\|_{\omega}\\
&\hphantom{\le}\quad+\jp{\cE_{\infty,m}^{J}} {(t-s)^\gamma}\int_s^tW^{(\infty)}_{{q}}(t,r)\left(\jp{W_{m,q}^{\cS_{J}}}\left\|\psi^{(\bm{m})}\right\|_{\omega}+\jp{W_{m,q}^{\cS_{J}}}\jp{\cE_{m,\infty}^{J}} \int_s^r\left\|b^{(\infty)}(\sigma)\right\|_{\omega} d\sigma\right) dr\\
&\le {(t-s)^\gamma}W^{(\infty)}_{{q}}(t,s)\left\|\psi^{(\infty)}\right\|_{\omega}+\jp{W_{m,q}^{\cS_{J}}}\jp{\overline{W}_{\infty,q}^{\cS_{J}}} \jp{\cE_{\infty,m}^{J}} \left\|\psi^{(\bm{m})}\right\|_{\omega} + \jp{W_{m,q}^{\cS_{J}}}\jp{\doverline{W}_{\infty,q}^{\cS_{J}}} \jp{\cE_{m,\infty}^{J}} \jp{\cE_{\infty,m}^{J}} \left\|b^{(\infty)}\right\|_{{\cX}}.
\end{align}
Since $\kappa=1-\jp{W_{m,q}^{\cS_{J}}}\jp{\doverline{W}_{\infty,q}^{\cS_{J}}} \jp{\cE_{m,\infty}^{J}} \jp{\cE_{\infty,m}^{J}} >0$ holds from the sufficient condition of the theorem, $\cX$ norm bounds of $b^{(\bm{m})}$ and $b^{(\infty)}$ are given by
\begin{align}
	\left\|b^{(\bm{m})}\right\|_{{\cX}} &\le \frac{{\tau^\gamma}\jp{W_{m,q}^{\cS_{J}}}\left\|\psi^{(\bm{m})}\right\|_{\omega}+\jp{W_{m,q}^{\cS_{J}}}\jp{\overline{W}_{\infty,q}^{\cS_{J}}} \jp{\cE_{m,\infty}^{J}} \left\|\psi^{(\infty)}\right\|_{\omega}}{\kappa}\label{eq:bm_bounds}\\
	\left\|b^{(\infty)}\right\|_{{\cX}} &\le \frac{\jp{W_{\infty,q}^{\cS_{J}}}\left\|\psi^{(\infty)}\right\|_{\omega}+\jp{W_{m,q}^{\cS_{J}}}\jp{\overline{W}_{\infty,q}^{\cS_{J}}} \jp{\cE_{\infty,m}^{J}} \left\|\psi^{(\bm{m})}\right\|_{\omega}}{\kappa}.\label{eq:binf_bounds}
\end{align}
Finally, we have 
\begin{align*}
	\left\|b\right\|_{{\cX}} &\le \left\|b^{(\bm{m})}\right\|_{{\cX}} + \left\|b^{(\infty)}\right\|_{{\cX}}\\
	&=\left\|
	\begin{pmatrix}
		\left\|b^{(\bm{m})}\right\|_{{\cX}}\\
		\left\|b^{(\infty)}\right\|_{{\cX}}
	\end{pmatrix}\right\|_1\\
	&=\kappa^{-1}\left\|\begin{pmatrix}
		{\tau^\gamma}\jp{W_{m,q}^{\cS_{J}}} & \jp{W_{m,q}^{\cS_{J}}}\jp{\overline{W}_{\infty,q}^{\cS_{J}}} \jp{\cE_{m,\infty}^{J}} \\
		\jp{W_{m,q}^{\cS_{J}}}\jp{\overline{W}_{\infty,q}^{\cS_{J}}} \jp{\cE_{\infty,m}^{J}}  & \jp{W_{\infty,q}^{\cS_{J}}}
	\end{pmatrix}\begin{pmatrix}
	\left\|\psi^{(\bm{m})}\right\|_{\omega}\\\left\|\psi^{(\infty)}\right\|_{\omega}
\end{pmatrix}
\right\|_1\\
	&\le \bm{\jp{W_q^{\cS_{J}}}}\|\psi\|_{\omega}. \qedhere
\end{align*}
\end{proof}

\begin{cor}\label{cor:solutionmap}
	Under the same assumptions in Theorem \ref{thm:solutionmap}, assume that $\bU^{(\infty)}(t,s)$ satisfies
	\begin{align}\label{eq:Winfts_bound_classic}
		\left\| \bU^{(\infty)}(t,s){\phi^{(\infty)}}\right\|_{\omega}\le W^{(\infty)}(t,s)\|\phi^{(\infty)}\|_{\omega},\quad \forall\phi\in\ell_{\omega}^1.
	\end{align}
	Assume also that there exists $\jp{W_{\infty}^{\cS_{J}}}>0$, $\jp{\overline{W}_{\infty}^{\cS_{J}}}\ge0$, $\jp{{\overline{W}}_{\infty,q}^{\prime\cS_{J}}}\ge0$, $\jp{{\doverline{W}}_{\infty,q}^{\prime\cS_{J}}}\ge0$ satisfying
	\begin{align}
		\label{eq:ineqW_infty_sup_classic}
		\sup_{(t,s)\in \jp{\cS_{J}}} W^{(\infty)}(t,s)  & \le \jp{W_{\infty}^{\cS_{J}}},\\
		\label{eq:ineqW_infty_classic}
		\sup _{(t,s)\in \jp{\cS_{J}}} \int_{s}^{t} W^{(\infty)}(r, s) d r &\le \jp{\overline{W}_{\infty}^{\cS_{J}}},\\
		\label{eq:ineqhatW_infty_classic}
		\sup _{(t,s)\in \jp{\cS_{J}}} \int_{s}^{t} W^{(\infty)}_{{q}}(t, r) d r &\le \jp{{\overline{W}}_{\infty,q}^{\prime\cS_{J}}},\\
		\label{eq:ineqbarW_infty_classic}
		\sup _{(t,s)\in \jp{\cS_{J}}} \int_{s}^{t} \int_{s}^{r} W^{(\infty)}_{{q}}(r, \sigma)d \sigma d r,~
		\sup _{(t,s)\in \jp{\cS_{J}}} \int_{s}^{t} W^{(\infty)}_{{q}}(t, r)(r-s)d r  &\le  \jp{{\doverline{W}}_{\infty,q}^{\prime\cS_{J}}},
	\end{align}
	respectively.
	Then, taking $\tilde{\kappa}\bydef 1-\jp{W_{m,q}^{\cS_{J}}}\jp{{\doverline{W}}_{\infty,q}^{\prime\cS_{J}}} \jp{\cE_{m,\infty}^{J}} \jp{\cE_{\infty,m}^{J}} $,
	the uniform bound $\bm{\jp{W^{\cS_{J}}}}>0$ in \eqref{eq:Wh_bound_classic} is given by
	\begin{align}\label{eq:W_tau_bound}
	\bm{\jp{W^{\cS_{J}}}} \bydef\left\|\begin{pmatrix}
		\jp{W_{m,0}^{\cS_{J}}} & \jp{W_{m,q}^{\cS_{J}}}\jp{\overline{W}_{\infty}^{\cS_{J}}} \jp{\cE_{m,\infty}^{J}} \\
		\jp{W_{m,0}^{\cS_{J}}}\jp{{\overline{W}}_{\infty,q}^{\prime\cS_{J}}} \jp{\cE_{\infty,m}^{J}}  & \jp{W_{\infty}^{\cS_{J}}}
	\end{pmatrix}\right\|_1\tilde{\kappa}^{-1}.
	\end{align}
\end{cor}

The proof of this corollary and each bound are presented in Appendix \ref{sec:Wtau_bound}. 
We only note here that \jp{$\tilde{\kappa}>0$ holds under the same assumptions as in Theorem \ref{thm:solutionmap}. This result follows from the inequality $\tilde{\kappa}\ge\kappa\,(>0)$, which is confirmed by the relation ${{\doverline{W}}_{\infty,q}^{\prime\cS_{J}}} \le \doverline{W}_{\infty,q}^{\cS_{J}}$ in general (see \eqref{eq:Wbar_infty_q} in Lemma~\ref{lem:other_bounds_q} and \eqref{eq:Wbar_infty} in Lemma~\ref{lem:other_bounds_alt} for their explicit definitions).
}

\subsubsection{Computing \boldmath $\jp{W_{m,q}^{\cS_{J}}}$ \boldmath: rigorous numerics for the fundamental solution \jp{matrix}}\label{sec:fundamental_sol}

Here, we start explaining how we get the $\jp{W_{m,q}^{\cS_{J}}}$ bound in \eqref{eq:Wm_bound}. 
\jp{The method presented here is mostly standard for computer-assisted proofs and is closely connected to that dicussed  in \cite{MR4444839}.
	
From \eqref{eqn:bUm}, the uniformly bounded  evolution operator $\{U(t, s)\}_{0\le s\le t\le \tau}$ corresponds to the fundamental solution matrix $ C^{(\bbm)}(t,s)$ of the finite system \eqref{eq:cm}.  
To construct $C^{(\bbm)}(t,s)$, we use the relation $C^{(\bbm)}(t,s) = \Phi(t)\Phi(s)^{-1}$, where $\Phi(t)$ is the $m_1\dots m_d$ dimensional matrix-valued function such that
\begin{align}\label{eq:Phi_system}
	\frac{d}{dt}\Phi(t)=\cL \Phi(t) +\varPi^{(\bm{m})} \cQ D\mathcal{N}(\ba(t))\Phi(t),\quad \Phi(0)=\mathrm{Id}.
\end{align}
Furthermore, letting $\Psi(s)\bydef\Phi(s)^{-1}$, $\Psi(s)$ solves the adjoint problem of the fundamental system \eqref{eq:Phi_system}, that is
\begin{align}\label{eq:Psi_system}
	\frac{d}{dt}\Psi(s)=-\cL \Psi(s) -\varPi^{(\bm{m})} \cQ D\mathcal{N}(\ba(s))\Psi(s),\quad \Psi(0)=\mathrm{Id}.
\end{align}
Using these formulas, the fundamental solution matrix $C^{(\bbm)}(t,s)$ is explicitly constructed via rigorous numerics, and one can obtain the ${W_{m,q}^{\cS_{J}}}$ bound from the definition of $\bU^{(\bbm)}(t,s)$ in \eqref{eqn:bUm} such that
\begin{align}\label{eq:Wmq}
	\sup_{(t,s)\in \jp{\cS_{J}}}\left\|\bU^{(\bbm)}(t,s)\cQ\right\|_{B(\ell_{\omega}^1)}&=\sup_{(t,s)\in \jp{\cS_{J}}}\left(\max_{\bj \in \Fm}\frac1{\omega_{\bj}}\sum_{\bk\in\Fm}|C_{\bk,\bj}(t,s)|(\bj \bL)^q\omega_{\bk}\right)\\
	&\le \max_{\bj \in \Fm}\frac1{\omega_{\bj}}\sum_{\bk\in\Fm}\left(\sup_{(t,s)\in \jp{\cS_{J}}}|C_{\bk,\bj}(t,s)|\right)(\bj \bL)^q\omega_{\bk} \bydef {W_{m,q}^{\cS_{J}}}.\label{eq:W_mq_bound}
\end{align}

Our main task here is to solve the finite differential systems \eqref{eq:Phi_system} for $t\in[0,\tau]$ and \eqref{eq:Psi_system} for $s\in[0,\tau]$. We present only the case of \eqref{eq:Phi_system}, as the other case can be handled by a straight forward extension.
Let us define the fundamental solution matrix $\Phi(t) \in M_{m_1\dots m_d}(\R)$ as
\begin{equation} \label{eq:Psi}
	\Phi(t) \bydef \left(c_{\bk}^{(\bj)}(t)\right)_{\bk,\bj\in\Fm}=\left(\cdots,~c^{(\bj)}(t),~ \cdots\right)
	\quad (\bj\in\Fm).
\end{equation}
From \eqref{eq:Phi_system}, each $c_{\bk}^{(\bj)}$ ($\bk,\bj\in\Fm$) solves the IVP of the linear system
\begin{align}\label{eq:ode_system}
	\dot{c}_{\bk}^{(\bj)}=\mu_{\bk}c_{\bk}^{(\bj)} + \im^q (\bk \bL)^q\left(D\cN(\ba)c^{(\bj)}\right)_{\bk},\quad c^{(\bj)}_\bk(0) = ({\rm e}_{\textbf{j} })_{\bk} \bydef \delta_{k_1,j_1} \cdots \delta_{k_d,j_d}.
\end{align}
In the following, we omit the superscript $(\bj)$.
By rescaling the time interval $[0,\tau]$ to $[-1,1]$, 
and defining as $G_{\bk}(c)\bydef \frac\tau2\left(\mu_{\bk}c_{\bk} + \im^q (\bk \bL)^q\left(D\cN(\ba)c\right)_{\bk}\right)$,
we can rewrite the IVP \eqref{eq:ode_system} as an integral equation given by
\begin{equation} \label{eq:integral_equation_Fisher}
c_{\bk}(t) = b_{\bk} +  \int_{-1}^t G_{\bk}(c(s))~ds, \qquad \bk \in \Fm, \quad t \in [-1,1],
\end{equation}
where $b_{\bk}$ is the given initial data (determined as $b_{\bk}=({\rm e}_{\textbf{j} })_{\bk}$ when $\bj$ is fixed).
For each $\bk \in \Fm$,  we \jp{use the idea presented in \cite{MR4444839,MR3148084} and} solve the finite number of integral equations \eqref{eq:integral_equation_Fisher} using Chebyshev series expansions, that is we expand $c_\bk(t)$ as
}
\begin{equation} \label{eq:a_k_Chebyshev_expansion}
c_\bk(t) = c_{0,\bk} + 2 \sum_{\ell \ge 1} c_{\ell,\bk} T_\ell(t) = c_{0,\bk} + 2 \sum_{\ell \ge 1} c_{\ell,\bk} \cos(\ell \theta) 
= \sum_{\ell \in \Z} c_{\ell,\bk} e^{\im \ell \theta},
\end{equation}
where $c_{-\ell,\bk} = c_{\ell,\bk}$ and $t = \cos(\theta)$. For each $\bk \in \Fm$, we expand $G_\bk(c(t))$ using a Chebyshev series, that is 
\begin{equation} \label{eq:f_k(a)_Chebyshev_expansion}
G_\bk(c(t)) = \psi_{0,\bk}(c)  + 2 \sum_{\ell \ge 1} \psi_{\ell,\bk}(c) \cos(\ell \theta) 
=  \sum_{\ell \in \Z} \psi_{\ell,\bk}(c) e^{\im \ell \theta},
\end{equation}
where
\[
\psi_{\ell,\bk}(c) \jp{\bydef} \lambda_\bk c_{\ell,\bk} + \jp{\im^q} (\bm{kL})^qN_{\ell,\bk}(c),\quad\jp{\lambda_\bk \bydef  \frac{\tau}{2} \mu_\bk}
\]
\jp{
and $N_{\ell,\bk}(c)\bydef\left((\tau/2)D\cN \left(\ba \right)c\right)_{\ell,\bk}$ is the Chebyshev coefficients of the Fr\'echet derivative $D\cN(\ba)$ acting on $c$}. Letting $N_\bk(c) \bydef (N_{\ell,\bk}(c))_{\ell \ge 0}$, $\psi_\bk(c) \bydef (\psi_{\ell,\bk}(c))_{\ell \ge 0}$ and noting that $(\lambda_\bk c_\bk)_\ell = \lambda_\bk c_{\ell,\bk}$, we get that
\begin{equation} \label{eq:phi_expansion_general}
\psi_\bk(c) = \lambda_\bk c_\bk +\jp{\im^q} (\bm{kL})^q N_\bk(c).
\end{equation}
Combining expansions \eqref{eq:a_k_Chebyshev_expansion} and \eqref{eq:f_k(a)_Chebyshev_expansion} leads to %
\[
\sum_{\ell \in \Z} c_{\ell,\bk} e^{\im \ell \theta} = c_\bk(t) = b_\bk+\int_{-1}^t G_\bk(c(s))~ds =  b_\bk + \int_{-1}^t \sum_{\ell \in \Z} \psi_{\ell,\bk}(c) e^{\im \ell \theta}~ds 
\]
and this results in solving $f=0$, where 
$f = \left(f_{\ell,\bk} \right)_{\jp{\ell\ge 0,\,\bk\in\Fm}}$ is given component-wise by
\[
f_{\ell,\bk}(c) = 
\begin{cases}
\displaystyle
c_{0,\bk} + 2 \sum_{j=1}^\infty (-1)^j c_{j,\bk} - b_\bk, & \ell=0, \bk \in \Fm \\
\displaystyle
2\ell c_{\ell,\bk} + ( \psi_{\ell+1,\bk}(c) - \psi_{\ell-1,\bk}(c)), & \ell>0, \bk \in \Fm.
\end{cases}
\]
Hence, for $\ell>0$ and $\bk \in \Fm$, we aim at solving
\[
f_{\ell,\bk}(c)  = 
2\ell c_{\ell,\bk} +  \lambda_k ( c_{\ell+1,\bk} - c_{\ell-1,\bk})
+ \jp{\im^q} (\bm{kL})^q( N_{\ell+1,\bk}(c) - N_{\ell-1,\bk}(c)) = 0.
\]
Finally, the problem that we solve is $f=0$, where $f = \left(f_{\ell,\bk} \right)_{\ell,\bk}$ is given component-wise by
\[
f_{\ell,\bk}(c) \bydef 
\begin{cases}
\displaystyle
c_{0,\bk} + 2 \sum_{j=1}^\infty (-1)^j c_{j,\bk} - b_\bk, & \ell=0, \bk \in \Fm \\
\displaystyle
- \lambda_\bk c_{\ell-1,\bk} + 2\ell c_{\ell,\bk} + \lambda_\bk c_{\ell+1,\bk}
+ \jp{\im^q} (\bm{kL})^q(N_{\ell+1,\bk}(c) - N_{\ell-1,\bk}(c)), & \ell > 0 , \bk \in \Fm.
\end{cases}
\]
Define the operators (acting on Chebyshev sequences, that is for a fixed $\bk$) by
\begin{equation} \label{eq:tridiagonal_T}
\cT \bydef
\begin{pmatrix} 
0&0&0&0&0&\cdots\\
-1&0&1&0&\cdots&\ \\
0&-1&0&1&0&\cdots\\
\ &\ddots&\ddots&\ddots&\ddots&\ddots\\
\ &\dots&0&-1&0&1 \\
\ &\ &\dots&\ddots&\ddots&\ddots
\end{pmatrix} \quad  \mbox{and} \quad
\Lambda  \bydef
\begin{pmatrix} 
0&0&0&0&0&\cdots\\
0&2&0&0&\cdots&\ \\
0&0&4&0&0&\cdots\\
\ &\ddots&\ddots&\ddots&\ddots&\ddots\\
\ &\dots&0&0&2\ell&0 \\
\ &\ &\dots&\ddots&\ddots&\ddots
\end{pmatrix}.
\end{equation}
Using the operators $\cT$ and $\Lambda$, we may write more densely for the cases $\ell>0$ and $\bk \in \Fm$
\[
f_{\bk}(c)  = \Lambda c_{\bk} + \cT( \lambda_\bk  c_{\bk} + \jp{\im^q} (\bm{kL})^qN_\bk(c) ).
\]
Hence,
\begin{equation} \label{eq:f_{ell,k}}
f_{\ell,\bk}(c) =
\begin{cases}
\displaystyle
c_{0,\bk} + 2 \sum_{j=1}^\infty (-1)^j c_{j,\bk} - b_\bk, & \ell=0, \bk \in \Fm \\
\displaystyle
\left( \Lambda c_{\bk} + \cT( \lambda_\bk c_{\bk} + \jp{\im^q} (\bm{kL})^qN_\bk(c) ) \right)_\ell, & \ell > 0 , \bk \in \Fm.
\end{cases}
\end{equation}
Denoting the set of indices $\cI = \{ (\ell,\bk) \in \N^{d+1} : \ell \ge 0 \text{ and } \bk \in \Fm \}$ and denote $f=(f_{\bfj})_{\bfj \in \cI}$. Assume that using Newton's method, we computed $\bar c = (\bar c_{\ell,\bk})_{\ell=0,\dots,n-1 \atop \bk \in \Fm}$ such that $f(\bar c) \approx 0$. Fix $\nu_{\jp{C}} \ge 1$ the {\em Chebyshev decay rate} and define the weights $\omega_{\ell,\bk} =\alpha_{\ell,\bk} \nu_{\jp{{C}}}^\ell$\jp{, where $\alpha_{\ell,\bk}\bydef 2^{\delta_{\ell,0}}\alpha_{\bk}$}. 
Given a sequence 
$c = (c_{\bfj})_{\bfj \in \cI} = (c_{\ell,\bk})_{\bk \in \Fm \atop \ell \ge 0}$, denote the Chebyshev-weighed $\ell^1$ norm by
\[
\| c \|_{\cX^{\bm{m}}_{\jp{C}}} \bydef \sum_{\bfj \in \cI} |c_{\bfj}| \omega_{\bfj}.
\]
The Banach space in which we prove the existence of the solutions of $f=0$ is given by
\[
\cX^{\bm{m}}_{\jp{C}} \bydef
\left\{ c = (c_{\bfj})_{\bfj \in \cI} : 
\| c \|_{\cX^{\bm{m}}_{\jp{C}}}  < \infty
\right\}.
\]
The following result is useful to perform the nonlinear analysis when solving $f=0$ in $\cX^{\bm{m}}_{\jp{C}}$. 
\begin{lem} \label{lem:banach_algebra}
	For all $a,b \in \cX^{\bm{m}}_{\jp{C}}$, $\| a * b\|_{\cX^{\bm{m}}_{\jp{C}}} \le \| a\|_{\cX^{\bm{m}}_{\jp{C}}} \| b\|_{\cX^{\bm{m}}_{\jp{C}}}$.
\end{lem}
 The computer-assisted proof of existence of a solution of $f=0$ relies on showing that a certain Newton-like operator $c \mapsto c-Af(c)$ has a unique fixed point in the closed ball $B_r(\bar c) \subset \cX^{\bm{m}}_{\jp{C}}$, where $r$ is a radius to be determined. Let us now define the operator $A$. Given $n$, a finite number of Chebyshev coefficients used for the computation of $c$, denote by $f^{(n,\bm{m})}:\R^{nm_1\cdots m_d} \to \R^{nm_1\cdots m_d}$ the finite dimensional projection used to compute $\bar c \in \R^{nm_1\cdots m_d}$, that is 
\[
f^{(n,\bm{m})}_{\ell,\bk}(c^{(n,\bm{m})}) \bydef 
\begin{cases}
\displaystyle
c_{0,\bk} + 2 \sum_{j=1}^{n-1} (-1)^j c_{j,\bk} - b_\bk, & \ell=0, \bk \in \Fm \\
\displaystyle
\left( \Lambda c_{\bk} + T( \lambda_\bk c_{\bk} +\jp{\im^q} (\bm{kL})^q N_\bk(c^{(n,\bm{m})}) ) \right)_\ell, & 0<\ell<n, \bk \in \Fm.
\end{cases}
\]

First consider $A^\dagger$ an approximation for the Fr\'echet derivative $Df(\bar c)$: 
\[
(A^\dagger c)_{\ell,\bk} = 
\begin{cases}
\left( Df^{(n,\bm{m})}(\bar c) c^{(n,\bm{m})} \right)_{\ell,\bk}, & 0 \le \ell < n, \bk \in \Fm
\\
2\ell c_{\ell,\bk}, & \ell \ge n, \bk \in \Fm,
\end{cases}
\]
where $Df^{(n,\bm{m})}(\bar c) \in M_{nm_1\cdots m_d}(\R)$ denotes the Jacobian matrix. Consider now a numerical inverse $A^{(n,\bm{m})}$ of $Df^{(n,\bm{m})}(\bar c)$, that is $\| I - A^{(n,\bm{m})} Df^{(n,\bm{m})}(\bar c) \| \ll 1$. We define the action of $A$ on a vector $c \in \cX^{\bm{m}}_{\jp{C}}$ as 
\[
(Ac)_{\ell,\bk} = 
\begin{cases}
\left( A^{(n,\bm{m})} c^{(n,\bm{m})} \right)_{\ell,\bk}, & 0 \le \ell < n, \bk \in \Fm
\\
\frac{1}{2\ell} c_{\ell,\bk}, & \ell \ge n, \bk \in \Fm.
\end{cases}
\]
The following Newton-Kantorovich type theorem (for linear problems posed on Banach spaces) is useful to show 
the existence of zeros of $f$ \jp{(similar formulations can be found for instance in \cite{BerBreLesVee21,MR4444839,MR2338393})}.
\begin{thm} \label{thm:radii_polynomial}
Assume that there are constants $Y_0, Z_0, Z_1 \ge 0$ having that 
\begin{align}
\label{eq:Y0}
\| A f(\bar c) \|_{\cX^{\bm{m}}_{\jp{C}}} &\le Y_0,
\\
\label{eq:Z0}
\| {\rm Id} - A A^\dagger \|_{B(\cX^{\bm{m}}_{\jp{C}})} &\le Z_0,
\\
\label{eq:Z1}
\| A (Df(\bar c) - A^\dagger) \|_{B(\cX^{\bm{m}}_{\jp{C}})} &\le Z_1.
\end{align}
If 
\begin{equation} \label{eq:radii_polynomial}
Z_0 + Z_1 <1,
\end{equation}
then for all 
\[
r \in \left( \frac{Y_0}{1-Z_0-Z_1},\infty \right),
\]
there exists a unique $\tilde c \in B_r(\bar c)$ such that $f(\tilde c) = 0$. 
\end{thm}
\begin{proof}
We omit the details of this standard proof. Denote $\kappa \bydef Z_0 + Z_1 <1$. The idea is show that $T(c)\bydef c - A f(c)$ satisfies $T(B_r(\bar c)) \subset B_r(\bar c)$ and then that $\| T(c_1)-T(c_2) \|_{\cX^{\bm{m}}_{\jp{C}}} \le \kappa \| c_1-c_2 \|_{\cX^{\bm{m}}_{\jp{C}}}$ for all $c_1,c_2 \in B_r(\bar c)$. From the Banach fixed point theorem, there exists a unique $\tilde c \in B_r(\bar c)$ such that $T(\tilde c) = \tilde c$. The condition \eqref{eq:radii_polynomial} implies that $\| {\rm Id} - A A^\dagger \|_{B(\cX^{\bm{m}}_{\jp{C}})} < 1$, and by construction of the operators $A$ and $A^\dagger$, it can be shown that $A$ is an injective operator. By injectivity of $A$, we conclude that there exists a unique $\tilde c \in B_r(\bar c)$ such that $f(\tilde c) = 0$.
\end{proof}

\jp{
Summing up the above, for a given Fourier projection dimension $\bm{m}=(m_1, \hdots, m_d)$, we apply Theorem~\ref{thm:radii_polynomial} to validate the solutions of $m_1 \cdots  m_d$ problems of the form $f=0$, as given in \eqref{eq:f_{ell,k}}, which is equivalent to the IVP of linear system \eqref{eq:ode_system}.
This approach yields a sequence of solutions $c^{(\bj)}:[0,\tau] \to \R^{m_1\cdots m_d}$ ($\bj\in\Fm$) with the following Chebyshev series representation:
\[
	c_{\bk}^{(\bj)}(t) =  c^{(\bj)}_{0,\bk} + 2 \sum_{\ell \ge 1}  c^{(\bj)}_{\ell,\bk} T_\ell(t),\quad \bk\in\Fm.
\]
Thus, we obtain the explicit form of the fundamental solution matrix $\Phi(t)$ that satisfies \eqref{eq:Phi_system}. Using this representation and $\Psi(s)$, which is the solution to the adjoint problem \eqref{eq:Psi_system}, we can derive the $\jp{W_{m,q}^{\cS_{J}}}$ bound defined in \eqref{eq:W_mq_bound} via rigorous evaluation of Chebyshev polynomials with interval arithmetic.
}

The rest of this section is dedicated to the explicit construction of the bounds $Y_0$, $Z_0$ and $Z_1$ of Theorem~\ref{thm:radii_polynomial}. \\

\noindent{\bf The bound \boldmath$Y_0$\unboldmath.} 
The first bound will be given by
\begin{align*}
\| A f(\bar c) \|_{\cX^{\bm{m}}_{\jp{C}}} = \| A^{(n,\bm{m}) } f^{(n,\bm{m}) }(\bar c) \|_{\cX^{\bm{m}}_{\jp{C}}} + \sum_{\ell \geq n} \sum_{\bk \in F_{\bm{m}}} \frac{1}{2 \ell} \left| f_{\ell,\bk}(\bar{c}) \right| \omega_{\ell,\bk}.
\end{align*}
Using that the term $N_\bk$ is a finite polynomial and that the numerical solution $\bar c = (\bar c_{\ell,\bk})_{\ell=0,\dots,n-1 \atop \bk \in \Fm}$, there exists a $J \in \mathbb{N}$ such that $f_{\ell,\bk}(\bar c)=0$ for all $\ell \ge J$. In other words, the term $f(\bar c)=(f_{\bfj}(\bar c))_{\bfj \in \cI}$ has only finitely many nonzero terms. Hence, the computation of $Y_0$ satisfying $\|Af(\bar c)\|_{\cX^{\bm{m}}_{\jp{C}}} \le Y_0$ is finite and can be rigorously computed with interval arithmetic. \\

\noindent{\bf The bound \boldmath$Z_0$\unboldmath.} The computation of the bound $Z_0$ satisfying \eqref{eq:Z0} requires defining the operator
\[
B \bydef  {\rm Id} - A A^\dagger, 
\]
which action is given by
\[
(Bc)_{\ell,\bk} = 
\begin{cases}
\left( ({\rm Id} - A^{(n,m)} Df^{(n,m)}(\bar c) )c^{(n,m)} \right)_{\ell,\bk}, & 0 \le \ell < n, \bk \in \Fm ,
\\
0, & \ell \ge n, \bk \in \Fm.
\end{cases}
\]
Using interval arithmetic, compute $Z_0$ such that
\[
\| B \|_{B(\cX^{\bm{m}}_{\jp{C}})} = \sup_{\bfj \in \cI} \frac{1}{\omega_{\bfj}} \sum_{\bfi \in \cI} | B_{\bfi,\bfj} | \omega_{\bfi}
= \max_{\ell_2=0,\dots,n-1 \atop \bk_2 \in \Fm} \frac{1}{\omega_{\ell_2,\bk_2}} \sum_{\ell_1=0,\dots,n-1 \atop \bk_1 \in \Fm} | B_{(\ell_1,\bk_1),(\ell_2,\bk_2)} | \omega_{\ell_1,\bk_1} \le Z_0.
\]

\noindent{\bf The bound \boldmath$Z_1$\unboldmath.} For any $c \in B_1(0)$, let
\[
z \bydef [Df(\bar c)-A^\dagger]c 
\]
which is given component-wise by 
\[
z_{\ell,\bk}
= z_{\ell,\bk}(\ba,c) \bydef 
\begin{cases}
\displaystyle
2 \sum_{j \ge n} (-1)^j c_{j,k}, & \ell=0, \bk \in \Fm
\\
\displaystyle
\jp{\im^q} \left( \cT  (\bk \bm{L})^q N_k(c^{(\infty,\bm{m})} \right)_\ell , & 0<\ell<n , \bk \in \Fm
\\
\displaystyle
\lambda_k (\cT c_{k})_\ell + \jp{\im^q} \left( \cT  (\bk \bm{L})^q N_k(c^{(\bm{m})} \right)_\ell, & \ell \ge n , \bk \in \Fm.
\end{cases}
\]
To simplify the notation of the bound $Z_1$, lets us first define component-wise uniform bounds $\hat z_{\ell,\bk}$  for $ 0 \leq \ell < n$ and $\bk \in \Fm$ such that $|z_{\ell,\bk}(\ba,c)| \le \hat z_{\ell,\bk}$ for all $c \in \cX^{\bm{m}}_{\jp{C}}$ with $\|c\|_{\cX^{\bm{m}}_{\jp{C}}} \le 1$. The first case $\ell = 0$ can be bounded by
\begin{equation} \label{eq:hatz_{0,k}(c)}
|z_{0,\bk}(c)| = \left| 2 \sum_{j \ge n} (-1)^j c_{j,\bk} \right| \le 2 \sum_{j \ge n} |c_{j,\bk}| \frac{\omega_{j, \bk}}{\omega_{j ,\bk}} \le
\frac{2}{\omega_{n,\bk}} \sum_{j \ge n} |c_{j,\bk}| \omega_{j,\bk} \le \frac{2}{\omega_{n,\bk}} \|c\|_{\cX^{\bm{m}}_{\jp{C}}} \le  \hat z_{0,k} \bydef \frac{2}{\omega_{n,\bk}}.
\end{equation} 
For the case $\ell = 1,\dots, n-1$, the challenging part to bound is the non-linear term $N_\bk(c^{c^{(\infty,\bm{m} )}})$. First we need to recall that the term $N_k$ can be express as a finite polynomial in $\bar{a}$ such that $N_\bk(c) = \sum_{i = 1}^p \beta_{i} (\bar{a}^i* c)$
for some $p \in \mathbb{N}$ and $\beta_i \in \mathbb{R}$ with $i = 1, \dots, p$.
Before bounding this non-linear term, we first look at how to bound the discrete convolution of $\ba^p$ with $c^{(\infty, \bm{m})}$ for some $p >0$. This convolution can be written as 
\[
(\ba^p*c^{(\infty,\bm{m})})_{\ell,\bk} = \sum_{{{\ell_1+\ell_2 = \ell \atop \bk_1 + \bk_2 = \bk} \atop |\ell_1| \leq p(n-1),   |\ell_2| \ge n} \atop \bk_1 \in \Fm, \bk_2 \in \Fm} \ba^p_{\ell_1,\bk_1} c_{\ell_2,\bk_2} = \sum_{ n \leq i \leq \ell + p(n-1) \atop \bm{j} \in \Fm } \ba^p_{\ell-i,\bk - \bm{j}} c_{i,\bm{j}},.
\]
Using this notation, we see that
\begin{align*}
\left| (\ba*c^{(\infty,m)})_{\ell,\bk} \right| &\leq \sum_{ n \leq i \leq \ell + p(n-1) \atop \bm{j} \in \Fm } \left| \ba_{\ell-i,\bk - \bm{j}}  \right|  \left| c_{i,\bm{j}}  \right| = \sum_{ n \leq i \leq \ell + p(n-1) \atop \bm{j} \in \Fm } \left| \ba_{\ell-i,\bk - \bm{j}}  \right|  \left| c_{i,\bm{j}}  \right| \frac{\omega_{\ell,\bk}}{\omega_{\ell,\bk}} \\&\leq \Psi^{(\ell,\bk)}(\ba)\|c\|_{\cX^{\bm{m}}_{\jp{C}}} \leq \Psi^{(\ell,\bk)}(\ba)
\end{align*}
where
\begin{align*}
\Psi^{(\ell,\bk)}(\ba) \bydef \sup_{|i| \in \mathbb{N} \atop \bm{j} \in \Fm } \frac{\left| \ba_{\ell-i,\bk - \bm{j}}  \right|}{\omega_{i,\bm{j}}} = 
\max_{ \bm{j} \in \Fm \atop n \leq i \leq \ell +p(n -1)} \left\{ \frac{\left| \ba_{\ell-i,\bk - \bm{j}}  \right|}{\omega_{i,\bm{j}}}  \right\},
\end{align*}
which can easily be computed using interval arithmetic. For the cases $1\le \ell<n$ and $\bk \in \Fm$, this leads to the bound
\begin{align*} \label{eq:hatz_{ell,k}(c)}
|z_{\ell,\bk}| &= \left| \left( \cT (\bk \bm{L})^q N_\bk(c^{(\infty,\bm{m})} )  \right)_\ell\right| = \left| \left( \cT (\bk \bm{L})^q 
\sum_{i = 1}^p \beta_i (\ba^i*c^{(\infty,\bm{m})} )_\bk  \right)_\ell \right| 
\\& \le  \left( |\cT|  |(\bk \bm{L})^q |
\sum_{i = 1}^p |\beta_i| \Psi^{(\cdot,\bk)} (\ba^i) \right)_\ell \bydef \hat z_{\ell,\bk} 
\end{align*}
for all $c \in B_1(0)$, where $|\cT|$ denotes the operator with component-wise absolute values. Given $c \in B_1(0)$, we use Lemma~\ref{lem:banach_algebra} to conclude that the bound $Z_1$ is given by

\begin{align*}
\| A [Df(\bar c)-A^\dagger]c\|_{\cX^{\bm{m}}_{\jp{C}}}  
&= \| Az \|_{\cX^{\bm{m}}_{\jp{C}}} 
\\
& = \sum_{\bfj \in \cI} |(Az)_{\bfj}| \omega_{\bfj} \\
& = \sum_{\ell = 0,\dots,n-1 \atop \bk \in \Fm} \left( |A^{(n,m)}| \hat z^{(n,m)}\right)_{\ell,\bk} \omega_{\ell,\bk} 
\\&\quad + \sum_{\ell \ge n \atop \bk \in \Fm} \frac{1}{2 \ell} \left| \lambda_k (\cT c_{k})_\ell +\jp{\im^q} \left( \cT (\bk \bm{L})^q N_k(c^{(\bm{m})} \right)_\ell \right|\omega_{\ell,\bk} 
\\
& \leq \sum_{\ell = 0,\dots,n-1 \atop \bk \in \Fm} \left( |A^{(n,m)}| \hat z^{(n,m)}\right)_{\ell,\bk} \omega_{\ell,\bk}  + \frac{|\lambda_m|}{2n} \sum_{\ell \ge n \atop \bk \in \Fm} |  -c_{\ell-1,k} + c_{\ell+1,k}|\omega_{\ell,\bk}  
\\
& \quad +\frac{1}{2n}\sum_{\ell \ge n \atop \bk \in \Fm} \ \left| \left( \cT (\bk \bm{L})^q \sum_{i = 1}^p \beta_i (\ba^i*c^{(\bm{m})})_{\bk}\right)_\ell \right|\omega_{\ell,\bk} 
\\
&
\leq \sum_{\ell = 0,\dots,n-1 \atop \bk \in \Fm} \left( |A^{(n,m)}| \hat z^{(n,m)}\right)_{\ell,\bk} \omega_{\ell,\bk} \\
& \quad + \frac{1}{2n} \left( \nu_{\jp{C}} + \frac{1}{\nu_{\jp{C}}} \right) \left( |\lambda_m| \|c\|_{\cX^{\bm{m}}_{\jp{C}}}
+ \max_{\bk \in \Fm } \left\{ (\bk \bm{L})^q \right\} \sum_{i = 1}^{p} | \beta_i |  \| \ba^i*c\|_{\cX^{\bm{m}}_{\jp{C}}} \right) 
\\
& \leq \sum_{\ell = 0,\dots,n-1 \atop \bk \in \Fm} \left( |A^{(n,m)}| \hat z^{(n,m)}\right)_{\ell,\bk} \omega_{\ell,\bk} \\
& \quad + \frac{1}{2n} \left( \nu_{\jp{C}} + \frac{1}{\nu}_{\jp{C}} \right) \left( |\lambda_m| 
+ \max_{\bk \in \Fm} \left\{ (\bk \bm{L})^q \right\} \sum_{i = 1}^{p} 2^{d+1}| \beta_i |  \| \ba^i\|_{\cX^{\bm{m}}_{\jp{C}}} \right) 
\end{align*} 
Hence, by construction
\begin{equation} \label{eq:Z1_explicit}
Z_1 \bydef \sum_{ 0 \leq \ell \leq n-1 \atop \bk \in \Fm} \left( |A^{(n,m)}| \hat z^{(n,m)}\right)_{\ell,\bk} \omega_{\ell,\bk} 
+ \frac{\nu_{\jp{C}} + \nu_{\jp{C}}^{-1}}{2n}  \left( |\lambda_m| 
+ \max_{\bk \in \Fm } \left\{ (\bk \bm{L})^q \right\} \sum_{i = 1}^{p} 2^{d+1}| \beta_i |  \| \ba^i\|_{\cX^{\bm{m}}_{\jp{C}}} \right)
\end{equation}
satisfies \eqref{eq:Z1}.
\subsubsection{Generation of the evolution operator and \boldmath $W^{(\infty)}_{{q}}(t,s)$\boldmath\ bounds}\label{sec:evolution_op}
In the followings, unless otherwise noted the index $\bbm$ is the one set in Theorem \ref{thm:solutionmap} and $\mu_\ast$ is that in \eqref{eq:ass_m}.
Let $\ell^1_{\infty}\bydef \left({\rm Id}-\varPi^{(\bbm)}\right) \ell^1_{\omega} = \left\{\left(a_{\bk}\right)_{\bk\ge 0}\in\ell^1_{\omega}:a_{\bk}=0~(\bk\in\Fm)\right\}$ endowed with the norm $\|a\|_{\ell^1_{\infty}}\bydef \sum_{\bk\not\in\Fm}|a_{\bk}|\omega_{\bk}$.
Consider a linear operator $\cL_\infty$ whose action is restricted on $\ell^1_{\infty}$, that is
\[
(\cL_\infty a)_{\bk} = \begin{cases}
	0, & \bk\in\Fm\\
	\mu_{\bk}a_{\bk}, & \bk\not\in\Fm
\end{cases}\quad \mbox{for}~ a \in D(\cL_\infty),
\]
where the domain of the operator $\cL_\infty$ is $D(\cL_\infty) = \left\{a\in\ell^1_{\infty}: \cL_\infty a\in\ell^1_{\infty}\right\}$.
More explicitly, $\cL_\infty = \left({\rm Id}-\varPi^{(\bbm)}\right) \cL $ holds. 
The operator $\cL_\infty$ generates the semigroup on $\ell^1_{\infty}$, which is denoted by $\left\{e^{\cL_\infty t}\right\}_{t\ge 0}$. Furthermore, the action of such a semigroup $\left\{e^{\cL_\infty t}\right\}_{t\ge 0}$ can be naturally extended to $\ell^1_{\omega}$ by
\begin{equation} \label{eq:semigroup_on_ell1}
	\left(e^{\cL_\infty t}\phi\right)_{\bk}=\begin{cases}
		0,&\bk\in\Fm\\
		e^{\mu_{\bk}t} \phi_{\bk}, &\bk\not\in\Fm
	\end{cases}\quad \mbox{for}~\phi\in\ell^1_{\omega}.
\end{equation}

To get a norm bound of the evolution operator $\bU^{(\infty)}(t,s)$,  let us prepare two lemmas.
\begin{lem} \label{lem:semigroup_es_tail}
	Let $\left\{e^{\cL_\infty t}\right\}_{t\ge 0}$ be the semigroup on $\ell^1_{\omega}$ defined in \eqref{eq:semigroup_on_ell1}.
	For $\gamma\in (0,1)$, $\xi\in(0,1]$, \jp{and $\mu_* < 0 $ defined by \eqref{eq:ass_m},} the following estimates:
	\begin{align}\label{eq:semigroup_estimate_0}
		\left\|e^{\cL_\infty t} \psi^{(\infty)}\right\|_{\omega}\le e^{-|\mu_\ast |t}\|\psi^{(\infty)}\|_{\omega}
	\end{align}
	and
	\begin{align}\label{eq:semigroup_estimate_q}
		\left\|e^{\cL_\infty t}\cQ \psi^{(\infty)}\right\|_{\omega}\le C_\infty t^{-\gamma} e^{-(1-\xi)|\mu_\ast |t}\|\psi^{(\infty)}\|_{\omega}
	\end{align}
	hold for $\psi\in \jp{\ell_{\omega}^1}$ and $t\in J$, where
	\begin{align}\label{eq:C_inf}
		C_\infty = \left(\frac{\gamma}{e\xi}\right)^\gamma\sup_{\bk\notin\Fm}\frac{(\bk \bL)^q}{|\mu_{\bk}|^\gamma}.
	\end{align}
\end{lem}
\begin{rem}
	The parameter $\gamma\in(0,1)$ is taken to make the constant $C_\infty$ in \eqref{eq:C_inf} bounded. 
	For example, let us set $\mu_{\bk}=(\bk \bL)^2(-\varepsilon^2(\bk \bL)^2+1)-\sigma$ and $q=2$, which is the case of our example in Section \ref{sec:OK}.
	In this case, $\mu_{\bk}$ is a fourth-order term with respect to $\bk$, whereas $(\bk \bL)^q$ is a second-order term.
	Consequently, it is necessary for $\gamma$ to be asymptotically greater than or equal to $0.5$ to ensure the boundedness of $C_\infty$.
\end{rem}
\begin{proof}
Let $\psi\in \jp{\ell_{\omega}^1}$. It follows for $\bk\notin\Fm$ from \eqref{eq:ass_m} that
\begin{align}
e^{\mu_{\bk}t}(\bk \bL)^q|\psi_{\bk}|\alpha_{\bk}\nu_{\jp{F}}^\bk&= |\mu_{\bk}|^\gamma e^{-\xi|\mu_{\bk}|t}e^{-(1-\xi)|\mu_{\bk}|t}\frac{(\bk \bL)^q}{|\mu_{\bk}|^\gamma}|\psi_{\bk}|\alpha_{\bk}\nu_{\jp{F}}^\bk\\
&\le \sup_{|\mu_{\bk}|}\left(|\mu_{\bk}|^\gamma e^{-\xi|\mu_{\bk}|t}\right)\sup_{|\mu_{\bk}|}\left(e^{-(1-\xi)|\mu_{\bk}|t}\right)\frac{(\bk \bL)^q}{|\mu_{\bk}|^\gamma}|\psi_{\bk}|\alpha_{\bk}\nu_{\jp{F}}^\bk\\
&\le \left(\frac{\gamma}{e\xi}\right)^\gamma t^{-\gamma} \cdot e^{-(1-\xi)|\mu_\ast |t}\frac{(\bk \bL)^q}{|\mu_{\bk}|^\gamma}|\psi_{\bk}|\alpha_{\bk}\nu_{\jp{F}}^\bk\\
&\le C_\infty t^{-\gamma} e^{-(1-\xi)|\mu_\ast |t}|\psi_{\bk}|\alpha_{\bk}\nu_{\jp{F}}^\bk,
\end{align}
where we used the inequality
\begin{align}
	\sup_{x>0} \left(x^\gamma e^{-\xi xt}\right) \le \left(\frac{\gamma}{e\xi}\right)^\gamma t^{-\gamma},\quad\mbox{for~}\gamma,~t>0.
\end{align}
Therefore, it yields from the definition \eqref{eq:semigroup_on_ell1} that
\begin{align}
\left\|e^{\cL_\infty t}\cQ\psi^{(\infty)}\right\|_{\omega}&=\sum_{\bk\not\in\Fm}e^{\mu_{\bk}t}(\bk \bL)^q|\psi_{\bk}|\alpha_{\bk}\nu_{\jp{F}}^\bk\\
&\le C_\infty t^{-\gamma} e^{-(1-\xi)|\mu_\ast |t}\sum_{\bk\not\in\Fm}|\psi_{\bk}|\alpha_{\bk}\nu_{\jp{F}}^\bk = C_\infty t^{-\gamma} e^{-(1-\xi)|\mu_\ast |t}\|\psi^{(\infty)}\|_{\omega}.
\end{align}
This directly gives the semigroup estimate \eqref{eq:semigroup_estimate_q}.
In the case of $q=0$, one can directly get the estimate \eqref{eq:semigroup_estimate_0} by setting $\gamma=0$, $\xi=0$, and $C_\infty=1$.
\end{proof}
%

%
From Lemma \ref{lem:semigroup_es_tail}, we show the existence of the solution of \eqref{eq:cinf} with any initial sequence $c^{(\infty)}(s)=\cQ\psi^{(\infty)}$ in the following theorem.
Consequently, the existence of the evolution operator $\bU^{(\infty)}(t,s)$ is obtained. To state the result, we recall $\ba(t)$ defined in \eqref{eq:a_bar} and recall \eqref{eq:N_assumption} that, for each $t\in J$, there exists a non-decreasing function $g:(0,\infty)\to(0,\infty)$ such that $\left\|D\cN(\ba(t))\phi\right\|_{\omega}\le g(\| \ba(t) \|_{\omega})\|\phi\|_{\omega}$ for all $\phi\in\ell_{\omega}^1$. 

\begin{thm}\label{thm:ev_op}
	Let $\left\{e^{\cL_\infty t}\right\}_{t\ge 0}$ be the semigroup on $\ell^1_{\omega}$ defined in \eqref{eq:semigroup_on_ell1} and assume that Lemma \ref{lem:semigroup_es_tail} holds for $\gamma\in (0,1)$ and $\xi\in(0,1]$.
	There exists a unique solution of the integral equation:
	\begin{align}\label{eq:cinf_integral}
	c^{(\infty)}(t)=e^{\cL_\infty (t-s)}{\cQ}\psi^{(\infty)} + \int_s^t e^{\cL_\infty (t-r)}({\rm Id} -\varPi^{(\bbm)})\cQ D\cN(\ba(r)) c^{(\infty)}(r) d r.
	\end{align}
	Hence, such a solution $c^{(\infty)}$ solves the infinite-dimensional system of differential equations \eqref{eq:cinf}.
	It yields that the evolution operator $\bU^{(\infty)}(t,s)$ exists and the following estimate holds
	\[
	\left\|\bU^{(\infty)}(t,s){\cQ}\psi^{(\infty)}\right\|_{\omega}\le W_{{q}}^{(\infty)}(t,s)\left\|\psi ^{(\infty)}\right\|_{\omega},\quad\forall\psi\in \jp{\ell_{\omega}^1},
	\]
	where $W_{{q}}^{(\infty)}(t,s)$ is defined by
	{\begin{align}\label{eq:W^inf_q}
			W_q^{(\infty)}(t,s)\bydef 
			\begin{cases}
				e^{(-|\mu_\ast | + g(\|\ba\|))(t-s)}, & q=0\\
				C_\infty (t-s)^{-\gamma} e^{-(1-\xi)|\mu_\ast |(t-s) + C_\infty (t-s)^{1-\gamma} g(\|\ba\|) \mathrm{B}(1-\gamma,1-\gamma)}, & q>0
			\end{cases}.
	\end{align}
	Here, $\mathrm{B}(x,y)$ denotes the Beta function.}
\end{thm}

\begin{proof}
	For a fixed $s> 0$, let us define a map $\cP$ acting on the $c^{(\infty)}(t)$ as
	\[
	\cP c^{(\infty)}(t) \bydef e^{\cL_\infty (t-s)}{\cQ}\psi^{(\infty)} + \int_s^t e^{\cL_\infty(t-r)}({\rm Id} -\varPi^{(\bbm)})\cQ D\cN(\ba(r)) c^{(\infty)}(r) d r
	\]
	and let a function space $\cX_{\infty}$ be defined by
	\[
	\cX_{\infty} \bydef\left\{c^{(\infty)}\in C\left(\jp{(s,\tau]},\ell_{\infty}^1\right) :
	\sup_{\jp{s<t\le\tau}}
	{C_\infty^{-1}(t-s)^{\gamma}e^{(1-\xi)|\mu_\ast |(t-s)}}\left\| c^{(\infty)}(t)\right\|_{\omega}<\infty
	\right\}.
	\]
	Consider $\beta>1$ and define a distance on $\cX_{\infty}$ as
	\begin{align}
	\mathbf{d}\left(c_1^{(\infty)},c_2^{(\infty)}\right)\bydef\sup_{\jp{s<t\le\tau}}\left(
	{(t-s)^\gamma e^{(1-\xi)|\mu_\ast |(t-s) - \beta\tau^\gamma  C_\infty \int_{s}^t H(r)dr}}
	\left\|
	c_1^{(\infty)}(t)-c_2^{(\infty)}(t)
	\right\|_{\omega}\right),
	\end{align}
	{where $H(r)$ is defined by
	\begin{align}
		H(r)\bydef (t-r)^{-\gamma} g(\|\ba(r)\|_{\omega})(r-s)^{-\gamma}.
	\end{align}}%
	\jp{Then, $(\cX_{\infty},\mathbf{d})$ is a complete metric space.}
	Let us denote $g(r)\equiv g(\|\ba(r)\|_{\omega})$ for the simplicity.
	We prove that the map $\cP $ becomes a contraction mapping under the distance $\mathbf{d}$ on $\cX_{\infty}$.
	For $c_1^{(\infty)},~c_2^{(\infty)}\in \cX_{\infty}$, it follows using \eqref{eq:semigroup_estimate_q} that
	{
	\begin{align*}
		&(t-s)^\gamma e^{(1-\xi)|\mu_\ast |(t-s) - \beta\tau^\gamma   C_\infty \int_{s}^t H(r)dr}
		\left\| \cP  c_1^{(\infty)}(t)-\cP  c_2^{(\infty)}(t)\right\|_{\omega}\\
		&\le (t-s)^\gamma e^{(1-\xi)|\mu_\ast |(t-s) - \beta\tau^\gamma   C_\infty \int_{s}^t H(r)dr}
		\int_{s}^t\left\|e^{\cL_\infty(t-r)}({\rm Id} -\varPi^{(\bbm)})\cQ D\cN(\ba(r)) \left(c_1^{(\infty)}(r)-c_2^{(\infty)}(r)\right)\right\|_{\omega}dr\\
		&\le (t-s)^\gamma e^{(1-\xi)|\mu_\ast |(t-s) - \beta\tau^\gamma   C_\infty \int_{s}^t H(r)dr}\\
		&\hphantom{\le}\quad\cdot
		\int_{s}^t C_\infty (t-r)^{-\gamma} e^{-(1-\xi)|\mu_\ast |(t-r)}\left\|({\rm Id} -\varPi^{(\bbm)}) D\cN(\ba(r)) \left(c_1^{(\infty)}(r)-c_2^{(\infty)}(r)\right)\right\|_{\omega}dr\\
		&\le (t-s)^\gamma e^{(1-\xi)|\mu_\ast |(t-s) - \beta\tau^\gamma   C_\infty \int_{s}^t H(r)dr}
		\int_{s}^t C_\infty (t-r)^{-\gamma} e^{-(1-\xi)|\mu_\ast |(t-r)}g(r) \left\| c_1^{(\infty)}(r)-c_2^{(\infty)}(r)\right\|_{\omega}dr\\
		&\le (t-s)^\gamma e^{(1-\xi)|\mu_\ast |(t-s) - \beta\tau^\gamma   C_\infty \int_{s}^t H(r)dr}
		\mathbf{d}\left(c_1^{(\infty)},c_2^{(\infty)}\right) \int_{s}^t   C_\infty H(r) e^{-(1-\xi)|\mu_\ast |(t-s)+\beta\tau^\gamma  C_\infty\int_{s}^r H(\sigma)d\sigma} 
		dr\\
		&= (t-s)^\gamma e^{- \beta\tau^\gamma   C_\infty \int_{s}^t H(r)dr}
		\mathbf{d}\left(c_1^{(\infty)},c_2^{(\infty)}\right) \int_{s}^t   C_\infty H(r) e^{\beta\tau^\gamma    C_\infty\int_{s}^r H(\sigma)d\sigma} dr\\
		&\le e^{- \beta\tau^\gamma   C_\infty \int_{s}^t H(r)dr}\mathbf{d}\left(c_1^{(\infty)},c_2^{(\infty)}\right) 
		\left[\frac1\beta e^{\beta\tau^\gamma    C_\infty\int_{s}^r H(\sigma)d\sigma}\right]_{r=s}^{r=t}\\
		&\le \frac{1}{\beta} \mathbf{d}\left(c_1^{(\infty)},c_2^{(\infty)}\right).
	\end{align*}}%
	Since $\beta>1$, the map $\cP $ becomes a contraction mapping on $\cX_{\infty}$.
	This yields that the solution of \eqref{eq:cinf} uniquely exists in $\cX_{\infty}$, which satisfies \eqref{eq:cinf_integral}.
	%
	Moreover, letting 
	\[
	y(t)\bydef {C_\infty^{-1}(t-s)^{\gamma}e^{(1-\xi)|\mu_\ast |(t-s)}}\| c^{(\infty)}(t)\|_{\omega},
	\]
	 it follows from \eqref{eq:cinf_integral} using \eqref{eq:semigroup_estimate_q} that
	\begin{align}
		y(t)\le \|\psi^{(\infty)}\|_{\omega} + {G(t) \int_s^t H(r) y(r)d r},
	\end{align}
	{where $G(t)\bydef C_\infty (t-s)^\gamma$.
	From the Gronwall lemma {\cite[Chap.\,7, p.\,356]{Mitrinovic1991}}, it follows that
	\[
	y(t)\le \left(1 + G(t)\int_s^t H(r)\exp\left(\int_{r}^{t}G(\sigma)H(\sigma)d\sigma\right)dr\right)\|\psi^{(\infty)}\|_{\omega}.
	\]
	Since $G(t)$ is a non-decreasing function, one gets that
	\begin{align}
		y(t)&\le \left(1 + G(t)\int_s^t H(r) \exp\left(\int_{r}^{t}G(\sigma)H(\sigma)d\sigma\right)dr\right)\|\psi^{(\infty)}\|_{\omega}\\
		&\le \left(1 + G(t)\int_s^t H(r) \exp\left(G(t)\int_{r}^{t}H(\sigma)d\sigma\right)dr\right)\|\psi^{(\infty)}\|_{\omega}\\
		&= \left(1 + \left[-\exp\left(G(t)\int_{r}^{t}H(\sigma)d\sigma\right) \right]_{r=s}^{r=t}\right)\|\psi^{(\infty)}\|_{\omega}\\
		&= \exp\left(G(t)\int_{s}^{t}H(\sigma)d\sigma\right)\|\psi^{(\infty)}\|_{\omega}.
	\end{align}
	Furthermore, 
	\begin{align} 
		G(t)\int_s^t H(r)dr  &\le C_ \infty (t-s)^\gamma g(\|\ba\|)\int_s^t (t-r)^{-\gamma}(r-s)^{-\gamma} dr\\
		&= C_\infty (t-s)^{1-\gamma} g(\|\ba\|) \mathrm{B}(1-\gamma,1-\gamma)
	\end{align}
	holds. Here, $\mathrm{B}(x,y)$ is the Beta function.}%
	
	Then, we conclude that the following inequality holds for any $\psi\in \jp{\ell_{\omega}^1}$:
	\begin{align}\label{eq:ev_op_estimate2}
		\left\|\bU^{(\infty)}(t,s)\cQ\psi ^{(\infty)}\right\|_{\omega}
		&\le {C_\infty (t-s)^{-\gamma} e^{-(1-\xi)|\mu_\ast |(t-s) + C_\infty (t-s)^{1-\gamma} g(\|\ba\|) \mathrm{B}(1-\gamma,1-\gamma)}}\\
		&= W_{{q}}^{(\infty)}(t,s)\left\|\psi ^{(\infty)}\right\|_{\omega},
	\end{align}
	where $W_{{q}}^{(\infty)}(t,s)$ is defined in \eqref{eq:W^inf_q}. When $q=0$, an analogous discussion holds by setting $\gamma=0$, $\xi=0$, and $C_\infty=1$. It directly follows \eqref{eq:W^inf_q} in the case of $q=0$.
\end{proof}

\subsubsection{The other bounds}\label{sec:other_bounds}
Here, we will give $\jp{W_{\infty,q}^{\cS_{J}}}$, $\jp{\overline{W}_{\infty,q}^{\cS_{J}}}$, and $\jp{\doverline{W}_{\infty,q}^{\cS_{J}}}$ bounds satisfying \eqref{eq:ineqW_infty_sup}, \eqref{eq:ineqW_infty}, and \eqref{eq:ineqbarW_infty}, respectively.
Firstly, let us rewrite $W^{(\infty)}_{0}(t,s)$ in \eqref{eq:W^inf_q} as $e^{\vartheta(t-s)}$, where $\vartheta=-|\mu_\ast |+g(\|\ba\|)$.
\begin{lem}\label{lem:other_bounds}
	\jp{Set $q=0$ and d}efine the constants $\jp{W_{\infty,0}^{\cS_{J}}}>0$, $\jp{\overline{W}_{\infty,0}^{\cS_{J}}}\ge0$, $\jp{\doverline{W}_{\infty,0}^{\cS_{J}}}\ge0$ as
	\begin{align}
		\jp{W_{\infty,0}^{\cS_{J}}}&\bydef \begin{cases}
			1, & \vartheta\le 0\\
			e^{\vartheta \tau}, & \vartheta>0
		\end{cases}\\[1mm]
		\jp{\overline{W}_{\infty,0}^{\cS_{J}}}&\bydef \frac{1}{\vartheta}\left(e^{\vartheta \tau}-1\right)\\[1mm]
		\jp{\doverline{W}_{\infty,0}^{\cS_{J}}}&\bydef\frac{\jp{\overline{W}_{\infty,0}^{\cS_{J}}}- \tau}{\vartheta},
	\end{align}
	respectively. Then $W^{(\infty)}_{{0}}(t,s)$, defined in \eqref{eq:W^inf_q}, obeys the inequalities \eqref{eq:ineqW_infty_sup}, \eqref{eq:ineqW_infty}, and \eqref{eq:ineqbarW_infty} \jp{with $\gamma=0$}.
\end{lem}
\begin{proof}
	First, we note that from \eqref{eq:W^inf_q} in the case of $q=0$
	\begin{align}
		\sup_{(t,s)\in \jp{\cS_{J}}} W^{(\infty)}_{{0}}(t,s) = \sup_{(t,s)\in \jp{\cS_{J}}} e^{\vartheta(t-s)}\le \begin{cases}
			1, & \vartheta\le 0\\
			e^{\vartheta \tau}, & \vartheta>0
		\end{cases}=\jp{W_{\infty,0}^{\cS_{J}}}.
	\end{align}
	The inequality \eqref{eq:ineqW_infty_sup} holds.
	Second, we note that
	\begin{align}
		\sup _{(t,s)\in \jp{\cS_{J}}} \int_{s}^{t} W^{(\infty)}_{{0}}(r, s) d r &= \sup _{(t,s)\in \jp{\cS_{J}}} \int_{s}^{t}  e^{\vartheta(r-s)} d r\\
		&=\sup _{(t,s)\in \jp{\cS_{J}}} \frac1\vartheta\left(e^{\vartheta (t-s)}-1\right)\\
		&\le \frac1\vartheta\left(e^{\vartheta \tau}-1\right)=\jp{\overline{W}_{\infty,0}^{\cS_{J}}}
	\end{align}
	and that
	\begin{align}
		\sup _{(t,s)\in \jp{\cS_{J}}} \int_{s}^{t} W^{(\infty)}_{{0}}(t, r) d r &= \sup _{(t,s)\in \jp{\cS_{J}}} \int_{s}^{t}  e^{\vartheta(t-r)} d r\\
		&=\sup _{(t,s)\in \jp{\cS_{J}}} \frac1\vartheta\left(e^{\vartheta (t-s)}-1\right)\\
		&\le \frac1\vartheta\left(e^{\vartheta \tau}-1\right)=\jp{\overline{W}_{\infty,0}^{\cS_{J}}}.
	\end{align}
	These yield that \eqref{eq:ineqW_infty} holds.
	Third, note that
	\begin{align}
		\sup _{(t,s)\in \jp{\cS_{J}}}\int_{s}^{t} \int_{s}^{r} W^{(\infty)}_{{0}}(r, \sigma) d \sigma d r&=\sup _{(t,s)\in \jp{\cS_{J}}}\int_{s}^{t} \int_{s}^{r} e^{\vartheta(r-\sigma)} d \sigma d r\\
		&= \sup _{(t,s)\in \jp{\cS_{J}}}\int_{s}^{t} \frac1\vartheta\left(e^{\vartheta (r-s)}-1\right)dr\\
		&= \sup _{(t,s)\in \jp{\cS_{J}}}\frac1\vartheta\left[\frac{e^{\vartheta (r-s)}}{\vartheta}-r\right]^{r=t}_{r=s}\\
		&= \sup _{(t,s)\in \jp{\cS_{J}}}\frac1\vartheta\left[\frac{e^{\vartheta (t-s)}-1}{\vartheta}-(t-s)\right]\\
		&\le \frac1\vartheta\left(\frac{e^{\vartheta \tau}-1}{\vartheta}-\tau\right)\\
		&\le \frac{\jp{\overline{W}_{\infty,0}^{\cS_{J}}}- \tau}\vartheta=\jp{\doverline{W}_{\infty,0}^{\cS_{J}}}.
	\end{align}
	Finally, we have
	\begin{align}
		\sup _{(t,s)\in \jp{\cS_{J}}}\int_{s}^{t} W^{(\infty)}_{{0}}(t, r)(r -s) d r&= \sup _{(t,s)\in \jp{\cS_{J}}}\int_{s}^{t}   e^{\vartheta(t-r)}(r -s) d r\\
		&=\sup _{(t,s)\in \jp{\cS_{J}}}\left\{\left[-\frac1\vartheta e^{\vartheta(t-r)}(r-s)\right]^{r=t}_{r=s} + \int_{s}^t\frac1\vartheta e^{\vartheta(t-r)}dr\right\}\\
		&=\sup _{(t,s)\in \jp{\cS_{J}}}\left\{-\frac1\vartheta (t-s) + \frac1\vartheta\left[ -\frac{e^{\vartheta(t-r)}}{\vartheta}\right]^{r=t}_{r=s} \right\}\\
		&=\sup _{(t,s)\in \jp{\cS_{J}}}\frac1\vartheta \left[-(t-s) +  \frac{e^{\vartheta(t-s)}-1}{\vartheta}\right]\\
		&\le \frac1\vartheta\left(\frac{e^{\vartheta \tau}-1}{\vartheta}-\tau\right)\\
		&\le \frac{\jp{\overline{W}_{\infty,0}^{\cS_{J}}}- \tau}\vartheta=\jp{\doverline{W}_{\infty,0}^{\cS_{J}}}.
	\end{align}
	Hence, \eqref{eq:ineqbarW_infty} holds.
\end{proof}

\begin{rem}
	The inequalities \eqref{eq:ineqW_infty} and \eqref{eq:ineqbarW_infty} hold regardless of the positivity or negativity of the variable $\vartheta$. Furthermore, in the above proof, we used the monotonicity of the functions
	\begin{align}
		\frac{e^{\vartheta t}-1}{\vartheta}\quad\mbox{and}\quad \frac1\vartheta\left(\frac{e^{\vartheta t}-1}{\vartheta}-t\right)
	\end{align}
	with respect to $t\in J$.
\end{rem}

{
Next, we consider the case of $q>0$. Let us rewrite $W^{(\infty)}_{q}(t,s)$ in \eqref{eq:W^inf_q} as  $C_\infty (t-s)^{-\gamma}e^{-\iota(t-s) + \vartheta(t-s)^{1-\gamma}}$, where
\begin{align}
	\iota = (1-\xi)|\mu_\ast |,\quad \vartheta=C_\infty g(\|\ba\|) \mathrm{B}(1-\gamma,1-\gamma).
\end{align}
Similar to Lemma \ref{lem:other_bounds}, we have the following lemma:
\begin{lem}\label{lem:other_bounds_q}
	Define the constants $\jp{W_{\infty,q}^{\cS_{J}}}>0$, $\jp{\overline{W}_{\infty,q}^{\cS_{J}}}\ge0$, $\jp{\doverline{W}_{\infty,q}^{\cS_{J}}}\ge0$ as
	\begin{align}
		\jp{W_{\infty,q}^{\cS_{J}}}&\bydef C_\infty e^{\vartheta \tau^{1-\gamma}}\\[1mm]
		\jp{\overline{W}_{\infty,q}^{\cS_{J}}}&\bydef \frac{\tau}{1-\gamma}\jp{W_{\infty,q}^{\cS_{J}}}\\[1mm]
		\jp{\doverline{W}_{\infty,q}^{\cS_{J}}}&\bydef\frac{\jp{\overline{W}_{\infty,q}^{\cS_{J}}}}{2}\tau^{1-\gamma}\mathrm{B}(1-\gamma,1-\gamma),\label{eq:Wbar_infty_q}
	\end{align}
	respectively. Then $W^{(\infty)}_{q}(t,s)$, defined in \eqref{eq:W^inf_q}, obeys the inequalities \eqref{eq:ineqW_infty_sup}, \eqref{eq:ineqW_infty}, and \eqref{eq:ineqbarW_infty}.
\end{lem}
\begin{proof}
	First, it follows from \eqref{eq:W^inf_q} that
	\begin{align}
		\sup_{(t,s)\in \jp{\cS_{J}}} (t-s)^\gamma W^{(\infty)}_{q}(t,s) = \sup_{(t,s)\in \jp{\cS_{J}}} C_\infty e^{-\iota(t-s) + \vartheta(t-s)^{1-\gamma}}\le C_\infty e^{\vartheta \tau^{1-\gamma}} =\jp{W_{\infty,q}^{\cS_{J}}}.
	\end{align}
	The inequality \eqref{eq:ineqW_infty_sup} holds.
	Second, we note that
	\begin{align}
		\sup _{(t,s)\in \jp{\cS_{J}}} (t-s)^\gamma \int_{s}^{t} W^{(\infty)}_{q}(r, s) d r 
		&= C_\infty \sup _{(t,s)\in \jp{\cS_{J}}} (t-s)^\gamma \int_{s}^{t}  (r-s)^{-\gamma}e^{-\iota(r-s) + \vartheta(r-s)^{1-\gamma}} d r\\
		&\le C_\infty \sup _{(t,s)\in \jp{\cS_{J}}} e^{\vartheta(t-s)^{1-\gamma}} (t-s)^\gamma \int_{s}^{t}(r-s)^{-\gamma} e^{-\iota(r-s)}dr\\
		&\le C_\infty \sup _{(t,s)\in \jp{\cS_{J}}} e^{\vartheta(t-s)^{1-\gamma}}\frac{t-s}{1-\gamma}
		\le \jp{\overline{W}_{\infty,q}^{\cS_{J}}}
	\end{align}
	and that
	\begin{align}
		\sup _{(t,s)\in \jp{\cS_{J}}} (t-s)^\gamma \int_{s}^{t} W^{(\infty)}_{q}(t, r) d r 
		&= C_\infty \sup _{(t,s)\in \jp{\cS_{J}}}  (t-s)^\gamma \int_{s}^{t}  (t-r)^{-\gamma}e^{-\iota(t-r) + \vartheta(t-r)^{1-\gamma}} d r\\
		&\le C_\infty \sup _{(t,s)\in \jp{\cS_{J}}}  (t-s)^\gamma e^{\vartheta(t-s)^{1-\gamma}}\int_{s}^{t}  (t-r)^{-\gamma}e^{-\iota(t-r)} d r\\
		&\le C_\infty \sup _{(t,s)\in \jp{\cS_{J}}} e^{\vartheta(t-s)^{1-\gamma}}\frac{t-s}{1-\gamma}
		\le \jp{\overline{W}_{\infty,q}^{\cS_{J}}}.
	\end{align}
	These yield that \eqref{eq:ineqW_infty} holds.
	Third, note that
	\begin{align}
		&\sup _{(t,s)\in \jp{\cS_{J}}}(t-s)^\gamma\int_{s}^{t} \int_{s}^{r} W^{(\infty)}_{q}(r, \sigma) (\sigma-s)^{-\gamma}d \sigma d r\\
		&=C_\infty \sup _{(t,s)\in \jp{\cS_{J}}}(t-s)^\gamma\int_{s}^{t} \int_{s}^{r} (r-\sigma)^{-\gamma} e^{-\iota(r-\sigma) + \vartheta (r-\sigma)^{1-\gamma}} (\sigma-s)^{-\gamma} d \sigma d r\\
		&\le C_\infty \sup _{(t,s)\in \jp{\cS_{J}}}(t-s)^\gamma\int_{s}^{t} e^{ \vartheta (r-s)^{1-\gamma}} \int_{s}^{r} (r-\sigma)^{-\gamma} e^{-\iota(r-\sigma)} (\sigma-s)^{-\gamma} d \sigma d r\\
		&\le C_\infty \sup _{(t,s)\in \jp{\cS_{J}}}(t-s)^\gamma\int_{s}^{t} e^{ \vartheta (r-s)^{1-\gamma}} (r-s)^{1-2\gamma} \mathrm{B}(1-\gamma,1-\gamma) d r\\
		&\le C_\infty \sup _{(t,s)\in \jp{\cS_{J}}} e^{ \vartheta (t-s)^{1-\gamma}}\frac{(t-s)^{2-\gamma}}{2(1-\gamma)}\mathrm{B}(1-\gamma,1-\gamma)\\
		&\le \frac{\jp{\overline{W}_{\infty,q}^{\cS_{J}}}}2\tau^{1-\gamma}\mathrm{B}(1-\gamma,1-\gamma) =\jp{\doverline{W}_{\infty,q}^{\cS_{J}}}.
	\end{align}
	Finally, we have
	\begin{align}
		&\sup _{(t,s)\in \jp{\cS_{J}}} (t-s)^\gamma\int_{s}^{t} W^{(\infty)}_{q}(t, r)\int_{s}^r(\sigma -s)^{-\gamma}d\sigma d r\\
		&= \sup _{(t,s)\in \jp{\cS_{J}}} (t-s)^\gamma\int_{s}^{t} W^{(\infty)}_{q}(t, r)\frac{(r -s)^{1-\gamma}}{1-\gamma} d r\\
		&= C_\infty \sup _{(t,s)\in \jp{\cS_{J}}} \frac{(t-s)^\gamma}{1-\gamma} \int_{s}^{t} (t-r)^{-\gamma} e^{-\iota(t-r) + \vartheta (t-r)^{1-\gamma}}(r -s)^{1-\gamma}d r\\
		&\le C_\infty \sup _{(t,s)\in \jp{\cS_{J}}} e^{\vartheta (t-s)^{1-\gamma}} \frac{(t-s)^\gamma}{1-\gamma} \int_{s}^{t} (t-r)^{-\gamma} e^{-\iota(t-r)}(r -s)^{1-\gamma}d r\\
		&\le C_\infty \sup _{(t,s)\in \jp{\cS_{J}}} e^{\vartheta (t-s)^{1-\gamma}} \frac{(t-s)^{2-\gamma}}{1-\gamma} \mathrm{B}(1-\gamma,2-\gamma)\\
		&\le \frac{\jp{\overline{W}_{\infty,q}^{\cS_{J}}}}2\tau^{1-\gamma}\mathrm{B}(1-\gamma,1-\gamma) =\jp{\doverline{W}_{\infty,q}^{\cS_{J}}}.
	\end{align}
	Here we used the formula $\mathrm{B}(x,x) = 2\mathrm{B}(x+1,x)$ for $x>0$ in the last line.
	Hence, \eqref{eq:ineqbarW_infty} holds.
\end{proof}
}

\subsection{Validation theorem for local inclusion of the solution}\label{sec:local_inclusion}

\begin{lem}\label{lem:DFinv_bounds}
	Recall the definition of $DF(\ba)^{-1}$ in \eqref{eq:invDF} and assume that the uniform bounds satisfying \eqref{eq:Wh_bound_classic} and \eqref{eq:Wh_bound} hold. Let $p$, $\psi\in X$ and $\phi \in \ell_{\omega}^1$. Then $\cQ \psi + p \in Y$ and 
	\begin{align}\label{eq:inv_DF_estimate}
		\left\|DF(\ba)^{-1}(\cQ \psi + p,\phi)\right\| \le \bm{\jp{W^{\cS_{J}}}}\|\phi\|_{\omega} + \frac{\tau^{1-\gamma}\bm{\jp{W_q^{\cS_{J}}}}}{1-\gamma}\|\psi\| + \tau\bm{\jp{W^{\cS_{J}}}}\|p\|,\quad \gamma\in (0,1).
	\end{align}
\end{lem}
\begin{proof}
Recalling \eqref{eq:invDF}, and applying \eqref{eq:Wh_bound_classic}  and \eqref{eq:Wh_bound}, it follows that
\begin{align}
	\|DF(\ba)^{-1}(\cQ \psi+p,\phi)\| &\le \sup_{t\in J}\|U(t,0)\phi\|_{\omega} +  \sup_{t\in J}\int_{0}^t\|U(t,s)(\cQ \psi(s)+p(s))\|_{\omega}ds\\
	&\le \bm{\jp{W^{\cS_{J}}}}\|\phi\|_{\omega} +  \bm{\jp{W_q^{\cS_{J}}}}\|\psi\| \sup_{t\in J} \int_{0}^t(t-s)^{-\gamma}ds + \bm{\jp{W^{\cS_{J}}}}\sup_{t\in J} \int_{0}^t\|p(s)\|_{\omega}ds \\
	&\le \bm{\jp{W^{\cS_{J}}}}\|\phi\|_{\omega} + \frac{\tau^{1-\gamma}\bm{\jp{W_q^{\cS_{J}}}}}{1-\gamma}\|\psi\|+ \tau\bm{\jp{W^{\cS_{J}}}}\|p\|.\qedhere
\end{align}
\end{proof}
\begin{rem}
	The parameter $\gamma$ is taken as $\gamma=0$ in the case of $q=0$ in Lemma \ref{lem:DFinv_bounds}.
\end{rem}

\begin{thm}[\bf Local existence of the solution to IVP]\label{thm:local_inclusion}
	Given the approximate solution $\bar{a}\in \jp{\cD}$ of \eqref{eq:ODEs_finite}, assume that $\left\|(F(\bar{a}))_1\right\|=\|\dot{\ba}-\cL\ba-\cQ\cN(\ba)\|\le\delta$ and $\|(F(\ba))_{2}\|_{\omega}=\|\varphi-\bar{a}(0)\|_{\omega}\le\varepsilon$ hold.
	Assume also that for $a_1,a_2\in B_J\left(\bar{a},\varrho\right)$ there exists a non-decreasing function $L_{\ba}:(0,\infty)\to (0,\infty)$ such that
	\[
	\left\|\cN(a_1)-\mathcal{N}(a_2)-D\cN(\ba)(a_1-a_2)\right\|\le L_{\ba}(\varrho)\|a_1-a_2\|.
	\]
	Let us assume that there exists $\bm{\jp{W^{\cS_{J}}}}$, $\bm{\jp{W_q^{\cS_{J}}}}>0$ satisfying \eqref{eq:Wh_bound_classic}, \eqref{eq:Wh_bound}, respectively.
	Define
	\[
	p_{\varepsilon}\left(\varrho\right)\bydef {\bm{\jp{W^{\cS_{J}}}}\left(\varepsilon+\tau\delta\right) + \frac{\tau^{1-\gamma}\bm{\jp{W_q^{\cS_{J}}}}}{1-\gamma}L_{\ba}(\varrho)\varrho}.
	\]
	If there exists $\varrho_0>0$ such that
	\[
	p_{\varepsilon}\left(\varrho_0\right)\le\varrho_0,
	\]
	then there exists a unique $\tilde{a}\in B_J\left(\bar{a},\varrho\right)$ satisfying $F(\tilde{a})=0$ defined in \eqref{ODE_Cauchy}.
	Hence, the solution of the IVP \eqref{eq:IVP_PDE} exists locally in $J$.
\end{thm}

\begin{rem}\label{rem:Lba_rho}
	It is worth noting that the non-decreasing function $L_{\ba}(\varrho)$ depends on the nonlinear term $\cN$. For example, if the nonlinearity is cubic, i.e., 
	$\cN_{\bk}(a) = -(a^3)_{\bk}$, 
	it follows that for $a_1,a_2\in B_J\left(\bar{a},\varrho\right)$
	\begin{align}
		\left\|\cN(a_1)-\mathcal{N}(a_2)-D\cN(\ba)(a_1-a_2)\right\|
		&= \left\|a_1^3 - a_2^3 - 3\ba^2(a_1-a_2)\right\|\\
		&=\left\|(a_1-a_2)\left(a_1^2 + a_1a_2 + a_2^2 - 3\ba^2\right)\right\|\\
		&=\left\|(a_1-a_2)\left[(a_1-\ba)(a_1+a_2+\ba) + (a_2-\ba)(a_1+2\ba)\right]\right\|\\
		&\le 3\varrho\left(2\|\ba\|+\varrho\right)\|a_1-a_2\|.
	\end{align}
	Therefore, $L_{\ba}(\varrho)=3\varrho\left(2\|\ba\|+\varrho\right)$ holds in this case.
\end{rem}

\begin{proof}
	We prove that the operator $T$ defined in \eqref{eq:opT} is the contraction mapping on $B_J(\ba,\varrho_0)$ defined in \eqref{eq:theBall}.
	Firstly, for any $a\in B_J(\ba,\varrho_0)$, we have using \eqref{eq:DFa_bar} and \eqref{eq:opT}
	\begin{align}
		T(a)-\ba &= DF(\ba)^{-1}\left(DF(\ba)a-F(a)\right) - \ba\\
		&=DF(\ba)^{-1}\left(DF(\ba)(a-\ba)-F(a)\right)\\
		&=DF(\ba)^{-1}\left(\cQ\left(\cN(a)-D\cN(\ba)(a-\ba)\right)-\left(\dot{\ba}-\cL\ba\right),~\varphi-\ba(0)\right)\\
		&=DF(\ba)^{-1}\left[\left(\cQ\left(\cN(a)-\cN(\ba)-D\cN(\ba)(a-\ba)\right)-(F(\ba))_1,~(F(\ba))_{2}\right)\right].\label{eq:U_form}
	\end{align}
	It follows from \eqref{eq:inv_DF_estimate} and the assumption of the theorem that
	\begin{align}
		\|T(a)-\ba\|&\le {\bm{\jp{W^{\cS_{J}}}}\left\|(F(\ba))_2\right\|_{\omega}+\frac{\tau^{1-\gamma}\bm{\jp{W_q^{\cS_{J}}}}}{1-\gamma}\left\|\left(\cN(a)-\cN(\ba)-D\cN(\ba)(a-\ba)\right)\right\|+\tau\bm{\jp{W^{\cS_{J}}}}\left\|\left(F(\ba)\right)_1\right\|}\\
		&\le {\bm{\jp{W^{\cS_{J}}}}\left(\varepsilon+\tau\delta\right) + \frac{\tau^{1-\gamma}\bm{\jp{W_q^{\cS_{J}}}}}{1-\gamma}L_{\ba}(\varrho_0)\varrho_0}=p_{\varepsilon}(\varrho_0).\label{eq:g_estimate}
	\end{align}
	From the assumption of the theorem $p_{\varepsilon}\left(\varrho_0\right)\le\varrho_0$, $T(a)\in B_J(\ba,\varrho_0)$ holds.
	
	Secondary, we show the contraction property of $T$ on $B_J(\ba,\varrho_0)$. From \eqref{eq:opT} we have for $a_1, a_2\in B_J(\ba,\varrho_0)$
	\begin{align}
		T(a_1) - T(a_2)&=DF(\ba)^{-1}\left[\left(\cQ \cN(a_1)-\cQ D\cN(\ba)a_1,~\varphi\right)-\left(\cQ \cN(a_2)-\cQ D\cN(\ba)a_2,~\varphi\right)\right]\\
		&=DF(\ba)^{-1}\left[\left( \cQ\left(\cN(a_1)- \cN(a_2)- D\cN(\ba)(a_1-a_2)\right),~0\right)\right].
	\end{align}
	It follows from \eqref{eq:inv_DF_estimate} and the assumption of the theorem that
	\begin{align}
		\|T(a_1)-T(a_2)\|&\le {\frac{\tau^{1-\gamma}\bm{\jp{W_q^{\cS_{J}}}}}{1-\gamma}}\left\|\left(\cN(a_1)-\cN(a_2)-D\cN(\ba)(a_1-a_2)\right)\right\|\\
		&\le {\frac{\tau^{1-\gamma}\bm{\jp{W_q^{\cS_{J}}}}}{1-\gamma}}L_{\ba}(\varrho_0)\|a_1-a_2\|.
	\end{align}
	Hence, 
	${\frac{\tau^{1-\gamma}\bm{\jp{W_q^{\cS_{J}}}}}{1-\gamma}L_{\ba}(\varrho_0)}<p_{\varepsilon}(\varrho_0)/\varrho_0\le 1$ holds. Therefore, it is proved that the operator $T$ is the contraction mapping on $B_J(\ba,\varrho_0)$.
\end{proof}

The rest of this section is dedicated to the construction of bounds $\varepsilon$ and $\delta$ of Theorem \ref{thm:local_inclusion}.

\paragraph{The bound \boldmath$\varepsilon$\boldmath.}

We obtain the $\varepsilon$ bound of Theorem \ref{thm:local_inclusion}.
It follows from \eqref{eq:appsol} and \eqref{eq:a_bar} that
\begin{align}
	\ba_{\bk}(0)=\ba_{0,\bk}+2\sum_{\ell=1}^{n-1}\ba_{\ell,\bk}T_{\ell}(0)=\ba_{0,\bk}+2\sum_{\ell=1}^{n-1}(-1)^{\ell}\ba_{\ell,\bk}.
\end{align}
Then  the $\varepsilon$ bound is given by
\begin{align}
	\varepsilon = \sum_{\bk\in\FN}\left|\varphi_{\bk} - \left(\ba_{0,\bk}+2\sum_{\ell=1}^{n-1}(-1)^{\ell}\ba_{\ell,\bk}\right)\right|\omega_{\bk} + \sum_{\bk\notin\FN} |\varphi_{\bk}|\omega_{\bk}.
\end{align}


\paragraph{The bound \boldmath$\delta$\boldmath.}

We next introduce the $\delta$ bound of Theorem \ref{thm:local_inclusion}.
From the definition of the operator $F$ in \eqref{ODE_Cauchy} and that of the approximate solution \eqref{eq:a_bar},  we recall
\begin{align}\label{eq:F_ba}
	(F_{\bk}(\ba))_2(t) = \begin{cases}
		\frac{d}{dt}\ba_{\bk}(t)-\mu_{\bk}\ba_{\bk}(t)-\jp{\im^q} (\bk\bL)^q\cN_{\bk}(\ba(t)) & (\bk\in\FN)\\
		-\jp{\im^q} (\bk\bL)^q\cN_{\bk}(\ba(t)) & (\bk\notin\FN).
	\end{cases}
\end{align}
We suppose that the derivative of $\ba$ with respect to $t$ is denoted by
\begin{align}
	\frac{d}{dt}\ba_{\bk}(t) = \ba_{0,\bk}^{(1)} + 2\sum_{\ell=1}^{n-2}\ba_{\ell,\bk}^{(1)} T_{\ell}(t),\quad \bk\in \FN,
\end{align}
where the coefficients $\ba_{\ell,\bk}^{(1)}$ can be computed by the explicit formula, see, e.g., \cite[Sec.\ 2.4.5]{Chebyshev_book},
\begin{align}
	\ba_{\ell,\bk}^{(1)}=\sum_{r =\ell+1 \atop r-\ell\ \textrm{odd}}^{n-1} 2r\ba_{r,\bk},\quad\ell=0,\dots,n-1,\quad\bk\in\FN.
\end{align}

On the one hand, for $\bk\in\FN$, it follows from \eqref{eq:F_ba} and the fact $|T_\ell(t)|\le 1$ for all $\ell\ge0$ that
\begin{align}
	&\sup_{t \in J}\left|(F_{\bk}(\ba))_2(t)\right|\\
	&=\sup_{t \in J}\left|\frac{d}{dt}\ba_{\bk}(t)-\mu_{\bk}\ba_{\bk}(t)-\jp{\im^q}(\bk\bL)^q\cN_{\bk}(\ba(t))\right|\\
	&\le \sum_{\ell\ge 0}\left|\ba_{\ell,\bk}^{(1)}- \mu_{\bk}\ba_{\ell,\bk}-\jp{\im^q}(\bk\bL)^q\cN_{\ell,\bk}(\ba)\right|\mathfrak{m}_\ell\\
	&\le \sum_{\ell=0}^{n-2}\left|\ba_{\ell,\bk}^{(1)}- \mu_{\bk}\ba_{\ell,\bk}-\jp{\im^q}(\bk\bL)^q\cN_{\ell,\bk}(\ba)\right|\mathfrak{m}_\ell + \left| \mu_{\bk}\ba_{n-1,\bk}+\jp{\im^q}(\bk\bL)^q\cN_{n-1,\bk}(\ba)\right|\mathfrak{m}_{n-1}+\sum_{\ell\ge n}\left|(\bk\bL)^q\cN_{\ell,\bk}(\ba)\right|\mathfrak{m}_\ell,\label{eq:defect_Fn}
\end{align}
where $\mathfrak{m}_\ell$ denotes the multiplicity for the Chebyshev coefficients 
\begin{align}
	\mathfrak{m}_\ell=\begin{cases}
		1 & (\ell = 0)\\
		2 & (\ell>0)
	\end{cases}
\end{align}
and
$\cN_{\ell,\bk}(\ba)$ is the Chebyshev-Fourier coefficients of the nonlinear term $\cN_{\bk}(\ba(t))$ such that
\begin{align}
	\cN_{\bk}(\ba(t)) = \sum_{\ell\ge0}\cN_{\ell,\bk}(\ba)\mathfrak{m}_\ell T_{\ell}(t).
\end{align} 

On the other hand, in the case of $\bk\notin\FN$, we have from \eqref{eq:F_ba}
\begin{align}\label{eq:defect_tail}
	\sup_{t \in J}\left|(F_{\bk}(\ba))_2(t)\right|=
	\sup_{t \in J}\left|(\bk\bL)^q\cN_{\bk}(\ba(t))\right|\le \sum_{\ell\ge 0}\left|(\bk\bL)^q\cN_{\ell,\bk}(\ba)\right|\mathfrak{m}_\ell.
\end{align}
It is worth noting that if we consider the cubic nonlinearity, i.e., $\cN_{\bk}(a) = -(a^3)_{\bk}$, for instance, the above coefficients are given by the discrete convolution for $d+1$ dimensional tensor.
\begin{align}
	\cN_{\ell,\bk}(\ba)=-\sum_{\substack{ \ell_1+\ell_2+\ell_3 = \pm\ell\\\bk_1+\bk_2 + \bk_3 = \pm\bk\\ |\ell_i|<n,~|\bk_i| \in \FN}} \ba_{|\ell_1|,\bk_1}\ba_{|\ell_2|,\bk_2}\ba_{|\ell_3|,\bk_3}.
\end{align}
Furthermore, thanks to the finiteness of the nonzero elements in $\ba$, these coefficients also have a finite number of nonzero elements. The indices of these nonzero elements are  $\bk\in \bm{F_{3N-2}}$ in the Fourier dimension and $\ell=0,\dots,3(n-1)$ in the Chebyshev dimension.
This implies that the last term of \eqref{eq:defect_Fn} and that of \eqref{eq:defect_tail} are finite sums in this case.

Finally, the defect bound $\delta$ is given from \eqref{eq:defect_Fn} and \eqref{eq:defect_tail} by
\begin{align}
	\left\|(F(\ba))_2\right\| &= \sup_{t \in J}\left(\sum_{\bk\in\FN}\left|(F_{\bk}(\ba))_2(t)\right|\omega_{\bk} + \sum_{\bk\notin\FN}\left|(F_{\bk}(\ba))_2(t)\right|\omega_{\bk}\right)\\
	&\le \sum_{\bk\in\FN} \sup_{t \in J}\left|(F_{\bk}(\ba))_2(t)\right|\omega_{\bk} + \sum_{\bk\notin\FN} \sup_{t \in J}\left|(F_{\bk}(\ba))_2(t)\right|\omega_{\bk}\\
	&\le  \sum_{\bk\in\FN}\Bigg(\sum_{\ell=0}^{n-2}\left|\ba_{\ell,\bk}^{(1)}- \mu_{\bk}\ba_{\ell,\bk}-\jp{\im^q} (\bk\bL)^q\cN_{\ell,\bk}(\ba)\right|\mathfrak{m}_\ell + \left| \mu_{\bk}\ba_{n-1,\bk}+\jp{\im^q}(\bk\bL)^q\cN_{n-1,\bk}(\ba)\right|\mathfrak{m}_{n-1}\\
	&\hphantom{\le}\quad +\sum_{\ell\ge n}\left|(\bk\bL)^q\cN_{\ell,\bk}(\ba)\right|\mathfrak{m}_\ell\Bigg)\omega_{\bk} + \sum_{\bk\notin\FN}\sum_{\ell\ge 0}\left|(\bk\bL)^q\cN_{\ell,\bk}(\ba)\right|\mathfrak{m}_\ell\omega_{\bk}\bydef \delta.\label{eq:delta_bound}
\end{align}

\section{Multi-step scheme of rigorous integration}\label{sec:time_stepping}

In this section, a \emph{multi-step scheme} is introduced to extend the existence time of the solution to the initial value problem \eqref{eq:IVP_PDE}.
\jp{This technique is similar to the \emph{domain decomposition} of the time interval, which is standard for computer-assisted proofs of integrating differential equations presented in \cite{JBMaxime,MR2518006}, etc.}
By defining the $F$ map within discrete time intervals and numerically verifying the zero-finding problem for the coupled system, it is possible to rigorously include the solution in multiple time intervals.
In this context, we consider a sequence of \emph{time steps} denoted by $J_i = (t_{i-1}, t_i]$ ($i=1,2,\dots,K$), where $0 = t_0 < t_1 < \dots<t_K$ and the step size $\tau_i\bydef t_i-t_{i-1}$ can be changed.
On each time step $J_i$, we represent the function spaces $X$, $Y$, and $\cD$ as $X_i$, $Y_i$, and $\cD_i$, respectively.
We denote the Banach space
\begin{align}
	\bm{X}\bydef\prod_{i=1}^{K} X_i.
\end{align}
The norm on the space $\bm{X}$ is defined for given $\bm{a}=(a^{J_1},a^{J_2},\dots,a^{J_K})\in\bm{X}$ by
\begin{align}
	\|\bm{a}\|_{\bm{X}} = \max\left\{\left\|a^{J_1}\right\|_{X_1},\left\|a^{J_2}\right\|_{X_2},\dots,\left\|a^{J_K}\right\|_{X_K}\right\},
\end{align}
where the superscript is used to indicate the dependence on each time step.
Similarly, we use the bold style notations in the same manner for denoting $\bm{\cD}\bydef\prod_{i=1}^{K} \cD_i$, $\bm{Y}\bydef\prod_{i=1}^{K} Y_i$, $\bm{\ell}_{\omega}^1=\prod_{i=1}^{K} \ell_{\omega}^1$, and $\bm{Y}\times \bm{\ell}_{\omega}^1\bydef\prod_{i=1}^{K} Y_i\times \ell_{\omega}^1$.

The mapping $F:\jp{\bm{\cD}}\to\bm{Y}\times \bm{\ell}_{\omega}^1$ for the initial value problem over multiple time steps is defined by
\begin{align}\label{eq:mapF_multisteps}
	F(\bm{a})=\left(F_1\left(a^{J_1}\right), F_2\left(a^{J_1},a^{J_2}\right),\dots,F_K\left(a^{J_{K-1}},a^{J_{K}}\right)\right),
\end{align}
where $F_1:\jp{\cD_1}\to Y_1\times \ell_{\omega}^1$ and $F_i:\jp{\cD_{i-1}}\times \jp{\cD_i}\to Y_i\times \ell_{\omega}^1$ ($i=2,3,\dots,K$) are given by
\begin{align}\label{Fi_maps}
	F_1\left(a^{J_1}\right)&\bydef\left(\dot{a}^{J_1}-\cL a^{J_1}-\cQ\,\cN\left(a^{J_1}\right),~a^{J_1}(t_0)-\varphi\right)\\
	F_2\left(a^{J_1},a^{J_2}\right)&\bydef\left(\dot{a}^{J_2}-\cL a^{J_2}-\cQ\,\cN\left(a^{J_2}\right),~a^{J_2}(t_1)-a^{J_1}(t_1)\right)\\
	&~~\vdots\\
	F_K\left(a^{J_{K-1}},a^{J_{K}}\right)&\bydef\left(\dot{a}^{J_{K}}-\cL a^{J_K}-\cQ\,\cN\left(a^{J_K}\right),~a^{J_K}(t_{K-1})-a^{J_{K-1}}(t_{K-1})\right).
\end{align}
As we denoted in Section \ref{sec:fixed-point-form}, the first and second components of $F_i$ ($i=1,2,\dots, K$) are also denoted by $(F_i)_1$ and $(F_i)_{2}$, respectively. From the above construction, solving the initial value problem \eqref{eq:IVP_PDE} over multiple time steps reduces into the zero-finding problem for the map $F$ given in \eqref{eq:mapF_multisteps}.

The approximate solution is set by
$\bm{\ba}=\left(	\ba^{J_1},\ba^{J_2},\dots,\ba^{J_K}\right)$,
which is numerically computed on each time step $J_i$.
The Fr\'echet derivative at the approximate solution $DF(\bm{\ba})$ is denoted by
\begin{align}\label{eq:DF_multisteps}
	DF(\bm{\ba})\bm{h}\bydef\begin{bmatrix}
		\left(\dot{h}^{J_1}-\cL h^{J_1}-\cQ D\cN\left(\ba^{J_1}\right)h^{J_1},~h^{J_1}(t_0)\right)\\
		\left(\dot{h}^{J_2}-\cL h^{J_2}-\cQ D\cN\left(\ba^{J_2}\right)h^{J_2},~h^{J_2}(t_1)-h^{J_1}(t_1)\right)\\
		\vdots\\
		\left(\dot{h}^{J_{K}}-\cL h^{J_K}-\cQ\,\cN\left(\ba^{J_K}\right)h^{J_K},~h^{J_K}(t_{K-1})-h^{J_{K-1}}(t_{K-1})\right)
	\end{bmatrix}\in \bm{Y}\times \bm{\ell}_{\omega}^1
\end{align}
for $\bm{h}=(h^{J_1},h^{J_2},\dots,h^{J_K})$.

To define the Newton-like operator for validating the solution over multiple time steps, we obtain a formula for the inverse of the Fr\'echet derivative, that is, we solve $DF(\bm{\ba})\bm{h}=(\bm{p},\bm{\phi})$ for given $\bm{p}=(p^{J_1},p^{J_2},\dots,p^{J_K})\in\bm{Y}$ and $\bm{\phi}=(\phi^{J_1},\phi^{J_2},\dots,\phi^{J_K})$.
\jp{
In order to perform this task, we first construct the evolution operators $\{U_{J_i}(t,s)\}_{t_{i-1}\le s\le t\le t_i}$ ($i=1,2,\dots,K$) validated on each time step $J_i$ via Theorem \ref{thm:solutionmap} in Section \ref{sec:solution_map} and then, denoting ${\cS_{J_i}} \bydef \{(t, s) : t_{i-1} < s < t\le t_{i}\}$ $(i=1,2,\dots,K)$, we obtain positive constants $\bm{{W^{\cS_{J_i}}}}$, $\bm{{W_q^{\cS_{J_i}}}}$, $\bm{\jp{W^{(t_i,J_i)}}}$, $\bm{\jp{W^{(t_i,J_i)}_q}}$, $\bm{\jp{W^{(t_i,t_{i-1})}}}$ such that
\begin{equation} \label{eq:bounds_evolution_operators_multi-steps}
\begin{aligned} 
		\sup_{(t,s)\in\jp{\cS_{J_i}}}\left\|U_{J_i}(t,s)\right\|_{B (\ell_{\omega}^1)}\le \bm{\jp{W^{\cS_{J_i}}}},&\quad
		\sup_{(t,s)\in\jp{\cS_{J_i}}}(t-s)^\gamma\left\|U_{J_i}(t,s)\cQ\right\|_{B (\ell_{\omega}^1)}\le \bm{{W_q^{\cS_{J_i}}}},
\\
		\sup_{s\in J_i}\|U_{J_i}(t_i,s)\|_{B(\ell_{\omega}^1)}\le \bm{\jp{W^{(t_i,J_i)}}},&\quad
		\sup_{s\in[t_{i-1},t_i)}(t_i-s)^\gamma\|U_{J_i}(t_i,s)\cQ\|_{B(\ell_{\omega}^1)}\le \bm{\jp{W^{(t_i,J_i)}_q}},
\\
		\|U_{J_i}(t_i,t_{i-1})\|_{B(\ell_{\omega}^1)}\le \bm{\jp{W^{(t_i,t_{i-1})}}},&
	\end{aligned}
	\end{equation}
respectively.
Note that these constants fall into three types: First, $\bm{\jp{W^{\cS_{J_i}}}}$ and $\bm{{W_q^{\cS_{J_i}}}}$ are bounds for the two parameters $(t,s)$, with the superscript $\cS_{J_i}$ indicating the dependency $(t,s)\in\cS_{J_i}$. Second, $\bm{\jp{W^{(t_i,J_i)}}}$ and $\bm{\jp{W^{(t_i,J_i)}_q}}$ are bounds for the one parameter $s$, where the superscript $(t_i,J_i)$ indicates that $t=t_i$ is fixed and $s\in J_i$. Third, $\bm{\jp{W^{(t_i,t_{i-1})}}}$ denotes the operator norm at the specific times $t=t_i$ and $s=t_{i-1}$.

\begin{rem}
Note that computing each evolution operator $U_{J_i}(t,s)$ by solving the linearized problems and then obtaining the bounds in \eqref{eq:bounds_evolution_operators_multi-steps} at each step $i=1,\dots,K$ can be done independently. Hence each computation can be made in parallel, which implies that the computational cost is additive in time rather than multiplicative. 
\end{rem}
Having obtained the evolution operators $U_{J_i}(t,s)$ on each $J_i$, the solution $\bm{h}$ of  $DF(\bm{\ba})\bm{h}=(\bm{p},\bm{\phi})$ is derived iteratively as follows:}
The first step is the same as that presented in Section \ref{sec:inverse_DF}. The variation of constants formula yields
\begin{align}
	h^{J_1}(t)=U_{J_1}(t,t_0)\phi^{J_1}+\int_{t_0}^tU_{J_1}(t,s)p^{J_1}(s)ds,\quad t\in J_1.
\end{align}

Next, the $h^{J_2}$ is given by
\begin{align}
	h^{J_2}(t)&=U_{J_2}(t,t_1)\left(\phi^{J_2}+h^{J_1}(t_1)\right)+\int_{t_1}^tU_{J_2}(t,s)p^{J_2}(s)ds\\
	&=U_{J_2}(t,t_1)\left(\phi^{J_2}+U_{J_1}(t_1,t_0)\phi^{J_1}+\int_{J_1}U_{J_1}(t_1,s)p^{J_1}(s)ds\right)+\int_{t_1}^tU_{J_2}(t,s)p^{J_2}(s)ds\\
	&=U_{J_2}(t,t_1)\phi^{J_2}+\int_{t_1}^tU_{J_2}(t,s)p^{J_2}(s)ds+U_{J_1\cup J_2}(t,t_0)\phi^{J_1}+\int_{J_1}U_{J_1\cup J_2}(t,s)p^{J_1}(s)ds,\quad t\in J_2.\label{eq:invDF_J2}
\end{align}
Here, we denote $U_{J_1\cup J_2}(t,t_0)\bydef U_{J_2}(t,t_1)U_{J_1}(t_1,t_0)$ ($t\in J_2$). 
We iterate this process for $i=1,2,\dots,K$. The resulting $h^{J_K}$ is obtained for $t\in J_K$ by
\begin{align}
	&h^{J_K}(t)\\
	&=U_{J_K}(t,t_{K-1})\left(\phi^{J_K}+h^{J_{K-1}}(t_{K-1})\right)+\int_{t_{K-1}}^{t}U_{J_K}(t,s)p^{J_K}(s)ds\\
	&=U_{J_K}(t,t_{K-1})\left(\phi^{J_K}+U_{J_{K-1}}(t_{K-1},t_{K-2})h^{J_{K-2}}(t_{K-2})+\int_{J_{K-1}}U_{J_{K-1}}(t_{K-1},s)p^{J_{K-1}}(s)ds\right)\\
	&\hphantom{=}\quad+\int_{t_{K-1}}^tU_{J_K}(t,s)p^{J_K}(s)ds\\
	&~~\vdots\\
	&=U_{J_K}(t,t_{K-1})\phi^{J_K} + \int_{t_{K-1}}^tU_{J_K}(t,s)p^{J_K}(s)ds+ \sum_{i=1}^{K-1}\left( U_{\bigcup_{j=i}^KJ_j}(t,t_{i-1})\phi^{J_{i}} + \int_{J_i}U_{\bigcup_{j=i}^KJ_j}(t,s)p^{J_i}(s)ds\right),\label{eq:invDF_Jk}
\end{align}
where
\begin{align}
	U_{\bigcup_{j=i}^KJ_j}(t,t_{i-1})\bydef U_{J_K}(t,t_{K-1})U_{J_{K-1}}(t_{K-1},t_{J_{K-2}})\cdots U_{J_i}(t_i,t_{i-1}),\quad t\in J_K.
\end{align}
Finally, we get an unique expression of the inverse of $DF(\bm{\ba})$ as
\begin{align}
	&DF(\bm{\ba})^{-1}(\bm{p},\bm{\phi})=\\
	&\begin{bmatrix}
		U_{J_1}(t,t_0)\phi^{J_1}+\int_{t_0}^tU_{J_1}(t,s)p^{J_1}(s)ds\\
		U_{J_2}(t,t_1)\phi^{J_2}+\int_{t_1}^tU_{J_2}(t,s)p^{J_2}(s)ds+U_{J_1\cup J_2}(t,t_0)\phi^{J_1}+\int_{J_1}U_{J_1\cup J_2}(t,s)p^{J_1}(s)ds\\
		\vdots\\
		U_{J_K}(t,t_{K-1})\phi^{J_K} + \int_{t_{K-1}}^tU_{J_K}(t,s)p^{J_K}(s)ds+\sum_{i=1}^{K-1}\left( U_{\bigcup_{j=i}^KJ_j}(t,t_{i-1})\phi^{J_{i}} + \int_{J_i}U_{\bigcup_{j=i}^KJ_j}(t,s)p^{J_i}(s)ds\right)
	\end{bmatrix}.\label{eq:DFinv_multisteps}
\end{align}

The Newton-like operator is defined by
\begin{align}\label{eq:opT_multisteps}
	T(\bm{a})\bydef DF(\bm{\ba})^{-1}\left(DF(\bm{\ba})\bm{a}-F(\bm{a})\right),\quad T:\bm{X}\to\jp{\bm{\cD}}\subset\bm{X}.
\end{align}
Therefore, we validate the solution of $F(\bm{a})=0$ given in \eqref{eq:mapF_multisteps} to show that $T:\bm{B}_K(\bm{\ba}, \bm{\varrho}) \to \bm{B}_K(\bm{\ba}, \bm{\varrho})$ is a contraction mapping, for some $\bm{\varrho}>0$, where
\begin{align}\label{eq:theBall_multisteps}
	\bm{B}_K(\bm{\ba}, \bm{\varrho}) \bydef \left\{\bm{a} \in \bm{X}:\|\bm{a}-\bm{\ba}\|_{\bm{X}} \leq \bm{\varrho}\right\}.
\end{align}

\begin{lem}\label{lem:DFinv_estimate}
\jp{For each $i=1,2,\dots,K$, consider the positive constants $\bm{\jp{W^{\cS_{J_i}}}}$, $\bm{\jp{W_q^{\cS_{J_i}}}}$, $\bm{\jp{W^{(t_i,J_i)}}}$, $\bm{\jp{W^{(t_i,J_i)}_q}}$, $\bm{\jp{W^{(t_i,t_{i-1})}}}$ satisfying \eqref{eq:bounds_evolution_operators_multi-steps}}, and let $\bm{W_\ell}>0$, $\bm{W_{X}}>0$ and $\bm{W_{Y}}>0$ be defined by the matrix infinity norm of the following lower triangular matrices:
	\begin{align}
		\bm{W_\ell}\bydef
		\left\|\begin{pmatrix}
			\bm{\jp{W^{\cS_{J_1}}}}&0&&\\
			\bm{\jp{W^{\cS_{J_2}}}}\bm{\jp{W^{(t_1,t_0)}}}&\bm{\jp{W^{\cS_{J_2}}}}&$\Huge 0$&\\
			\vdots&\vdots&\ddots&\\
			\bm{\jp{W^{\cS_{J_K}}}}\bm{\jp{W^{(t_{K-1},t_0)}}}&\bm{\jp{W^{\cS_{J_K}}}}\bm{\jp{W^{(t_{K-1},t_1)}}}&\dots&\bm{\jp{W^{\cS_{J_K}}}}
		\end{pmatrix}\right\|_\infty,
	\end{align}
	\begin{align}
		\bm{W_X}\bydef
		&{\left\|\begin{pmatrix}
				\hat{\tau}_1\bm{\jp{W_q^{\cS_{J_1}}}}&0&&  \\
				\hat{\tau}_1\bm{\jp{W^{\cS_{J_2}}}}\bm{\jp{W^{(t_1,J_1)}_q}}&\hat{\tau}_2\bm{\jp{W_q^{\cS_{J_2}}}}&$\Huge 0$&\\
				\vdots&\vdots&\ddots&\\
				\hat{\tau}_1\bm{\jp{W^{\cS_{J_K}}}}\bm{\jp{W^{(t_{K-1},t_1)}}}\bm{\jp{W^{(t_1,J_1)}_q}}&\hat{\tau}_2\bm{\jp{W^{\cS_{J_K}}}}\bm{\jp{W^{(t_{K-1},t_2)}}}\bm{\jp{W^{(t_2,J_2)}_q}}&\dots&\hat{\tau}_K\bm{\jp{W_q^{\cS_{J_K}}}}
			\end{pmatrix}\right\|_{\infty},}
	\end{align}
	where $\hat{\tau}_i\bydef\frac{\tau_i^{1-\gamma}}{1-\gamma}$ ($i=1,2,\dots,K$), and
	\begin{align}
		\bm{W_Y}\bydef
		&{\left\|\begin{pmatrix}
				\tau_1\bm{\jp{W^{\cS_{J_1}}}}&0&&  \\
				\tau_1\bm{\jp{W^{\cS_{J_2}}}}\bm{\jp{W^{(t_1,J_1)}}}&\tau_2\bm{\jp{W^{\cS_{J_2}}}}&$\Huge 0$&\\
				\vdots&\vdots&\ddots&\\
				\tau_1\bm{\jp{W^{\cS_{J_K}}}}\bm{\jp{W^{(t_{K-1},t_1)}}}\bm{\jp{W^{(t_1,J_1)}}}&\tau_2\bm{\jp{W^{\cS_{J_K}}}}\bm{\jp{W^{(t_{K-1},t_2)}}}\bm{\jp{W^{(t_2,J_2)}}}&\dots&\tau_K\bm{\jp{W^{\cS_{J_K}}}}
			\end{pmatrix}\right\|_{\infty},}
	\end{align}
	respectively. \jp{Here, we denote $\bm{W^{(t_i,t_j)}}=\bm{W^{(t_i,t_{i-1})}}\cdots\bm{W^{(t_{j+1},t_j)}}$ with $0\le j<i\le K-1$.} Then, these constants obey the following inequality:
	\begin{align}
		\left\|DF(\bm{\ba})^{-1}(\cQ\bm{\psi}+\bm{p},\bm{\phi})\right\|_{\bm{X}}\le \bm{W_\ell}\|\bm{\phi}\|_{\bm{\omega}} + \bm{W_X}\|\bm{\psi}\|_{\bm{X}} + \bm{W_Y}\|\bm{p}\|_{\bm{X}}\quad\mbox{for}\quad\bm{\phi}\in \bm{\ell_{\omega}^1},~ \bm{\psi}, \bm{p}\in\bm{X},
	\end{align}
	where $\|\bm{\phi}\|_{\bm{\omega}}\bydef \max\{\|\phi^{J_i}\|_\omega\}$.
\end{lem}

\begin{proof}
	For given $\bm{\phi}=(\phi^{J_1},\phi^{J_2},\dots,\phi^{J_K})\in \bm{\ell_{\omega}^1}$, $\bm{\psi}=(\psi^{J_1},\psi^{J_2},\dots,\psi^{J_K})\in\bm{X}$ and $\bm{p}=(p^{J_1},p^{J_2},\dots,p^{J_K})\in\bm{X}$,
	taking $X_i$ norm in each component of \eqref{eq:DFinv_multisteps}, we have from \eqref{eq:inv_DF_estimate}
	\begin{align}
		&\sup_{t \in J_1}\left\|U_{J_1}(t,t_0)\phi^{J_1}+\int_{t_0}^tU_{J_1}(t,s)(\cQ\psi^{J_1}(s) + p^{J_1}(s))ds\right\|_{\omega}\\
		&\le \bm{\jp{W^{\cS_{J_1}}}}\|\phi^{J_1}\|_{\omega} + \hat{\tau}_1\bm{\jp{W_q^{\cS_{J_1}}}}\|\psi^{J_1}\|_{X_1} + \tau_1\bm{\jp{W^{\cS_{J_1}}}}\|p^{J_1}\|_{X_1},
	\end{align}
	where $\hat{\tau}_i=\frac{\tau_i^{1-\gamma}}{1-\gamma}$ ($i=1,2,\dots,K$). We also have
	\begin{align}
		&\sup_{t \in J_2}\Bigg\|U_{J_2}(t,t_1)\phi^{J_2}+\int_{t_1}^tU_{J_2}(t,s)(\cQ\psi^{J_2}(s) + p^{J_2}(s))ds\\
		&\quad+U_{J_1\cup J_2}(t,t_0)\phi^{J_1} +\int_{J_1}U_{J_1\cup J_2}(t,s)(\cQ\psi^{J_1}(s) + p^{J_1}(s))ds\Bigg\|_{\omega}\\
		&\le \bm{\jp{W^{\cS_{J_2}}}}\bm{\jp{W^{(t_1,t_0)}}}\|\phi^{J_1}\|_{\omega} + \bm{\jp{W^{\cS_{J_2}}}}\|\phi^{J_2}\|_{\omega} + \hat{\tau}_1\bm{\jp{W^{\cS_{J_2}}}}\bm{\jp{W^{(t_1,J_1)}_q}}\|\psi^{J_1}\|_{X_1} +\hat{\tau}_2\bm{\jp{W_q^{\cS_{J_2}}}}\|\psi^{J_2}\|_{X_2}\\
		&\quad\hphantom{\le}\quad+ \tau_1\bm{\jp{W^{\cS_{J_2}}}}\bm{\jp{W^{(t_1,J_1)}}}\|p^{J_1}\|_{X_1} +\tau_2\bm{\jp{W^{\cS_{J_2}}}}\|p^{J_2}\|_{X_2},
	\end{align}
	where we used
	\begin{align}
		\sup_{t \in J_2}\|U_{J_1\cup J_2}(t,t_0)\|_{B (\ell_{\omega}^1)}\le \sup_{t \in J_2}\|U_{J_2}(t,t_1)\|_{B (\ell_{\omega}^1)}\|U_{J_1}(t_1,t_0)\|_{B (\ell_{\omega}^1)}\le\bm{\jp{W^{\cS_{J_2}}}}\bm{\jp{W^{(t_1,t_0)}}}
	\end{align}
	\begin{align}
		\sup_{t\in J_2,\,s\in J_1}(t_1-s)^\gamma\|U_{J_1\cup J_2}(t,s)\cQ\|_{B (\ell_{\omega}^1)}\le \sup_{t \in J_2}\|U_{J_2}(t,t_1)\|_{B (\ell_{\omega}^1)}\sup_{s\in J_1}(t_1-s)^\gamma\|U_{J_1}(t_1,s)\|_{B(\ell_{\omega}^1)}\le\bm{\jp{W^{\cS_{J_2}}}}\bm{\jp{W^{(t_1,J_1)}_q}}
	\end{align}
	\begin{align}
		\sup_{t\in J_2,\,s\in J_1}\|U_{J_1\cup J_2}(t,s)\|_{B (\ell_{\omega}^1)}\le \sup_{t \in J_2}\|U_{J_2}(t,t_1)\|_{B (\ell_{\omega}^1)}\sup_{s\in J_1}\|U_{J_1}(t_1,s)\|_{B(\ell_{\omega}^1)}\le\bm{\jp{W^{\cS_{J_2}}}}\bm{\jp{W^{(t_1,J_1)}}}.
	\end{align}
	Similarly, we obtain the $K$-th component
	\begin{align}
		&\sup_{t \in J_K}\Bigg\|U_{J_K}(t,t_{K-1})\phi^{J_K} + \int_{t_{K-1}}^tU_{J_K}(t,s)(\cQ\psi^{J_K}(s)+p^{J_K}(s))ds\\
		&\quad +\sum_{i=1}^{K-1}\left( U_{\bigcup_{j=i}^KJ_j}(t,t_{i-1})\phi^{J_{i}} + \int_{J_i}U_{\bigcup_{j=i}^KJ_j}(t,s)(\cQ\psi^{J_i}(s)+p^{J_i}(s))ds\right)\Bigg\|_{\omega}\\
		&\le \bm{\jp{W^{\cS_{J_K}}}}\bm{\jp{W^{(t_{K-1},t_0)}}}\|\phi^{J_1}\|_{\omega}+\bm{\jp{W^{\cS_{J_K}}}}\bm{\jp{W^{(t_{K-1},t_1)}}}\|\phi^{J_2}\|_{\omega}+\dots+\bm{\jp{W^{\cS_{J_K}}}}\|\phi^{J_K}\|_{\omega}\\
		&+\hat{\tau}_1\bm{\jp{W^{\cS_{J_K}}}}\bm{\jp{W^{(t_{K-1},t_1)}}}\bm{\jp{W^{(t_1,J_1)}_q}}\|\psi^{J_1}\|_{X_1} + \hat{\tau}_2\bm{\jp{W^{\cS_{J_K}}}}\bm{\jp{W^{(t_{K-1},t_2)}}}\bm{\jp{W^{(t_2,J_2)}_q}}\|\psi^{J_2}\|_{X_2} + \dots + \hat{\tau}_K\bm{\jp{W_q^{\cS_{J_K}}}}\|\psi^{J_K}\|_{X_K}\\
		&+\tau_1\bm{\jp{W^{\cS_{J_K}}}}\bm{\jp{W^{(t_{K-1},t_1)}}}\bm{\jp{W^{(t_1,J_1)}}}\|p^{J_1}\|_{X_1} + \tau_2\bm{\jp{W^{\cS_{J_K}}}}\bm{\jp{W^{(t_{K-1},t_2)}}}\bm{\jp{W^{(t_2,J_2)}}}\|p^{J_2}\|_{X_2} + \dots + \tau_K\bm{\jp{W^{\cS_{J_K}}}}\|p^{J_K}\|_{X_K},
	\end{align}
	\jp{where $\bm{W^{(t_i,t_j)}}=\bm{W^{(t_i,t_{i-1})}}\cdots\bm{W^{(t_{j+1},t_j)}}$ ($0\le j<i\le K-1$).}
	Summing up the above inequalities and taking the maximum norm of each components, the $\bm{X}$ norm of \eqref{eq:DFinv_multisteps} is given by
	\begin{align}
		\left\|DF(\bm{\ba})^{-1}(\cQ\bm{\psi}+\bm{p},\bm{\phi})\right\|_{\bm{X}}\le\bm{W_\ell}\|\bm{\phi}\|_{\bm{\omega,0}} + \bm{W_X}\|\bm{\psi}\|_{\bm{X}} + \bm{W_Y}\|\bm{p}\|_{\bm{X}}.
	\end{align}
\end{proof}

\begin{thm}[{\bf Local existence over multiple time steps}]\label{thm:local_inclusion_multisteps}
	Given the approximate solution $\bm{\ba}=\left(	\ba^{J_1},\ba^{J_2},\dots,\ba^{J_K}\right)\in \jp{\bm{\cD}}$, assume that
	\begin{align}
		\|(F_1(\ba^{J_1}))_1\|_{X_1}\le\delta^{J_1},~ \|(F_2(\ba^{J_1},\ba^{J_2}))_1\|_{X_2}\le\delta^{J_2},~\dots,~\|(F_K(\ba^{J_{K-1}},\ba^{J_K}))_1\|_{X_K}\le\delta^{J_K}
	\end{align}
	and
	\begin{align}
		\|(F_1(\ba^{J_1}))_2\|_{\omega}\le\varepsilon^{J_1},~ \|(F_2(\ba^{J_1},\ba^{J_2}))_2\|_{\omega}\le\varepsilon^{J_2},~\dots,~\|(F_K(\ba^{J_{K-1}},\ba^{J_K}))_2\|_{\omega}\le\varepsilon^{J_K}.
	\end{align}	
	Assume also that for $a_1^{J_i},a_2^{J_i}\in B_{J_i}\left(\ba^{J_i},\bm{\varrho}\right)$ $(i=1,2,\dots,K)$ there exists a non-decreasing function $L_{\ba^{J_i}}:(0,\infty)\to (0,\infty)$ such that
	\[
	\left\|\cN\left(a_1^{J_i}\right)-\mathcal{N}\left(a_2^{J_i}\right)-D\cN\left(\ba^{J_i}\right)\left(a_1^{J_i}-a_2^{J_i}\right)\right\|_{X_i}\le L_{\ba^{J_i}}(\bm{\varrho})\left\|a_1^{J_i}-a_2^{J_i}\right\|_{X_i}.
	\]
	Set $\bm{\delta}=\max\{\delta^{J_i}\}$, $\bm{\varepsilon}=\max\{\varepsilon^{J_i}\}$ and define
	\[
	p_{\bm{\varepsilon}}(\bm{\varrho})\bydef\bm{W_\ell}\,\bm{\varepsilon}+\bm{W_X}\max_{i}\left\{L_{\ba^{J_i}}(\bm{\varrho})\right\}\bm{\varrho}+\bm{W_Y}\bm{\delta},
	\]
	where positive constants $\bm{W_\ell}$, $\bm{W_X}$, and $\bm{W_Y}$ are defined in Lemma \ref{lem:DFinv_estimate}.
	If there exists $\bm{\varrho}_0>0$ such that
	\[
	p_{\bm{\varepsilon}}\left(\bm{\varrho}_0\right)\le\bm{\varrho}_0,
	\]
	then there exists a unique $\bm{\tilde{a}}\in\bm{B}_K(\bm{\ba}, \bm{\varrho}_0)$ satisfying $F(\bm{\tilde{a}})=0$ defined in \eqref{eq:mapF_multisteps}.
	Hence, the solution of the IVP \eqref{eq:IVP_PDE} exists locally over the multiple time steps.
\end{thm}

\begin{proof}
	We prove that the operator $T$ defined in \eqref{eq:opT_multisteps} is the contraction mapping on $\bm{B}_K(\bm{\ba}, \bm{\varrho}_0)$ in \eqref{eq:theBall_multisteps}.
	Firstly, for any $\bm{a}\in \bm{B}_K(\bm{\ba}, \bm{\varrho}_0)$, we have using \eqref{eq:DF_multisteps} and \eqref{eq:opT_multisteps}
	\begin{align}
		&T(\bm{a})-\bm{\ba}\\
		&= DF(\bm{\ba})^{-1}\left(DF(\bm{\ba})\bm{a}-F(\bm{a})\right) - \bm{\ba}\\
		&=DF(\bm{\ba})^{-1}\left(DF(\bm{\ba})(\bm{a}-\bm{\ba})-F(\bm{a})\right)\\
		&=DF(\bm{\ba})^{-1}\begin{bmatrix}
			\left(\cQ\left(\cN(a^{J_1})-\cN(\ba^{J_1})-D\cN(\ba^{J_1})(a^{J_1}-\ba^{J_1})\right)-(F_1(\ba^{J_1}))_1,~(F_1(\ba^{J_1}))_{2}\right)\\
			\left(\cQ\left(\cN(a^{J_2})-\cN(\ba^{J_2})-D\cN(\ba^{J_2})(a^{J_2}-\ba^{J_2})\right)-(F_2(\ba^{J_1},\ba^{J_2}))_1,~(F_2(\ba^{J_1},\ba^{J_2}))_{2}\right)\\
			\vdots\\
			\left(\cQ\left(\cN(a^{J_K})-\cN(\ba^{J_K})-D\cN(\ba^{J_K})(a^{J_K}-\ba^{J_K})\right)-(F_K(\ba^{J_{K-1}},\ba^{J_K}))_1,~(F_K(\ba^{J_{K-1}},\ba^{J_K}))_{2}\right)
		\end{bmatrix}.\label{eq:Uform_multisteps}
	\end{align}
	It follows using Lemma \ref{lem:DFinv_estimate} and the assumption of the theorem that
	\begin{align}
		\|T(\bm{a})-\bm{\ba}\|_{\bm{X}}
		&\le \bm{W_\ell}\max\left\{\left\|(F_1(\ba^{J_1}))_2\right\|_{\omega},\left\|(F_2(\ba^{J_1},\ba^{J_2}))_2\right\|_{\omega},\dots,\left\|(F_K(\ba^{J_{K-1}},\ba^{J_K}))_2\right\|_{\omega}\right\}\\
		&\hphantom{\le}\quad+\bm{W_X}\max_{i}\left\{\left\|\cN(a^{J_i})-\cN(\ba^{J_i})-D\cN(\ba^{J_i})(a^{J_i}-\ba^{J_i})\right\|_{X_i}\right\}\\
		&\hphantom{\le}\quad+\bm{W_Y}\max\left\{\|(F_1(\ba^{J_1}))_1\|_{X_1},\|(F_2(\ba^{J_1},\ba^{J_2}))_1\|_{X_2},\dots,\|(F_K(\ba^{J_{K-1}},\ba^{J_K}))_1\|_{X_K}\right\}\\
		&\le\bm{W_\ell}\,\bm{\varepsilon}+\bm{W_X}\max_{i}\left\{L_{\ba^{J_i}}(\bm{\varrho}_0)\right\}\bm{\varrho}_0+\bm{W_Y}\bm{\delta}=p_{\bm{\varepsilon}}(\bm{\varrho}_0).
	\end{align}
	From the assumption of the theorem $p_{\bm{\varepsilon}}(\bm{\varrho}_0)\le\bm{\varrho}_0$, then $T(\bm{a})\in \bm{B}_K(\bm{\ba}, \bm{\varrho}_0)$ holds.
	
	Secondary, we show the contraction property of $T$ on $\bm{B}_K(\bm{\ba}, \bm{\varrho}_0)$. From \eqref{eq:opT_multisteps} we have for $\bm{a}_1, \bm{a}_2\in \bm{B}_K(\bm{\ba}, \bm{\varrho}_0)$
	\begin{align}
		T(\bm{a}_1) - T(\bm{a}_2)
		&=DF(\bm{\ba})^{-1}\left[DF(\bm{\ba})(\bm{a}_1-\bm{a}_2)-\left(F(\bm{a}_1)-F(\bm{a}_2)\right)\right]\\
		&=DF(\bm{\ba})^{-1}\begin{bmatrix}
			\left(\cQ\left(\cN(a_1^{J_1})- \cN(a_2^{J_1})- D\cN(\ba^{J_1})(a_1^{J_1}-a_2^{J_1})\right),~0\right)\\
			\left(\cQ\left(\cN(a_1^{J_2})- \cN(a_2^{J_2})- D\cN(\ba^{J_2})(a_1^{J_2}-a_2^{J_2})\right),~0\right)\\
			\vdots\\
			\left( \cQ\left(\cN(a_1^{J_K})- \cN(a_2^{J_K})- D\cN(\ba^{J_K})(a_1^{J_K}-a_2^{J_K})\right),~0\right)\\
		\end{bmatrix}.
	\end{align}
	Using Lemma \ref{lem:DFinv_estimate} and the assumption of the theorem, we obtain
	\begin{align}
		\|T(\bm{a}_1)-T(\bm{a}_2)\|_{\bm{X}}&\le \bm{W_X}\max_{i}\left\{L_{\ba^{J_i}}(\bm{\varrho}_0)\right\}\|\bm{a}_1-\bm{a}_2\|_{\bm{X}}.
	\end{align}
	Hence, $\bm{W_X}\max_{i}\left\{L_{\ba^{J_i}}(\bm{\varrho}_0)\right\}<p_{\bm{\varepsilon}}(\bm{\varrho}_0)/\bm{\varrho}_0\le 1$ holds. Therefore, it is proved that the operator $T$ is the contraction mapping on $\bm{B}_K(\bm{\ba}, \bm{\varrho}_0)$.
\end{proof}

\subsection{Bounds on each time step}
The remaining bounds to complete the proof of  the local existence over the multiple time steps are positive constants $\bm{\jp{W^{(t_i,J_i)}}}$, $\bm{\jp{W^{(t_i,J_i)}_q}}$, and $\bm{\jp{W^{(t_i,t_{i-1})}}}$ such that
$\sup_{s\in J_i}\|U_{J_i}(t_i,s)\|_{B(\ell_{\omega}^1)}\le \bm{\jp{W^{(t_i,J_i)}}}$,
$\sup_{s\in [t_{i-1},t_i)}(t_i-s)^\gamma\|U_{J_i}(t_i,s)\cQ\|_{B(\ell_{\omega}^1)}\le \bm{\jp{W^{(t_i,J_i)}_q}}$,
and $\|U_{J_i}(t_i,t_{i-1})\|_{B(\ell_{\omega}^1)}\le \bm{\jp{W^{(t_i,t_{i-1})}}}$ ($i=1,2\dots,K$), respectively.

To get these bounds, we consider the linearized problem \eqref{LinearizedProb} and represent the solution of \eqref{LinearizedProb} as $b_s(t)\equiv b(t)$ ($(t,s)\in \jp{\cS_{J_i}}$).
Since we have a rigorous inclusion of the fundamental solution in Section \ref{sec:fundamental_sol}, one can obtain two bounds $\jp{W_{m,q}^{(t_i,J_i)}}>0$ and $ \jp{W_{m}^{(t_i,t_{i-1})}}>0$ from the definition of $\bU^{(\bbm)}(t,s)$ in \eqref{eqn:bUm} such that
\begin{align}\label{eq:WmJ1}
	\sup_{s\in J_i}\left\|\bU^{(\bbm)}(t_i,s)\cQ\right\|_{B(\ell_{\omega}^1)}=\sup_{s\in J_i}\left(\max_{\bj\in\Fm}\frac1{\omega_{\bj}}\sum_{\bk\in\Fm}|C_{\bk,\bj}(t_i,s)|(\text{\textbf{j}}\bL)^q\omega_{\bk}\right)\le \jp{W_{m,q}^{(t_i,J_i)}}
\end{align}
and 
\begin{align}\label{eq:Wmt1}
	\left\|\bU^{(\bbm)}(t_i,t_{i-1})\right\|_{B(\ell_{\omega}^1)}=\max_{\bj\in\Fm}\frac1{\omega_{\bj}}\sum_{\bk\in\Fm}|C_{\bk,\bj}(t_i,t_{i-1})|\omega_{\bk}\le  \jp{W_{m}^{(t_i,t_{i-1})}},
\end{align}
respectively.

Returning to consider $b_s(t)$ \jp{in Section \ref{sec:solution_map}}, we recall the bounds \eqref{eq:bm_bounds} and \eqref{eq:binf_bounds}. We define a matrix whose entries are the coefficients of the bounds as 
\begin{align}
	\jp{\bm{U}^{J_i}} \bydef \kappa^{-1}
	\begin{pmatrix}
		{\tau_i^\gamma}\jp{W_{m,q}^{\cS_{J_i}}} & \jp{W_{m,q}^{\cS_{J_i}}}\jp{\overline{W}_{\infty,q}^{\cS_{J_i}}} \jp{\cE_{m,\infty}^{J_i}} \\
		\jp{W_{m,q}^{\cS_{J_i}}}\jp{\overline{W}_{\infty,q}^{\cS_{J_i}}} \jp{\cE_{\infty,m}^{J_i}}  & \jp{W_{\infty,q}^{\cS_{J_i}}}
	\end{pmatrix}.
\end{align}
\jp{For any $s\in J_i$}, we have
\begin{align}
	\|b^{(\bbm)}_s\|_{\cX_i}\le \jp{\bm{U}_{11}^{J_i}}\|\psi^{(\bbm)}\|_{\omega} + \jp{\bm{U}_{12}^{J_i}}\|\psi^{(\infty)}\|_{\omega},\quad \|b^{(\infty)}_s\|_{\cX_i}\le \jp{\bm{U}_{21}^{J_i}}\|\psi^{(\bbm)}\|_{\omega} + \jp{\bm{U}_{22}^{J_i}}\|\psi^{(\infty)}\|_{\omega},
\end{align}
where $\jp{\bm{U}_{kj}^{J_i}}$ denotes the $(k,j)$ entry of $\jp{\bm{U}^{J_i}}$.
Similarly, we also define another matrix to obtain the $\bm{\jp{W^{\cS_{J_i}}}}$ bound in \eqref{eq:W_tau_bound} as
\begin{align}
	\jp{\tilde{\bm{U}}^{J_i}}\bydef\tilde{\kappa}^{-1}
	\begin{pmatrix}
		\jp{W_{m,0}^{\cS_{J_i}}} & \jp{W_{m,q}^{\cS_{J_i}}}\jp{\overline{W}_{\infty}^{\cS_{J_i}}} \jp{\cE_{m,\infty}^{J_i}} \\
		\jp{W_{m,0}^{\cS_{J_i}}}\jp{{\overline{W}}_{\infty,q}^{\prime\cS_{J_i}}} \jp{\cE_{\infty,m}^{J_i}}  & \jp{W_{\infty}^{\cS_{J_i}}}
	\end{pmatrix}.
\end{align}
We also represent the norm bound $b_s(t)$ as
\begin{align}
	\|b^{(\bbm)}_s\|\le \jp{\tilde{\bm{U}}_{11}^{J_i}}\|\psi^{(\bbm)}\|_{\omega} + \jp{\tilde{\bm{U}}_{12}^{J_i}}\|\psi^{(\infty)}\|_{\omega},\quad
	\|b^{(\infty)}_s\|\le \jp{\tilde{\bm{U}}_{21}^{J_i}}\|\psi^{(\bbm)}\|_{\omega} + \jp{\tilde{\bm{U}}_{22}^{J_i}}\|\psi^{(\infty)}\|_{\omega}
\end{align}
with the same manner.

Using these bounds, since the solution $b_s(t)$ of \eqref{LinearizedProb} satisfies the integral equations in \eqref{eq:b_form}, the finite part by the Fourier projection is bounded using the 
$\jp{W_{m,q}^{(t_i,J_i)}}$ bound in \eqref{eq:WmJ1}, \eqref{eq:ineqW_infty}, \eqref{eq:ineqbarW_infty}, \eqref{eq:Cm_bound}, and \eqref{eq:binf_omegaq} by
\begin{align}
	&\sup_{s\in J_i}(t_i-s)^\gamma\|b^{(\bbm)}_s(t_i)\|_{\omega}\\
	&\le \tau_i^\gamma \jp{W_{m,q}^{(t_i,J_i)}}\|\psi^{(\bbm)}\|_{\omega} + \jp{W_{m,q}^{(t_i,J_i)}}\jp{\cE_{m,\infty}^{J_i}} \sup_{s\in J_i}(t_i-s)^\gamma\int_{s}^{t_i}\|b^{(\infty)}_s(r)\|_{\omega}dr\\
	&\le \tau_i^\gamma \jp{W_{m,q}^{(t_i,J_i)}}\|\psi^{(\bbm)}\|_{\omega}\\
	&\hphantom{\le}\quad + \jp{W_{m,q}^{(t_i,J_i)}}\jp{\cE_{m,\infty}^{J_i}} \sup_{s\in J_i}(t_i-s)^\gamma\int_{s}^{t_i}\left(W_q^{(\infty)}(r,s)\|\psi^{(\infty)}\|_{\omega}+\jp{\cE_{\infty,m}^{J_i}} \int_s^rW_q^{(\infty)}(r,\sigma)\|b^{(\bm{m})}_s(\sigma)\|_{\omega} d\sigma\right)dr\\
	&\le \tau_i^\gamma \jp{W_{m,q}^{(t_i,J_i)}}\|\psi^{(\bbm)}\|_{\omega} + \jp{W_{m,q}^{(t_i,J_i)}}\jp{\cE_{m,\infty}^{J_i}}  \jp{\overline{W}_{\infty,q}^{\cS_{J_i}}}\|\psi^{(\infty)}\|_{\omega} + \jp{W_{m,q}^{(t_i,J_i)}}\jp{\doverline{W}_{\infty,q}^{\cS_{J_i}}} \jp{\cE_{m,\infty}^{J_i}}  \jp{\cE_{\infty,m}^{J_i}} \|b^{(\bbm)}_s\|_{\cX_i}\\
	&\le\left(\tau_i^\gamma \jp{W_{m,q}^{(t_i,J_i)}} + \jp{W_{m,q}^{(t_i,J_i)}}\jp{\doverline{W}_{\infty,q}^{\cS_{J_i}}} \jp{\cE_{m,\infty}^{J_i}}  \jp{\cE_{\infty,m}^{J_i}} \jp{\bm{U}_{11}^{J_i}}\right)\|\psi^{(\bbm)}\|_{\omega}\\
	&\hphantom{\le}\quad+\left(\jp{W_{m,q}^{(t_i,J_i)}}\jp{\cE_{m,\infty}^{J_i}}  \jp{\overline{W}_{\infty,q}^{\cS_{J_i}}} + \jp{W_{m,q}^{(t_i,J_i)}}\jp{\doverline{W}_{\infty,q}^{\cS_{J_i}}} \jp{\cE_{m,\infty}^{J_i}}  \jp{\cE_{\infty,m}^{J_i}} \jp{\bm{U}_{12}^{J_i}}\right)\|\psi^{(\infty)}\|_{\omega}.
\end{align}
On the other hand, the infinite part is bounded using \eqref{eq:ineqW_infty_sup}, \eqref{eq:ineqW_infty}, \eqref{eq:ineqbarW_infty}, \eqref{eq:Cinf_bound}, and \eqref{eq:bm_omegaq} by
\begin{align}
	&\sup_{s\in J_i}(t_i-s)^\gamma\|b^{(\infty)}_s(t_i)\|_{\omega}\\
	&\le \jp{W_{\infty,q}^{\cS_{J_i}}}\|\psi^{(\infty)}\|_{\omega} + \jp{\cE_{\infty,m}^{J_i}} \sup_{s\in J_i}(t_i-s)^\gamma\int_{s}^{t_i}W^{(\infty)}_q(t_i,r)\|b^{(\bbm)}_s(r)\|_{\omega}dr\\
	&\le \jp{W_{\infty,q}^{\cS_{J_i}}}\|\psi^{(\infty)}\|_{\omega}\\
	&\hphantom{\le}\quad + \jp{\cE_{\infty,m}^{J_i}} \sup_{s\in J_i}(t_i-s)^\gamma\int_{s}^{t_i}W^{(\infty)}_q(t_i,r)\left(\jp{W_{m,q}^{\cS_{J_i}}}\|\psi^{(\bbm)}\|_{\omega}+\jp{W_{m,q}^{\cS_{J_i}}}\jp{\cE_{m,\infty}^{J_i}} \int_s^r\|b^{(\infty)}_s(\sigma)\|_{\omega} d\sigma\right)dr\\
	&\le \jp{W_{\infty,q}^{\cS_{J_i}}}\|\psi^{(\infty)}\|_{\omega} + \jp{\cE_{\infty,m}^{J_i}}  \jp{\overline{W}_{\infty,q}^{\cS_{J_i}}}\jp{W_{m,q}^{\cS_{J_i}}}\|\psi^{(\bbm)}\|_{\omega} + \jp{\cE_{\infty,m}^{J_i}} \jp{\doverline{W}_{\infty,q}^{\cS_{J_i}}}\jp{W_{m,q}^{\cS_{J_i}}} \jp{\cE_{m,\infty}^{J_i}} 	\|b^{(\infty)}_s\|_{\cX_i}\\
	&\le\left(\jp{\cE_{\infty,m}^{J_i}}  \jp{\overline{W}_{\infty,q}^{\cS_{J_i}}}\jp{W_{m,q}^{\cS_{J_i}}} + \jp{W_{m,q}^{\cS_{J_i}}}\jp{\doverline{W}_{\infty,q}^{\cS_{J_i}}}\jp{\cE_{m,\infty}^{J_i}}  \jp{\cE_{\infty,m}^{J_i}} 	\jp{\bm{U}_{21}^{J_i}}\right)\|\psi^{(\bbm)}\|_{\omega}\\
	&\hphantom{\le}\quad+\left(\jp{W_{\infty,q}^{\cS_{J_i}}}+ \jp{W_{m,q}^{\cS_{J_i}}}\jp{\doverline{W}_{\infty,q}^{\cS_{J_i}}}\jp{\cE_{m,\infty}^{J_i}}  \jp{\cE_{\infty,m}^{J_i}} \jp{\bm{U}_{22}^{J_i}}\right)\|\psi^{(\infty)}\|_{\omega}.
\end{align}
Consequently, the $\bm{\jp{W^{(t_i,J_i)}_q}}$ bound is given by
\begin{align}\label{eq:WhJ1_bound}
	&\bm{\jp{W^{(t_i,J_i)}_q}}\\
	&\bydef
	\left\|\begin{pmatrix}
		\tau_i^\gamma \jp{W_{m,q}^{(t_i,J_i)}} + \jp{W_{m,q}^{(t_i,J_i)}}\jp{\doverline{W}_{\infty,q}^{\cS_{J_i}}} \jp{\cE_{m,\infty}^{J_i}}  \jp{\cE_{\infty,m}^{J_i}} \jp{\bm{U}_{11}^{J_i}}
		& \jp{W_{m,q}^{(t_i,J_i)}}\jp{\cE_{m,\infty}^{J_i}}  \jp{\overline{W}_{\infty,q}^{\cS_{J_i}}} + \jp{W_{m,q}^{(t_i,J_i)}}\jp{\doverline{W}_{\infty,q}^{\cS_{J_i}}} \jp{\cE_{m,\infty}^{J_i}}  \jp{\cE_{\infty,m}^{J_i}} \jp{\bm{U}_{12}^{J_i}}\\
		\jp{\cE_{\infty,m}^{J_i}}  \jp{\overline{W}_{\infty,q}^{\cS_{J_i}}}\jp{W_{m,q}^{\cS_{J_i}}} + \jp{W_{m,q}^{\cS_{J_i}}}\jp{\doverline{W}_{\infty,q}^{\cS_{J_i}}}\jp{\cE_{m,\infty}^{J_i}}  \jp{\cE_{\infty,m}^{J_i}} 	\jp{\bm{U}_{21}^{J_i}}
		& \jp{W_{\infty,q}^{\cS_{J_i}}}+ \jp{W_{m,q}^{\cS_{J_i}}}\jp{\doverline{W}_{\infty,q}^{\cS_{J_i}}}\jp{\cE_{m,\infty}^{J_i}}  \jp{\cE_{\infty,m}^{J_i}} \jp{\bm{U}_{22}^{J_i}}
	\end{pmatrix}\right\|_1.
\end{align}
Moreover, an analogous discussion using the bounds in Corollary \ref{cor:solutionmap} directly yields the $\bm{\jp{W^{(t_i,J_i)}}}$ bound given by
\begin{align}\label{eq:WhJ1_bound_classic}
	&\bm{\jp{W^{(t_i,J_i)}}}\\
	&\bydef \left\|\begin{pmatrix}
		\jp{W_{m,0}^{(t_i,J_i)}} + \jp{W_{m,q}^{(t_i,J_i)}}\jp{{\doverline{W}}_{\infty,q}^{\prime\cS_{J_i}}}	 \jp{\cE_{m,\infty}^{J_i}}  \jp{\cE_{\infty,m}^{J_i}} \jp{\tilde{\bm{U}}_{11}^{J_i}}
		& \jp{W_{m,q}^{(t_i,J_i)}}\jp{\cE_{m,\infty}^{J_i}}  \jp{\overline{W}_{\infty}^{\cS_{J_i}}} + \jp{W_{m,q}^{(t_i,J_i)}}\jp{{\doverline{W}}_{\infty,q}^{\prime\cS_{J_i}}} \jp{\cE_{m,\infty}^{J_i}}  \jp{\cE_{\infty,m}^{J_i}} \jp{\tilde{\bm{U}}_{12}^{J_i}}\\
		\jp{\cE_{\infty,m}^{J_i}}  \jp{{\overline{W}}_{\infty,q}^{\prime\cS_{J_i}}}\jp{W_{m,0}^{\cS_{J_i}}} + \jp{W_{m,q}^{\cS_{J_i}}}\jp{{\doverline{W}}_{\infty,q}^{\prime\cS_{J_i}}}\jp{\cE_{m,\infty}^{J_i}}  \jp{\cE_{\infty,m}^{J_i}} 	\jp{\tilde{\bm{U}}_{21}^{J_i}}
		& \jp{W_{\infty}^{\cS_{J_i}}}+ \jp{W_{m,q}^{\cS_{J_i}}}\jp{{\doverline{W}}_{\infty,q}^{\prime\cS_{J_i}}}\jp{\cE_{m,\infty}^{J_i}}  \jp{\cE_{\infty,m}^{J_i}} \jp{\tilde{\bm{U}}_{22}^{J_i}}
	\end{pmatrix}\right\|_1.
\end{align}

Next, setting $s=t_{i-1}$ in \eqref{eq:b_form} and the initial data $\phi\in \ell_{\omega}^1$ instead of $\cQ\psi$, the finite part $b^{(\bbm)}_{t_{i-1}}(t_i)$ is bounded using the  bounds $ \jp{W_{m}^{(t_i,t_{i-1})}}$ and $\jp{W_{m,q}^{(t_i,J_i)}}$, given in \eqref{eq:Wmt1} and \eqref{eq:WmJ1} respectively, \eqref{eq:Cm_bound}, \eqref{eq:binf_omegaq}, \eqref{eq:ineqW_infty_classic}, and \eqref{eq:ineqbarW_infty_classic} by
\begin{align}
	&\|b^{(\bbm)}_{t_{i-1}}(t_i)\|_{\omega}\\
	&\le  \jp{W_{m}^{(t_i,t_{i-1})}}\|\phi^{(\bbm)}\|_{\omega} + \jp{W_{m,q}^{(t_i,J_i)}}\jp{\cE_{m,\infty}^{J_i}} \int_{J_i}\left\|b^{(\infty)}_{t_{i-1}}(r)\right\|_{\omega}dr\\
	&\le  \jp{W_{m}^{(t_i,t_{i-1})}}\|\phi^{(\bbm)}\|_{\omega} + \jp{W_{m,q}^{(t_i,J_i)}}\jp{\cE_{m,\infty}^{J_i}} \int_{J_i}\left(W^{(\infty)}(r,t_{i-1})\|\phi^{(\infty)}\|_{\omega}+\jp{\cE_{\infty,m}^{J_i}} \int_{t_{i-1}}^rW^{(\infty)}_q(r,\sigma)\|b^{(\bm{m})}_{t_{i-1}}(\sigma)\|_{\omega} d\sigma\right)dr\\
	&\le  \jp{W_{m}^{(t_i,t_{i-1})}}\|\phi^{(\bbm)}\|_{\omega} + \jp{W_{m,q}^{(t_i,J_i)}}\jp{\cE_{m,\infty}^{J_i}}  \jp{\overline{W}_{\infty}^{\cS_{J_i}}}\|\phi^{(\infty)}\|_{\omega} + \jp{W_{m,q}^{(t_i,J_i)}}\jp{{\doverline{W}}_{\infty,q}^{\prime\cS_{J_i}}} \jp{\cE_{m,\infty}^{J_i}}  \jp{\cE_{\infty,m}^{J_i}} \|b^{(\bbm)}_{t_{i-1}}\|\\
	&\le\left( \jp{W_{m}^{(t_i,t_{i-1})}} + \jp{W_{m,q}^{(t_i,J_i)}}\jp{{\doverline{W}}_{\infty,q}^{\prime\cS_{J_i}}} \jp{\cE_{m,\infty}^{J_i}}  \jp{\cE_{\infty,m}^{J_i}} \jp{\tilde{\bm{U}}_{11}^{J_i}}\right)\|\phi^{(\bbm)}\|_{\omega}\\
	&\hphantom{\le}\quad+\left(\jp{W_{m,q}^{(t_i,J_i)}}\jp{\cE_{m,\infty}^{J_i}}  \jp{\overline{W}_{\infty}^{\cS_{J_i}}} + \jp{W_{m,q}^{(t_i,J_i)}}\jp{{\doverline{W}}_{\infty,q}^{\prime\cS_{J_i}}} \jp{\cE_{m,\infty}^{J_i}}  \jp{\cE_{\infty,m}^{J_i}} \jp{\tilde{\bm{U}}_{12}^{J_i}}\right)\|\phi^{(\infty)}\|_{\omega}.
\end{align}
The $\ell_{\omega}^1$ norm bound of the infinite part $b^{(\infty)}_{t_{i-1}}(t_i)$ is derived using $W^{(\infty)}(t,s)$ in \eqref{eq:Winfts_bound_classic}, \eqref{eq:Cinf_bound}, \eqref{eq:bm_omegaq}, \eqref{eq:ineqhatW_infty_classic}, and \eqref{eq:ineqbarW_infty_classic} by
\begin{align}
	&\|b^{(\infty)}_{t_{i-1}}(t_i)\|_{\omega}\\
	&\le W^{(\infty)}(t_i,t_{i-1})\|\phi^{(\infty)}\|_{\omega} + \jp{\cE_{\infty,m}^{J_i}} \int_{J_i}W^{(\infty)}_q(t_i,r)\|b^{(\bbm)}_{t_{i-1}}(r)\|_{\omega}dr\\
	&\le W^{(\infty)}(t_i,t_{i-1})\|\phi^{(\infty)}\|_{\omega} + \jp{\cE_{\infty,m}^{J_i}} \int_{J_i}W^{(\infty)}_q(t_i,r)\left(\jp{W_{m,0}^{\cS_{J_i}}}\|\phi^{(\bbm)}\|_{\omega}+\jp{W_{m,q}^{\cS_{J_i}}}\jp{\cE_{m,\infty}^{J_i}} \int_{t_{i-1}}^r\|b^{(\infty)}_{t_{i-1}}(\sigma)\|_{\omega} d\sigma\right)dr\\
	&\le W^{(\infty)}(t_i,t_{i-1})\|\phi^{(\infty)}\|_{\omega} + \jp{\cE_{\infty,m}^{J_i}}  \jp{{\overline{W}}_{\infty,q}^{\prime\cS_{J_i}}} W_{\bbm,0}\|\phi^{(\bbm)}\|_{\omega} + \jp{\cE_{\infty,m}^{J_i}} \jp{{\doverline{W}}_{\infty,q}^{\prime\cS_{J_i}}}\jp{W_{m,q}^{\cS_{J_i}}} \jp{\cE_{m,\infty}^{J_i}} 	
	\|b^{(\infty)}_s\|\\
	&\le\left(\jp{\cE_{\infty,m}^{J_i}}  \jp{{\overline{W}}_{\infty,q}^{\prime\cS_{J_i}}}\jp{W_{m,0}^{\cS_{J_i}}} + \jp{W_{m,q}^{\cS_{J_i}}}\jp{{\doverline{W}}_{\infty,q}^{\prime\cS_{J_i}}}\jp{\cE_{m,\infty}^{J_i}}  \jp{\cE_{\infty,m}^{J_i}} 	\jp{\tilde{\bm{U}}_{21}^{J_i}}\right)\|\phi^{(\bbm)}\|_{\omega}\\
	&\hphantom{\le}\quad+\left(W^{(\infty)}(t_i,t_{i-1})+ \jp{W_{m,q}^{\cS_{J_i}}}\jp{{\doverline{W}}_{\infty,q}^{\prime\cS_{J_i}}}\jp{\cE_{m,\infty}^{J_i}}  \jp{\cE_{\infty,m}^{J_i}} \jp{\tilde{\bm{U}}_{22}^{J_i}}\right)\|\psi^{(\infty)}\|_{\omega}.
\end{align}
The $\bm{\jp{W^{(t_i,t_{i-1})}}}>0$ bound satisfying $\|U_{J_i}(t_i,t_{i-1})\|_{B(\ell_{\omega}^1)}\le \bm{\jp{W^{(t_i,t_{i-1})}}}$ is finally given by
\begin{align}\label{eq:Wht1_bound}
	&\bm{\jp{W^{(t_i,t_{i-1})}}}\\
	&= \left\|\begin{bmatrix}
		 \jp{W_{m}^{(t_i,t_{i-1})}} + \jp{W_{m,q}^{(t_i,J_i)}}\jp{{\doverline{W}}_{\infty,q}^{\prime\cS_{J_i}}} \jp{\cE_{m,\infty}^{J_i}}  \jp{\cE_{\infty,m}^{J_i}} \jp{\tilde{\bm{U}}_{11}^{J_i}} & \jp{W_{m,q}^{(t_i,J_i)}}\jp{\cE_{m,\infty}^{J_i}}  \jp{\overline{W}_{\infty}^{\cS_{J_i}}} + \jp{W_{m,q}^{(t_i,J_i)}}\jp{{\doverline{W}}_{\infty,q}^{\prime\cS_{J_i}}} \jp{\cE_{m,\infty}^{J_i}}  \jp{\cE_{\infty,m}^{J_i}} \jp{\tilde{\bm{U}}_{12}^{J_i}}\\
		\jp{\cE_{\infty,m}^{J_i}}  \jp{{\overline{W}}_{\infty,q}^{\prime\cS_{J_i}}}\jp{W_{m,0}^{\cS_{J_i}}} + \jp{W_{m,q}^{\cS_{J_i}}}\jp{{\doverline{W}}_{\infty,q}^{\prime\cS_{J_i}}}\jp{\cE_{m,\infty}^{J_i}}  \jp{\cE_{\infty,m}^{J_i}} 	\jp{\tilde{\bm{U}}_{21}^{J_i}} & W^{(\infty)}(t_i,t_{i-1})+ \jp{W_{m,q}^{\cS_{J_i}}}\jp{{\doverline{W}}_{\infty,q}^{\prime\cS_{J_i}}}\jp{\cE_{m,\infty}^{J_i}}  \jp{\cE_{\infty,m}^{J_i}} \jp{\tilde{\bm{U}}_{22}^{J_i}}
	\end{bmatrix}\right\|_1.
\end{align}

\subsection{Error bounds at the end point}\label{sec:err_endpoint}

When the local existence of the solution is proved, we obtain error bounds at the end point of $J_i$ as follows: Firstly, setting $t=t_1$ in \eqref{eq:U_form}, it follows from Lemma \ref{lem:DFinv_bounds}
\begin{align}
	\left\|a^{J_1}(t_1) -\ba^{J_1}(t_1) \right\|_{\omega} 
	&\le \bm{\jp{W^{(t_1,t_0)}}}\varepsilon^{J_1} + \hat{\tau}_1\bm{\jp{W^{(t_1,J_1)}_q}}L_{\ba^{J_1}}(\bm{\varrho}_0)\bm{\varrho}_0 + \tau_1\bm{\jp{W^{(t_1,J_1)}}}\delta^{J_1},\label{eq:error_at_end point}
\end{align}
where $\bm{\varrho}_0>0$ is the validated radius in Theorem \ref{thm:local_inclusion_multisteps} of the ball defined in \eqref{eq:theBall_multisteps}.
Secondly, setting $t=t_2$ in \eqref{eq:Uform_multisteps}, the exact form of the inverse \eqref{eq:invDF_J2} yields
\begin{align}
	\left\|a^{J_2}(t_2) -\ba^{J_2}(t_2) \right\|_{\omega} 
	&\le \bm{\jp{W^{(t_2,t_1)}}}\varepsilon^{J_2} + \hat{\tau}_2\bm{\jp{W^{(t_2,J_2)}_q}}L_{\ba^{J_2}}(\bm{\varrho}_0)\bm{\varrho}_0 + \tau_2\bm{\jp{W^{(t_2,J_2)}}}\delta^{J_2}\\
	&\hphantom{\le}\quad +  \bm{\jp{W^{(t_2,t_0)}}}\varepsilon^{J_1} + \hat{\tau}_1\bm{\jp{W^{(t_2,t_1)}}}\bm{\jp{W^{(t_1,J_1)}_q}}L_{\ba^{J_1}}(\bm{\varrho}_0)\bm{\varrho}_0 + \tau_1\bm{\jp{W^{(t_2,t_1)}}}\bm{\jp{W^{(t_1,J_1)}}} \delta^{J_1}.
\end{align}
Similarly, we have an error bound at $t=t_K$ from \eqref{eq:invDF_Jk} and \eqref{eq:Uform_multisteps}
\begin{align}
	&\left\|a^{J_K}(t_K) -\ba^{J_K}(t_K) \right\|_{\omega}\\
	&\le \bm{\jp{W^{(t_K,t_0)}}}\varepsilon^{J_1}+\bm{\jp{W^{(t_K,t_1)}}}\varepsilon^{J_2}+\dots+\bm{\jp{W^{(t_K,t_{K-1})}}}\varepsilon^{J_K}\\
	&+\hat{\tau}_1\bm{\jp{W^{(t_K,t_1)}}}\bm{\jp{W^{(t_1,J_1)}_q}}L_{\ba^{J_1}}(\bm{\varrho}_0)\bm{\varrho}_0 + \hat{\tau}_2\bm{\jp{W^{(t_K,t_2)}}}\bm{\jp{W^{(t_2,J_2)}_q}}L_{\ba^{J_2}}(\bm{\varrho}_0)\bm{\varrho}_0 + \dots + \hat{\tau}_K\bm{\jp{W^{(t_K,J_K)}_q}}L_{\ba^{J_K}}(\bm{\varrho}_0)\bm{\varrho}_0\\
	&+{\tau}_1\bm{\jp{W^{(t_K,t_1)}}}\bm{\jp{W^{(t_1,J_1)}}}\delta^{J_1} + {\tau}_2\bm{\jp{W^{(t_K,t_2)}}}\bm{\jp{W^{(t_2,J_2)}}}\delta^{J_2}+\tau_K\bm{\jp{W^{(t_K,J_K)}}}\delta^{J_K}.\label{eq:error_at_end point_K}
\end{align}

	Finally, the $\varepsilon^{J_i}$ bound satisfying
	\begin{align}
		\left\|\ba^{J_i}(t_{i-1}) - \ba^{J_{i-1}}(t_{i-1})\right\|_{\omega}\le \varepsilon^{J_i}\quad (i=2,3,\dots,K)
	\end{align}
	is a numerical error between two different approximate solutions, which is caused by the Chebyshev interpolation. This is tiny error in general and easily computed by the form
	\begin{align}
		\left\|\ba^{J_i}(t_{i-1}) - \ba^{J_{i-1}}(t_{i-1})\right\|_{\omega}=\sum_{\bk\in\FN}\left|\ba_{0,\bk}^{J_i}+2\sum_{\ell=1}^{n_i-1}(-1)^{\ell}\ba_{\ell,\bk}^{J_i}-\left(\ba_{0,\bk}^{J_{i-1}}+2\sum_{\ell=1}^{n_{i-1}-1}\ba_{\ell,\bk}^{J_{i-1}}\right)\right|\omega_{\bk},
	\end{align}
	where $n_i$ denotes the size of Chebyshev coefficients on $J_i$.
	


\section{Application to the Swift-Hohenberg equation} \label{sec:global_existence}

In the following sections, we demonstrate some applications of our rigorous integrator. Firstly, in this section, we provide a computational approach of proving the existence of global in time solutions. We then apply the provided method for computer-assisted proofs of global existence of solutions to initial value problems of the 3D/2D Swift-Hohenberg equation. We combine the rigorous forward integration with explicit constructions of trapping regions to prove global existence converging to asymptotically stable nontrivial equilibria. 
Secondly, in the next section, the 2D Ohta-Kawasaki equation is considered to deal with derivatives of the nonlinear term.
All computations in this study were conducted on Windows 10, Intel(R) Core(TM) i9-10900K CPU @ 3.70GHz using MATLAB R2022a with INTLAB - INTerval LABoratory \cite{Ru99a} version 11, which supports interval arithmetic.
The code used to generate the results presented in the following sections is freely available from {\cite{bib:codes}}.

\subsection{Constructing trapping region}\label{sec:Trapping_region}
We propose a strategy to demonstrate global existence of \eqref{eq:general_PDE} via the mechanism of convergence towards an asymptotically stable equilibrium solution. We are specifically looking at cases where the nonlinearity of \eqref{eq:general_PDE} is of the form \eqref{eq:assumption_on_N} with $q = 0$.
We denote the vector field of the infinite-dimensional system of ODEs \eqref{eq:the_ODEs} by
\begin{equation} \label{eq:ODE_trapping_region}
\dot{a}(t) = f(a) \bydef  \cL a(t) + \cN(a).
\end{equation}
%
Assume the existence of an equilibrium solution $\ta \in \ell_{\omega}^1$, that is 
\begin{equation} \label{eq:steady_state_trapping_region}
f(\ta) = \cL \ta + \cN(\ta) = 0,
\end{equation}
and assume that the Fr\'{e}chet derivative $Df(\ta)$ is an invertible linear operator \jp{with real eigenvalues}. 
%
%
In addition, let $\cG:\ell_{\omega}^1 \to \ell_{\omega}^1$ be defined by 
\begin{align} \label{eq:definition_of_g_trapping_region}
\cG(h) \bydef f(\ta + h) - Df(\ta)h, \quad \text{for } h \in \ell_{\omega}^1,
\end{align}
so that $\cG(0)=f(\ta)=0$ and $D\cG(0) = Df(\ta) - Df(\ta)=0$ hold.

The following theorem asserts existence of the trapping region associated with the asymptotically stable equilibrium solution. Note that a similar argument was consider in \cite{MR3906120} to construct an attracting neighborhood in Fisher's equation, which was combined with a computer-assisted proof to prove global existence. 

\begin{thm}[\bf Global existence via asymptotic convergence] \label{thm:trapping_region}
Let $\ta \in \ell_{\omega}^1$ be an equilibrium solution of \eqref{eq:ODE_trapping_region}.
Assume that there exist $C\ge 1$ and $\lambda>0$ such that 
\begin{equation} \label{eq:exponential_operator_norm}
\|e^{Df(\ta)t}\|_{B(\ell_{\omega}^1)} \le C e^{-\lambda t}
\end{equation}
for all $t>0$.
Consider $\epsilon>0$ small enough so that 
\begin{equation} \label{eq:delta_trapping_region}
\delta \bydef \lambda - \epsilon > 0. 
\end{equation}
Since $\cG(0)=0$ and $D\cG(0)=0$, pick $\rho>0$ (as large as possible) such that 
\begin{equation} \label{eq:assumption_on_g_trapping_region}
\| \cG(\phi)\|_\omega \le \frac{\delta}{C} \| \phi \|_\omega, \quad\forall\phi \in \ell_{\omega}^1 \text{ with } \|\phi\|_\omega \le \rho.
\end{equation}
If
\begin{equation} \label{eq:assumption_for_trapping_region}
\|a(0) - \ta\|_\omega \le \frac{\rho}{C},
\end{equation}
then for all $t > 0$
\begin{equation} \label{eq:explicit_definition_trapping_region}
\| a(t) -\ta \|_{\omega} \le C \|a(0) - \ta\|_{\omega} e^{-\epsilon t} \le \rho e^{-\epsilon t}.
\end{equation}
More explicitly, if $a(0) \in B_{\frac{\rho}{C}}(\ta) \subset \ell_{\omega}^1$, then the solution $a(t)$ is bounded (it stays in $B_{\frac{\rho}{C}}(\ta)$), it exists globally in time and it satisfies
\[
\lim_{t \to \infty} a(t) = \ta.
\] 
\end{thm}

\begin{proof}
While the proof is standard in the theory of differential equations, we present it here since it provides a computational procedure to show global existence and asymptotic convergence.  

Consider $t>0$ and let $h(t) \bydef a(t) - \ta$. From \eqref{eq:definition_of_g_trapping_region}, one gets that
\begin{equation} \label{eq:linearization_about_ta}
\dot h(t) = \dot a(t) = f(a(t)) = f(h(t)+\ta) = f(\ta) + Df(\ta)h(t) + \cG(h(t)) = Df(\ta)h(t) + \cG(h(t)), 
\end{equation}
with $\cG(0)=0$ and $D\cG(0)=0$. The variation of constants formula in $h \in \ell_{\omega}^1$ for \eqref{eq:linearization_about_ta} yields  
\[
h(t) = e^{Df(\ta)t} h(0) + \int_0^t e^{Df(\ta)(t-s)} \cG(h(s)) ds,
\]
and therefore
\begin{equation} \label{eq:trapping_region_inequality_1}
\| h(t)\|_{\omega} \le \|e^{Df(\ta)t}\|_{B(\ell_{\omega}^1)} \|h(0)\|_{\omega} + \int_0^t \|e^{Df(\ta)(t-s)}\|_{B(\ell_{\omega}^1)} \| \cG(h(s))\|_{\omega} ds.
\end{equation}
Combining \eqref{eq:exponential_operator_norm} and \eqref{eq:trapping_region_inequality_1}, one obtains
\begin{equation} \label{eq:trapping_region_inequality_2}
\| h(t)\|_{\omega} \le  C e^{-\lambda t} \|h(0)\|_{\omega} + \int_0^t  C e^{-\lambda (t-s)} \| \cG(h(s))\|_{\omega} ds.
\end{equation}
Now, as long as $\|h(s)\|_{\omega} \le \rho$ for $s \in [0,t]$, then combining assumption \eqref{eq:assumption_on_g_trapping_region} and \eqref{eq:trapping_region_inequality_2} yields
\begin{align*} 
\| h(t)\|_{\omega} &\le  C e^{-\lambda t} \|h(0)\|_{\omega} + \int_0^t  C e^{-\lambda (t-s)} \| \cG(h(s))\|_{\omega} ds \\
&\le C e^{-\lambda t} \|h(0)\|_{\omega} + \int_0^t  \delta e^{-\lambda (t-s)} \| h(s) \|_{\omega} ds.
\end{align*}
Letting $y(t) \bydef e^{\lambda t}  \| h(t)\|_{\omega} \ge 0$, one obtains from the last inequality that
\begin{equation} \label{eq:trapping_region_inequality_3}
y(t) \le C \|h(0)\|_{\omega} + \int_0^t  \delta y(s) ds.
\end{equation}
Applying Gronwall's inequality to \eqref{eq:trapping_region_inequality_3}
\[
y(t) = e^{\lambda t}  \| h(t)\|_{\omega} \le C \|h(0)\|_{\omega} e^{\int_0^t \delta ds} = C \|h(0)\|_{\omega} e^{\delta t}
\]
which implies that
\begin{equation} \label{eq:trapping_region_inequality_4}
\| h(t)\|_{\omega} \le C \|h(0)\|_{\omega} e^{(\delta -\lambda)t} 
= C \|h(0)\|_{\omega} e^{-\epsilon t}.
\end{equation}
Observe that if $\|h(0)\|_{\omega} \le \rho/C$, then \eqref{eq:trapping_region_inequality_4} guarantees that $\|h(t)\|_{\omega} \le \rho e^{-\epsilon t} < \rho$ for all $t > 0$. Going back to the original coordinates $a(t) = h(t) + \ta$, then we conclude that 
%
\[
\|a(0) - \ta \|_{\omega} \le \rho/C \Longrightarrow \|a(t) - \ta \|_{\omega} \le \rho e^{-\epsilon t} < \rho \quad \text{for all } t>0. \qedhere
\]
%
\end{proof}

Next, we provide an explicit procedure which will be used to conclude, via a successful application of Theorem~\ref{thm:trapping_region}, that a solution exists globally in time and that it converges to an asymptotically stable equilibrium solution.

\begin{proc} \label{proc:global_existence}
To study the asymptotic behaviour of solutions close to an  asymptotically stable equilibrium solution, perform the following steps rigorously:
\begin{itemize}
\item[\em(P1)]Compute $\ta \in \ell_{\omega}^1$ such that $f(\ta)=0$; 
\vspace{-.1cm}
\item[\em(P2)] Define $\cG$ from \eqref{eq:definition_of_g_trapping_region}; 
\vspace{-.1cm}
\item[\em(P3)] Compute $C>0$ and $\lambda>0$ such that \eqref{eq:exponential_operator_norm} holds; 
\vspace{-.1cm}
\item[\em(P4)] Consider $\epsilon>0$ small enough (in fact the smaller the better) so that 
\eqref{eq:delta_trapping_region} holds; 
\vspace{-.1cm}
\item[\em(P5)] Let $\delta \bydef \lambda - \epsilon > 0$;
\vspace{-.1cm}
\item[\em(P6)] Compute $\rho=\rho(\delta,C)>0$ (as large as possible) such that 
\eqref{eq:assumption_on_g_trapping_region} holds; 
\vspace{-.1cm}
\item[\em(P7)] As we integrate equation \eqref{eq:ODE_trapping_region}, verify that $\|a(\tau) - \ta\|_{\omega} \le \frac{\rho}{C}$, for some $\tau \ge 0$; 
\vspace{-.1cm}
\item[\em(P8)] From Theorem~\ref{thm:trapping_region}, conclude that $a(t)$ stays in $B_{\rho}(\ta)$ for all $t \ge \tau$ (it is bounded), it exists globally in time and it satisfies $a(t) \to \ta$ as $t \to +\infty$. 
\end{itemize}
\end{proc}

If Procedure~\ref{proc:global_existence} is successfully applied, then as we perform the rigorous integration, we verify if $a(\tau) \in B_{\frac{\rho}{C}}(\ta)$. If so, then the solution is trapped in the ball $B_{\rho}(\ta)$ for all time $t \ge \tau$. For this reason, we say that the set $B_{\rho}(\ta)$ is a {\em trapping region}.


\subsection{Semigroup estimates}\label{sec:Semigroup_es}
We will show how one get the constants $C\ge 1$ and $\lambda>0$ satisfying
\begin{align}
	\left\|e^{Df(\tilde{a})t}\phi\right\|_{\omega}\le Ce^{-\lambda t}\|\phi\|_{\omega,0},\quad \mbox{for all}~t>0,~\phi\in \ell_{\omega}^1.
\end{align}
This corresponds to (P3) in Procedure \ref{proc:global_existence}.
Since the equilibrium solution is asymptotically stable, such $\lambda>0$ is expected to always exit. Before presenting the computation of the constants, we first need to construct an operator $P$, for which the column will be an approximation of the eigenvectors of $Df(\tilde{a})$. Let $\bar{\lambda}_{\bk}$ ($\bk\in\FM$) be approximate eigenvalues of the truncated matrix of $Df(\tilde{a})$ with the size $\bM=(M_1,\dots,M_d)$\footnote{The truncation size $\bM$ is determined appropriately later.}. These eigenvalues can be sorted in increasing order with respect to $\bk$. 
To present a decomposition of an infinite dimensional operator, say  $\mathcal{A}:\ell_{\omega}^1\to\ell_{\omega}^1$ for example, into finite and tail parts, we denote a finite truncation of the size $\bM$ by $\mathcal{A}_1:\ell_{\omega}^1\to\C^{N}$, where $N=\prod_{i=1}^dM_i$, and the tail part by $\mathcal{A}_\infty:\ell_{\omega}^1\to\ell_{\infty}^1=(\mathrm{Id}-\varPi^{(\bM)})\ell_{\omega}^1$, respectively.
Let $P_1$ be an invertible matrix whose columns are numerical approximation of the eigenvectors corresponding to $\bar{\lambda}_{\bk}$. Then we set an operator $P$ and its inverse $P^{-1}$ as
\begin{align}
	P=\begin{bmatrix}
		P_1 & 0\\ 0 & \mathrm{Id}_\infty
	\end{bmatrix},\quad
	P^{-1}=\begin{bmatrix}
		P_1^{-1} & 0\\ 0 & \mathrm{Id}_\infty
	\end{bmatrix}.
\end{align}
Therefore, using these operators, we define the operator $M\bydef P^{-1}Df(\tilde{a})P$. 
Since $P$ is an approximation of the eigenvectors of $Df(\tilde{a})$, and $P^{-1}$ is the exact inverse of $P$, we can see $M$ as being an approximate standard diagonalization of $Df(\tilde{a})$. 
For this reason we will refer to $M$ as the \emph{pseudo-diagonalization} of $Df(\tilde{a})$. The motivation for the pseudo-diagonalization is presented later in this section. We can rewrite $M$ with the form
\begin{align}\label{eq:operators}
	M=\Lambda+\mathscr{E},\quad
	\Lambda=\begin{bmatrix}
		\Lambda_1 & 0\\0 & \Lambda_\infty
	\end{bmatrix},\quad
	\mathscr{E}=\begin{bmatrix}
		\mathscr{E}_1^1 & \mathscr{E}_1^\infty\\
		\mathscr{E}_\infty^1 & \mathscr{E}_\infty^\infty
	\end{bmatrix},
\end{align}
where $\Lambda_1=\mathrm{diag}(\bar{\lambda}_{\bk})$ ($\bk\in\FM$) and $\Lambda_\infty=\mathrm{diag}(\mu_{\bk})$ ($\bk\notin\FM$) are the finite/infinite diagonal matrices, $\mathscr{E}_1^1\in\C^{N\times N}$ contains numerical errors of diagonalization, $\mathscr{E}_1^\infty:\ell_{\infty}^1\to \C^{N}$, $\mathscr{E}_\infty^1:\C^{N}\to \ell_{\infty}^1$, and $\mathscr{E}_\infty^\infty:\ell_{\infty}^1\to \ell_{\infty}^1$.
More precisely,
\begin{align}
	\mathscr{E}_1^\infty=P_1^{-1}\cQ D\cN_1(\tilde{a}),\quad\mathscr{E}_\infty^1=\cQ D\cN_\infty(\tilde{a})P_1,\quad\mathscr{E}_\infty^\infty=\cQ D\cN_\infty(\tilde{a}).
\end{align}
Using this notation, we get the bound
\begin{equation} \label{eq:bound_P}
\| e^{Df(\tilde{a}) t} \|_{B(\ell_\omega^1)} = \| e^{PMP^{-1} t} \|_{B(\ell_\omega^1)}  \leq  \| P  \|_{B(\ell_\omega^1)}  \| P^{-1}  \|_{B(\ell_\omega^1)}   \|e^{M t} \|_{B(\ell_\omega^1)}.
\end{equation} 
By construction of the operator $P$ and $P^{-1}$, the computation of $\| P \|_{B(\ell_\omega^1)}$ and $\| P^{-1} \|_{B(\ell_\omega^1)}$ are finite and given by 
\begin{align*}
\| P \|_{B(\ell_\omega^1)} = \max\left( \| P_1 \|_{B(\ell_\omega^1)} , 1 \right) \quad \mbox{ and } \quad \| P^{-1} \|_{B(\ell_\omega^1)} = \max\left( \| P_1^{-1} \|_{B(\ell_\omega^1)} , 1 \right).
\end{align*}
The next step is to define the necessary conditions to use the following theorem.
\begin{thm}[{\cite[Theorem B.1]{JB_Jay_Jaquette2022}}]\label{thm:semigroup_estimate_M}
	Consider the linear operators $M$, $\Lambda$, $\mathscr{E}: \C^{N} \times\ell_{\infty}^1\to\C^{N} \times\ell_{\infty}^1$ defined in \eqref{eq:operators}. We require $\Lambda$  to be densely defined and $\mathscr{E}$ to be bounded. Suppose that $\Lambda_1$ is diagonal and that $\Lambda_\infty:\ell_{\infty}^1\to\ell_{\infty}^1$ has a bounded inverse.
	
	Let $\mu_1$, $\mu_\infty > 0 $ such that for $t>0$ and $(\phi_1,\phi_\infty)\in \C^{N}\times\ell_{\infty}^1$ we have
	\begin{align}
		\left\|e^{\Lambda_1 t}\phi_1\right\|_{\omega,0}\le  e^{-\mu_1 t}\|\phi_1\|_{\omega, 0},\quad\left\|e^{\Lambda_\infty t}\phi_\infty\right\|_{\omega,0}\le  e^{-\mu_\infty t}\|\phi_\infty\|_{\omega, 0}.
	\end{align}
	Fix constants $\delta_a$, $\delta_b$, $\delta_c$, $\delta_d>0$ satisfying
	\begin{align}
		\|\mathscr{E}_1^1\|_{B(\ell_{\omega}^1)}\le\delta_a,\quad
		\|\mathscr{E}_1^\infty\|_{B(\ell_{\omega}^1)}\le\delta_b,\quad
		\|\mathscr{E}_\infty^1\|_{B(\ell_{\omega}^1)}\le\delta_c,\quad
		\|\mathscr{E}_\infty^\infty\|_{B(\ell_{\omega}^1)}\le\delta_d,\quad
	\end{align}
	and set
	\begin{align}
		\eta \bydef \sum_{\lambda \in \sigma(\Lambda_1)} \frac{ 1 }{\mu_\infty - \delta_d - |\lambda|}
	\end{align}
	Assume the following inequalities hold
	\begin{equation} \label{eq:conditions_thm_B1}
	\delta_d + \sup_{\lambda_k \in \sigma(\Lambda_1)} | \lambda_k |  <  \mu_\infty   \quad \mbox{ and } \quad - \mu_\infty + \delta_d + \eta \delta_b \delta_c(1 + \eta^2 \delta_b \delta_c ) \leq - \mu_1. 
\end{equation}	 
	Then, we have the following estimates of semigroup generated by $M$:
	\begin{align}
		\left\|e^{M t}\phi\right\|_{\omega}\le C_s e^{-\lambda_s t}\|\phi\|_{\omega, 0},\quad\phi\in\ell_{\infty}^1,
	\end{align}
	where
	\begin{align}
C_s & \bydef (1 + \eta \delta_b )^2 (1 + \eta \delta_c)^2, \\
\lambda_s & \bydef \mu_1 - C_s \delta_a - \eta \delta_b \delta_c \left( 1 + \eta(2 \delta_b + \delta_c ) + \eta^2 \delta_b \delta_c ( 1 + \eta \delta_b)  \right).
	\end{align}
\end{thm} 
Computing $C_s$ and $\lambda_s$ using interval arithmetic and combining them with \eqref{eq:bound_P} give us
\begin{align*}
\| e^{Df(\tilde{a}) t} \|_{B(\ell_\omega^1)} \leq C e^{-\lambda t }
\end{align*}
where 
\begin{align*}
C \geq \| P  \|_{B(\ell_\omega^1)}  \| P^{-1}  \|_{B(\ell_\omega^1)}  C_s \quad \mbox{ and } \quad  \lambda \leq \lambda_s.
\end{align*}
It is worth mentioning that if we applied Theorem \ref{thm:semigroup_estimate_M} directly to the operator $ \| e^{Df(\tilde{a}) t} \|_{B(\ell_\omega^1)}$ without pseudo-diagonalizing first, we would have $\lambda < 0$, and thus we would not have the hypothesis needed for Theorem \ref{thm:trapping_region} to prove global existence. Also, the truncation size $\bM$ of the finite part of the pseudo-diagonalization is  chosen such that the conditions \eqref{eq:conditions_thm_B1} are satisfied and such that $\lambda > 0$.

\subsection{Computer-assisted proofs}\label{sec:results}

Let us apply Procedure~\ref{proc:global_existence} to study global existence and asymptotic convergence of solutions in the Swift-Hohenberg equation, which is introduced in \cite{swift-hohenberg} to describe the onset of Rayleigh-B\'enard convection
\begin{align}\label{eq:SH}
u_t  = \sigma u - (1+\Delta)^2 u -u^3 = \left( (\sigma-1) - 2 \Delta  - \Delta^2 \right) u - u^3,
\end{align}
where $\sigma \in \R$ is a parameter. It is worth mentioning that several computer-assisted proofs have been presented in the last twenty years to study the dynamics in \eqref{eq:SH}, including constructive proofs of steady states \cite{MR2126387,MR2338393,MR2269503,MR2718657,MR4664056}, global dynamics \cite{MR2136516}, chaos \cite{MR2443030,MR3792791}, stable manifold of equilibria \cite{JB_Jay_Jaquette2022} and solutions to initial value problems in the one-dimensional case \cite{MR4379799,JBMaxime}.

Here, we impose the homogeneous Neumann boundary condition on 3D prism domain $[0,\pi/L_1] \times [0,\pi/L_2]\times [0,\pi/L_3]$ with $(L_1,L_2,L_3) = (1,1.1,1.2)$ (resp.\ 2D rectangle domain $[0,\pi/L_1] \times [0,\pi/L_2]$ with $(L_1,L_2)=(1,1.1)$). Using the general notation \eqref{eq:general_PDE}, 
$\lambda_0=\sigma-1$, $\lambda_1=-2$, $\lambda_2=-1$ and $\jp{\Delta^p N(u)} = -u^3$ ($p=0$).
In this case, the corresponding differential equation \eqref{eq:ODE_trapping_region} in the space of Fourier coefficients in $\ell_{\omega}^1$ is given by 
\[
\dot{a}(t) = f(a) \bydef \cL a(t) - a(t)^3,
\]
where $(\cL a)_\bk = (\sigma - ( 1- (\bk \bL)^2)^2)a_{\bk} (=\mu_{\bk}a_{\bk})$ and the cubic term is the usual convolution defined by
\[
(a^3)_{\bk} \bydef (a*a*a)_{\bk} = \sum_{\bk_1+\bk_2 + \bk_3 = \bk \atop \bk_1,\bk_2,\bk_3 \in \Z^d} a_{\bk_1} a_{\bk_2} a_{\bk_3}.
\]

For the part (P1) of Procedure~\ref{proc:global_existence}, we use the tools of computer-assisted proofs for equilibria of PDEs defined on 3D/2D domains with periodic boundary conditions (e.g. see \cite{MR4159300,MR2338393,MR2718657}) to compute $\ta \in \ell_{\omega}^1$ such that $f(\ta) = \cL \ta - \ta^3$. A brief introduction of such tools is presented in Appendix~\ref{sec:equilibria_SH}.
For the part (P2), we note that for the Swift-Hohenberg equation,
\begin{align*}
	\cG(h) &= f(\ta + h) - Df(\ta)h \\
	& = \cL (\ta + h) - (\ta + h)^3 - (\cL h - 3\ta^2*h) \\
	& = - 3 \ta*h^2 - h^3.
\end{align*}

For the part (P3), the explicit construction of the constants $C>0$ and $\lambda>0$ satisfying \eqref{eq:exponential_operator_norm} is performed rigorously via Theorem \ref{thm:semigroup_estimate_M} in Section \ref{sec:Semigroup_es}.
For the part (P4), fix $\epsilon= 10^{-16}$, which is small enough so that $\delta \bydef \lambda-\epsilon >0$ (the part (P5)).
For the part (P6), one must pick $\rho=\rho(\delta,C)>0$ such that \eqref{eq:assumption_on_g_trapping_region} hold. In our context, we use the Banach algebra structure of $\ell_{\omega}^1$ to reduce the verification of (P6) to
\[
\| \cG(\phi)\|_{\omega} = \| -3 \ta*\phi^2 - \phi^3 \|_{\omega}
\le 3 \| \ta\|_{\omega} \|\phi\|_{\omega}^2 +\| \phi \|_{\omega}^3
\le \frac{\delta}{C} \| \phi \|_{\omega}
\]
for all $\phi \in \ell_{\omega}^1$ with $\|\phi\|_{\omega} \le \rho$. Assuming $\phi \ne 0$, the previous inequality boils down to verify that 
\[
3 \| \ta\|_{\omega} \|\phi\|_{\omega} +\| \phi\|_{\omega}^2
\le \frac{\delta}{C}
\]
for all $\phi \in \ell_{\omega}^1$  $\|\phi\|_{\omega} \le \rho$. A sufficient condition for this to hold is that 
\[
3 \| \ta\|_{\omega} \rho +\rho^2 \le \frac{\delta}{C}.
\]
The largest positive $\rho$ which satisfies this inequality is given by
\begin{equation} \label{eq:explicit_trapping_region_radius_SH}
	\rho \bydef \frac{1}{2} \left( - 3 \| \ta\|_{\omega} + \sqrt{(3 \| \ta\|_{\omega})^2 + \frac{4\delta}{C}}\right) =
	\frac{3}{2} \| \ta\|_{\omega} \left( \sqrt{1 + \frac{4\delta}{9C \| \ta\|_{\omega}^2}} - 1 \right)>0.
\end{equation}

For the part (P7), we integrate \eqref{eq:ODE_trapping_region} via the multi-step scheme introduced in Section \ref{sec:time_stepping} to prove that $\|a(t_K) - \ta\|_{\omega} \le \frac{\rho}{C}$ holds for some $t_K \ge 0$.
To perform this, we denote $\cN_{\bk}(a)=-(a^3)_{\bk}$ in \eqref{eq:the_ODEs}.
The Fr\'echet derivative of $\cN$ at $\ba$ is given by $D\cN(\ba)\phi=-3\left(\ba^2*\phi\right)=-3(\ba*\ba*\phi)$ for $\phi\in\ell_{\omega}^1$ because of the fact
\begin{align}
	\lim_{\|\phi\|_{\omega} \to 0}\frac{\left\|\cN(\ba+\phi)-\cN(\ba)-D\cN(\ba)\phi\right\|_{\omega}}{\|\phi\|_{\omega}}\le\lim_{\|\phi\|_{\omega}\to 0}\left(3\|\ba\|_{\omega}\|\phi\|_{\omega}+\|\phi\|_{\omega}^2\right)= 0.
\end{align}
Then the function $g$ satisfying \eqref{eq:N_assumption} is given by $g(\|\psi\|_{\omega}) = 3\|\psi\|_{\omega}^2$ for $\psi\in\ell_{\omega}^1$. If we take $\psi=\ba(t)$ for a fixed $t$, this $g$ can also be expressed as $g(\|\ba(t)\|_{\omega})=3\|\ba^2(t)\|_{\omega}$. This is because that $\ba^2(t)$ is finite-dimensional sequence in the Fourier dimensions. In addition, $\ba(t)\in\ell_{\omega^{-1}}^\infty$ holds from its finiteness.
From Lemma \ref{lem:conv_es}, if we set $c=c^{(\infty)}\in\ell_{\omega}^1$ and $a=\ba^2(t)\in\ell_{\omega^{-1}}^\infty$ in \eqref{eq:conv_bounds1}, it follows that
\begin{align}
	\left|\left(\ba^2(t)*c^{(\infty)}\right)_\bk\right|  \le \left(\max_{|\bell| \in\bm{F_{k+2N-1}}\setminus\Fm} \left|\ba^2_{\bk-\bell}(t)\right|\omega_{|\bell|}^{-1}\right)\|c^{(\infty)}\|_{\omega},\label{eq:conv_bounds3}
\end{align}
where $\bN$ is the size of the Fourier dimension of $\ba$.
Let
\begin{align}
	\Psi_{\bk}\left(\ba^2(t)\right)\bydef\max_{|\bell| \in\bm{F_{k+2N-1}}\setminus\Fm} \left|\ba^2_{\bk-\bell}(t)\right|\omega_{|\bell|}^{-1}.
\end{align}
Using this notation, it follows from \eqref{eq:Fourier_proj} that
\begin{align}
	\left\|\varPi^{(\bm{m})} D\mathcal{N}(\ba(t))c^{(\infty)}\right\|_{\omega}
	&=3\sum_{\bk\in\Fm}\left|\left(\ba^2(t)*c^{(\infty)}\right)_{\bk}\right|\omega_{\bk}\\
	&\le3\sup_{t \in J}\sum_{\bk\in\Fm}\Psi_{\bk}\left(\ba^2(t)\right)\omega_{\bk}\|c^{(\infty)}\|_{\omega}.
\end{align}
Therefore, the constant $\jp{\cE_{m,\infty}^{J}}$ satisfying \eqref{eq:Cm_bound} is given by
\begin{align}
	\jp{\cE_{m,\infty}^{J}} = 3\sum_{\ell=0}^{2(n-1)}\sum_{\bk\in\Fm}\left|\Psi_{\ell,\bk}\left(\ba^2\right)\right|\omega_{\bk},
\end{align}
where $\Psi_{\ell,\bk}\left(\ba^2\right)$ denotes the Chebyshev coefficients of $\Psi_{\bk}\left(\ba^2(t)\right)$.
Similarly, 
we have for $\bk\notin\Fm$
\begin{align}
	\left(\ba^2(t)*c^{(\bbm)}\right)_\bk &=\sum_{\bk_1+\bk_2 = \bk \atop |\bk_1|\in \bm{F_{2N-1}}\setminus\bm{F_{1}},~|\bk_2|\in\Fm}\left(\ba^2(t)\right)_{\bk_1}c_{\bk_2}.
\end{align}
To get the $\jp{\cE_{\infty,m}^{J}}$ bound satisfying \eqref{eq:Cinf_bound}, we have from the Banach algebra property
\begin{align}
	\left\|(\mathrm{Id}-\varPi^{(\bm{m})} )D\mathcal{N}(\ba(t))c^{(\bm{m})}\right\|_{\omega}
	&=3\sum_{\bk\not\in\Fm}\left|\left(\ba^2(t)c^{(\bbm)}\right)_\bk\right|\omega_{\bk}\\
	&=3\sum_{\bk\not\in\Fm}\left|\sum_{\bk_1+\bk_2 = \bk \atop |\bk_1|\in \bm{F_{2N-1}}\setminus\bm{F_{1}},~|\bk_2|\in\Fm}\left(\ba^2(t)\right)_{\bk_1}c_{\bk_2}\right|\omega_{\bk}\\
	&\le 3\left(\sum_{\bk\in \bm{F_{2N-1}}\setminus\bm{F_{1}}}\left|\left(\ba^2(t)\right)_{\bk}\right|\omega_{\bk}\right)\|c^{(\bbm)}\|_{\omega}.
\end{align}
Hence, we obtain $\jp{\cE_{\infty,m}^{J}}$ in \eqref{eq:Cinf_bound} as
\begin{align}
	\jp{\cE_{\infty,m}^{J}}=3\sum_{\ell=0}^{2(n-1)}\sum_{\bk\in \bm{F_{2N-1}}\setminus\bm{F_{1}}}\left|\left(\ba^2\right)_{\ell,\bk}\right|\omega_{\bk},
\end{align}
where $\left(\ba^2\right)_{\ell,\bk}$ also denotes the Chebyshev coefficients of $\left(\ba^2(t)\right)_{\bk}$.
Finally, recalling Remark \ref{rem:Lba_rho}, we have $L_{\ba}(\varrho)=3\varrho\left(2\|\ba\|_X+\varrho\right)$ in Theorem \ref{thm:local_inclusion}.

Our integrator based on Theorem \ref{thm:local_inclusion_multisteps} proves that there exists a solution of Swift-Hohenberg equation \eqref{eq:SH} in $\bm{B}_K(\bm{\ba}, \bm{\varrho}_0)$ defined in \eqref{eq:theBall_multisteps}. Then we rigorously compute the error bound at the end point, say $\bm{\varrho}_K$, by \eqref{eq:error_at_end point_K}.
We verify that the solution is in the trapping region via the following inequality:
\begin{align}
	\|a^{J_K}(t_K) - \ta\|_{\omega}\le \|a^{J_K}(t_K) - \ba^{J_K}(t_K)\|_{\omega} +\|\ba^{J_K}(t_K) - \ta\|_{\omega}\le \bm{\varrho}_K +  \|\ba^{J_K}(t_K) - \ta\|_{\omega} \le \frac{\rho}{C}.
\end{align}
The second term in the last inequality is rigorously computable via the result of computer-assisted proofs for the equiliubrium based on interval arithmetic.
Consequently, if the last inequality holds, then we have a computer-assisted proof of global existence of the solution to Swift-Hohenberg equation \eqref{eq:SH} for the part (P8).
\jp{We note that, while the global existence of solutions of the Swift-Hohenberg equation is already established, our numerical integrator enables us to begin with any initial condition far from equilibrium, without prior knowledge of the equilibrium to which it will converge. Moreover, our method is not restricted to gradient systems and could potentially be applied to more complex problems.}

\subsubsection{3D Swift-Hohenberg equation} \label{3D Swift-Hohenberg equation}

\begin{figure}[htbp]
	\begin{minipage}[b]{0.49\linewidth}
		\centering
		\includegraphics[width=\textwidth]{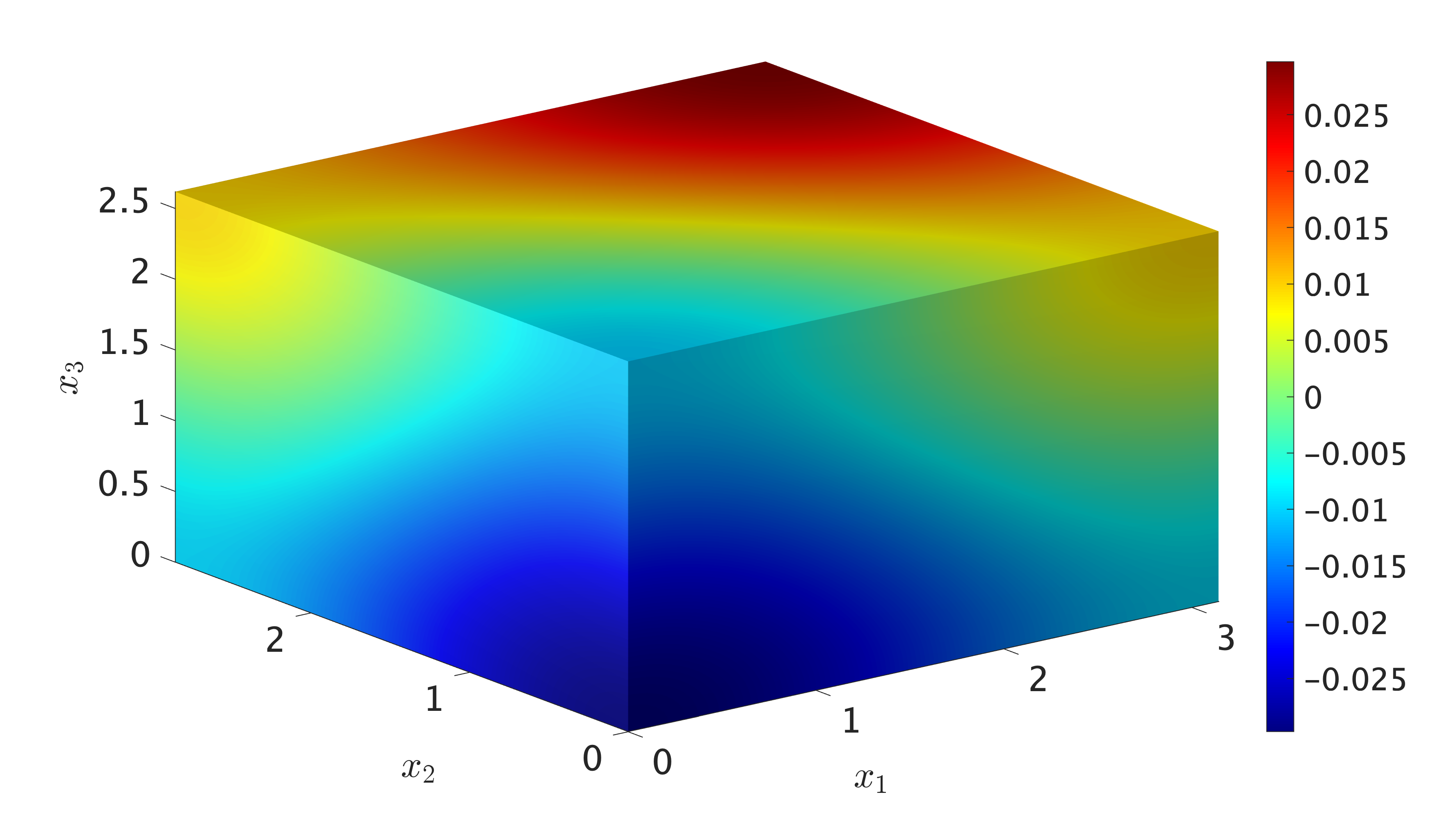}
		\subcaption{Initial data ($t_0=0$).}\label{fig:1_1}
	\end{minipage}
	\begin{minipage}[b]{0.49\linewidth}
		\centering
		\includegraphics[width=\textwidth]{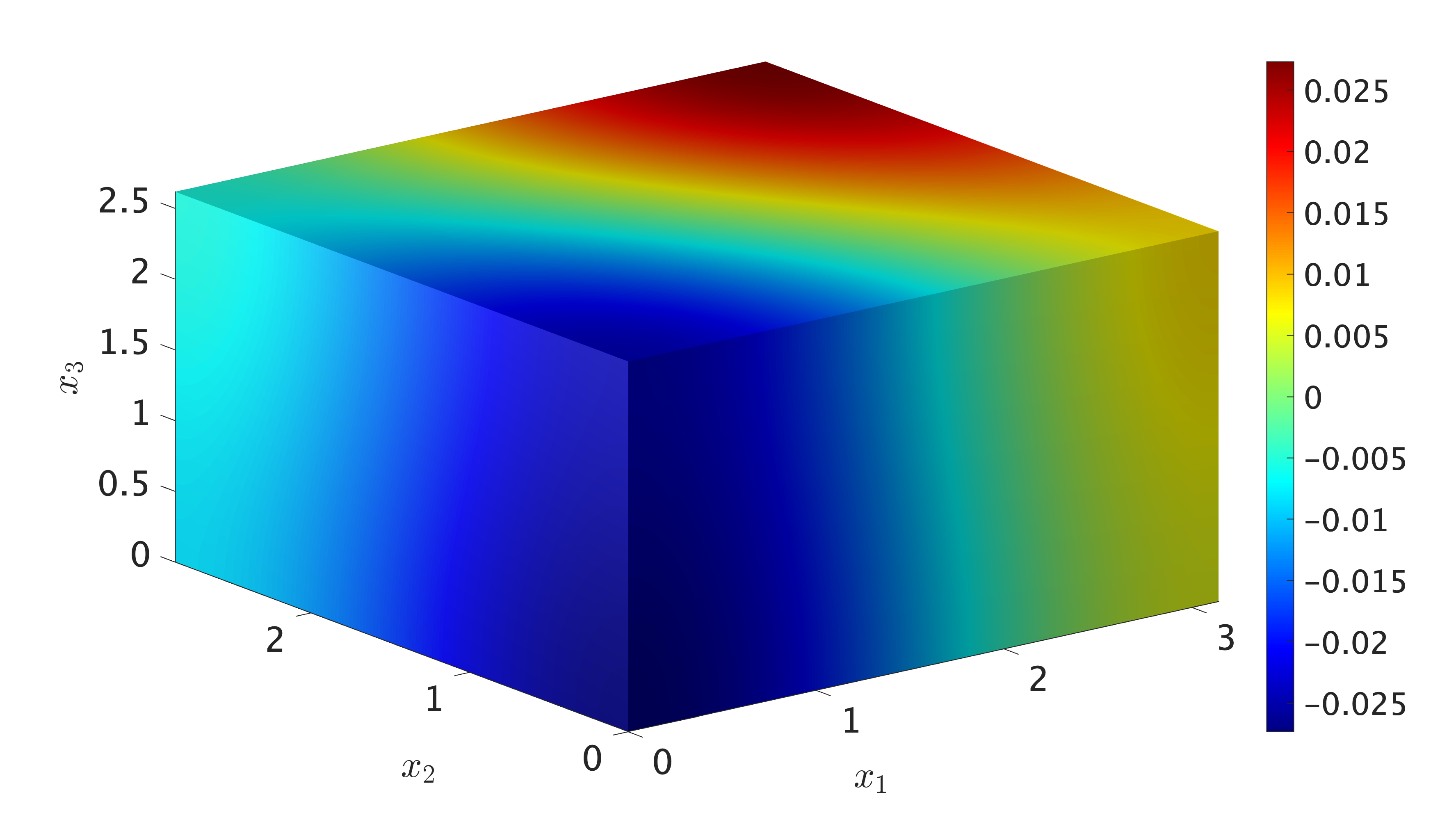}
		\subcaption{After $50$ steps ($t_{25}=12.5$).}\label{fig:1_2}
	\end{minipage}\\[2mm]
	\begin{minipage}[b]{0.49\linewidth}
		\centering
		\includegraphics[width=\textwidth]{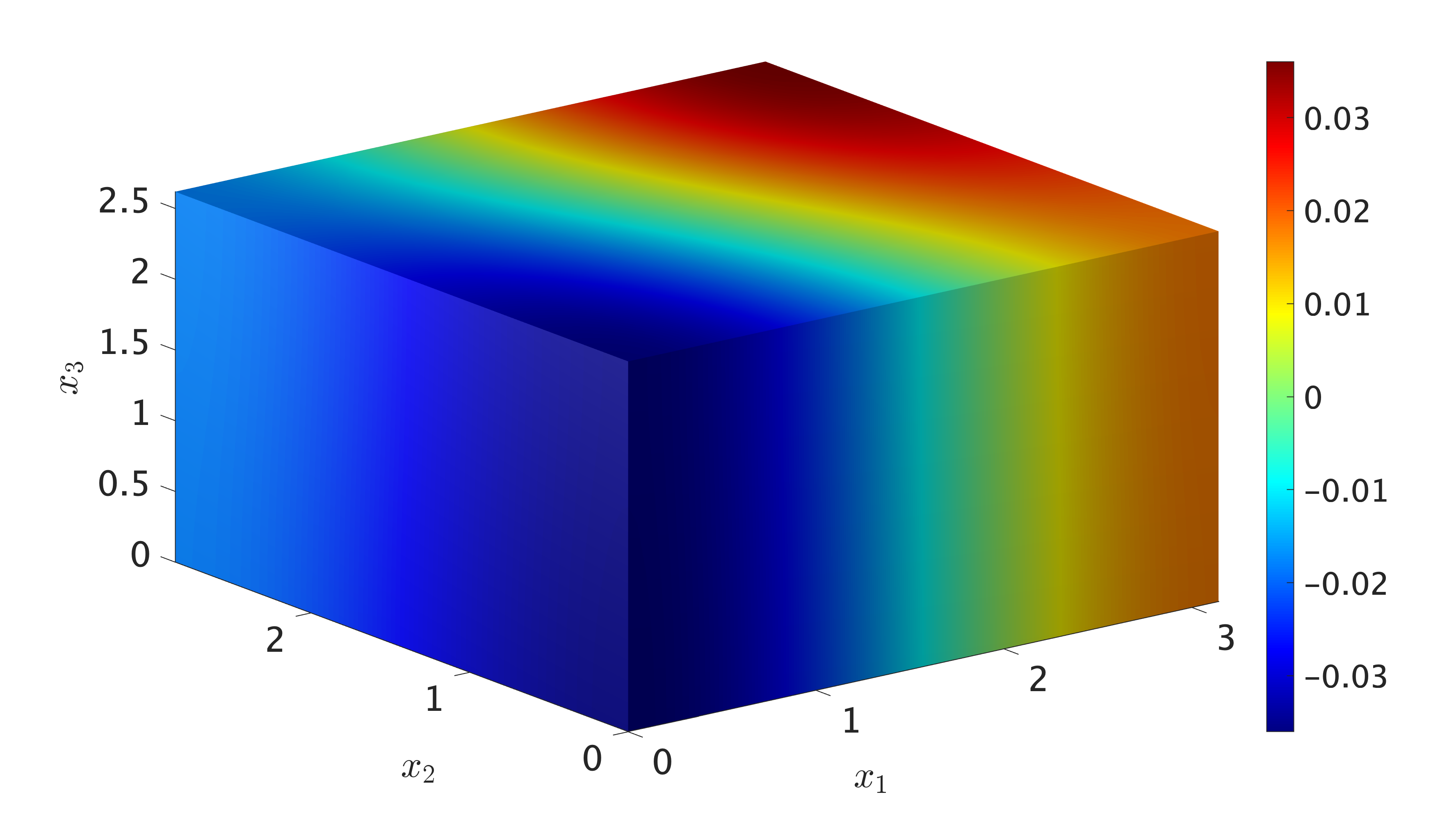}
		\subcaption{After $100$ steps ($t_{100}=25$).}\label{fig:1_3}
	\end{minipage}
	\begin{minipage}[b]{0.49\linewidth}
		\centering
		\includegraphics[width=\textwidth]{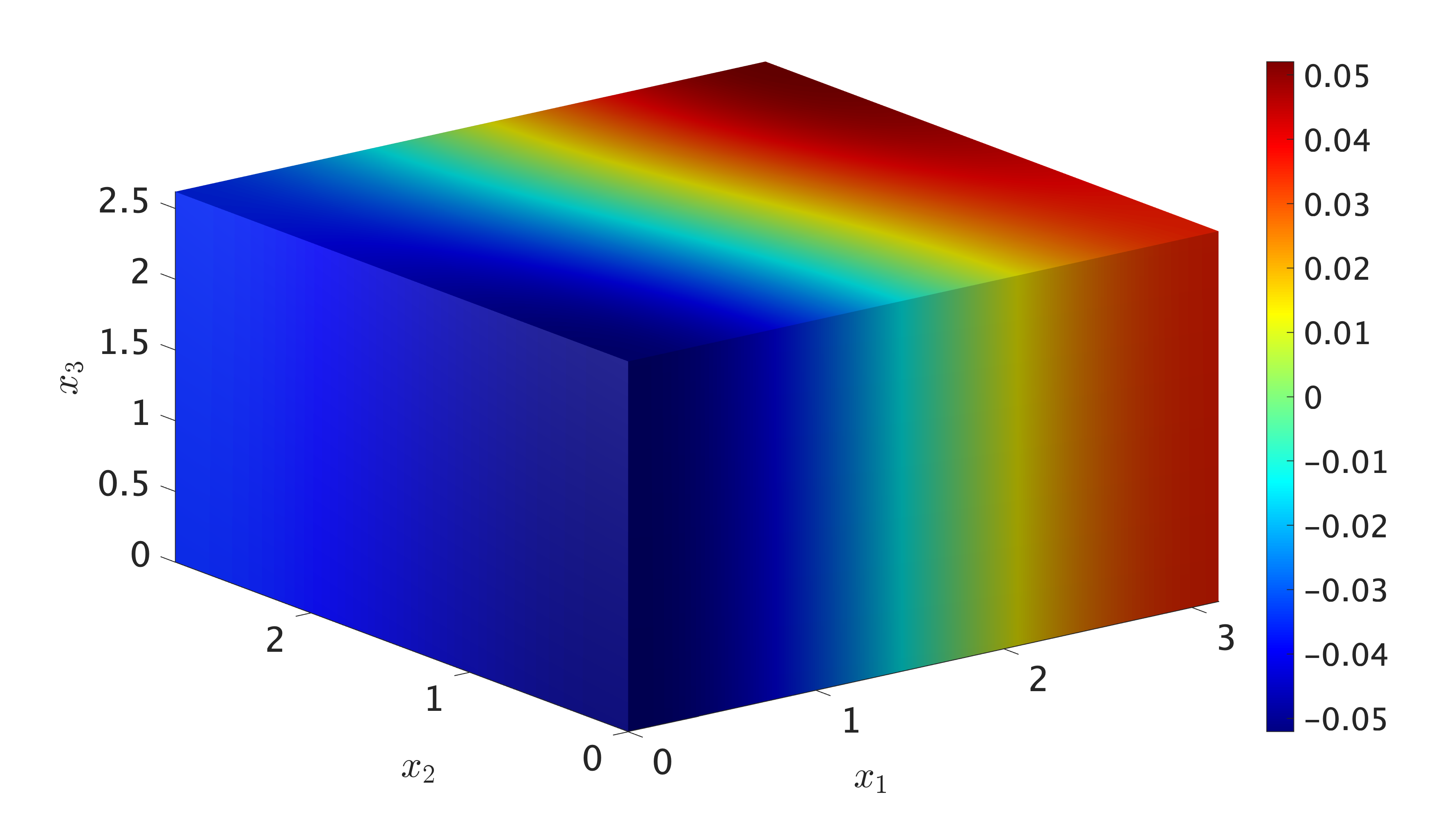}
		\subcaption{After $150$ steps ($t_{150}=37.5$).}\label{fig:1_4}
	\end{minipage}\\[2mm]
	\begin{minipage}[b]{0.49\linewidth}
		\centering
		\includegraphics[width=\textwidth]{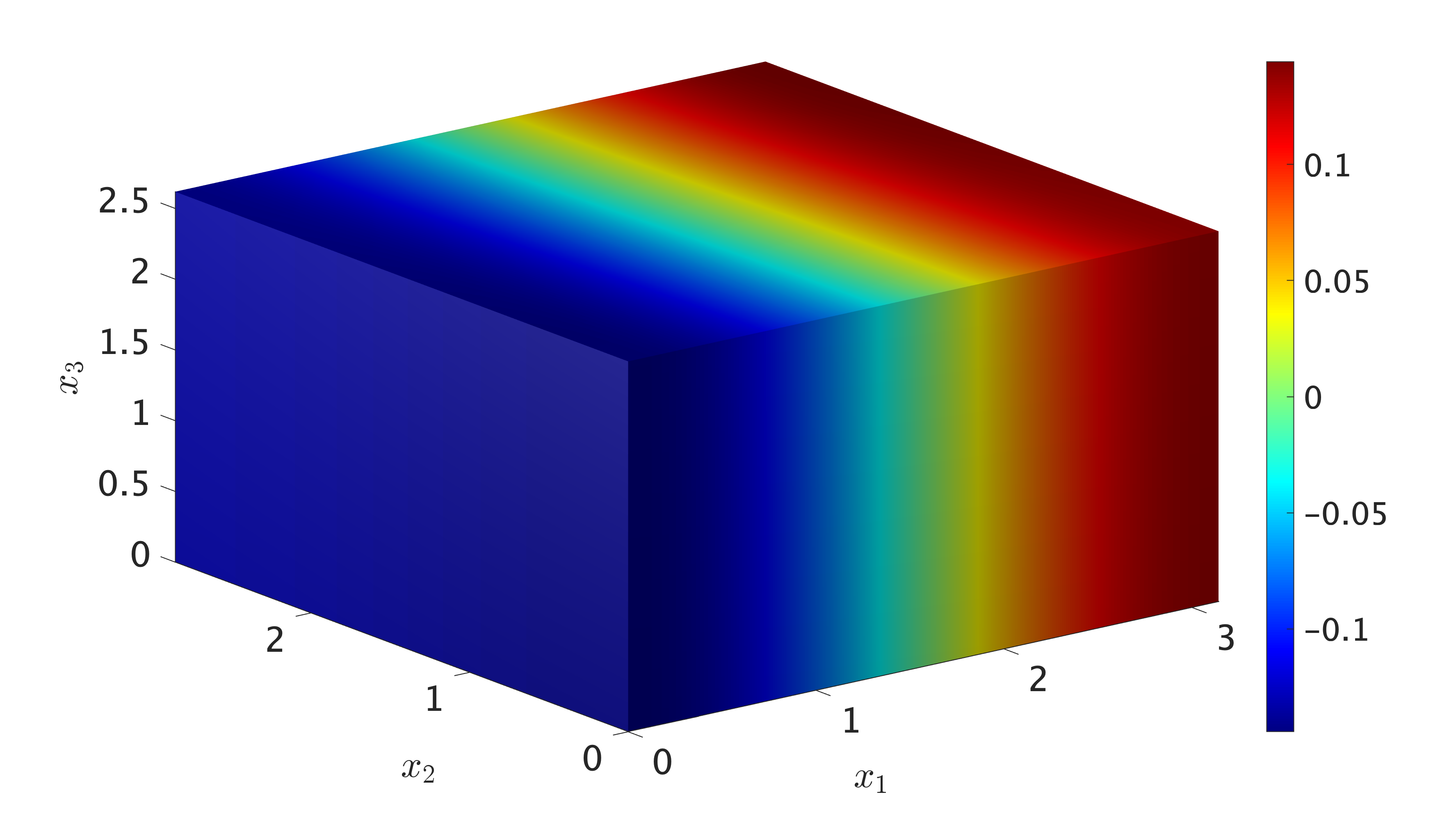}
		\subcaption{After $287$ steps ($t_{287}=71.75$).}\label{fig:1_5}
	\end{minipage}
	\begin{minipage}[b]{0.49\linewidth}
		\centering
		\includegraphics[width=\textwidth]{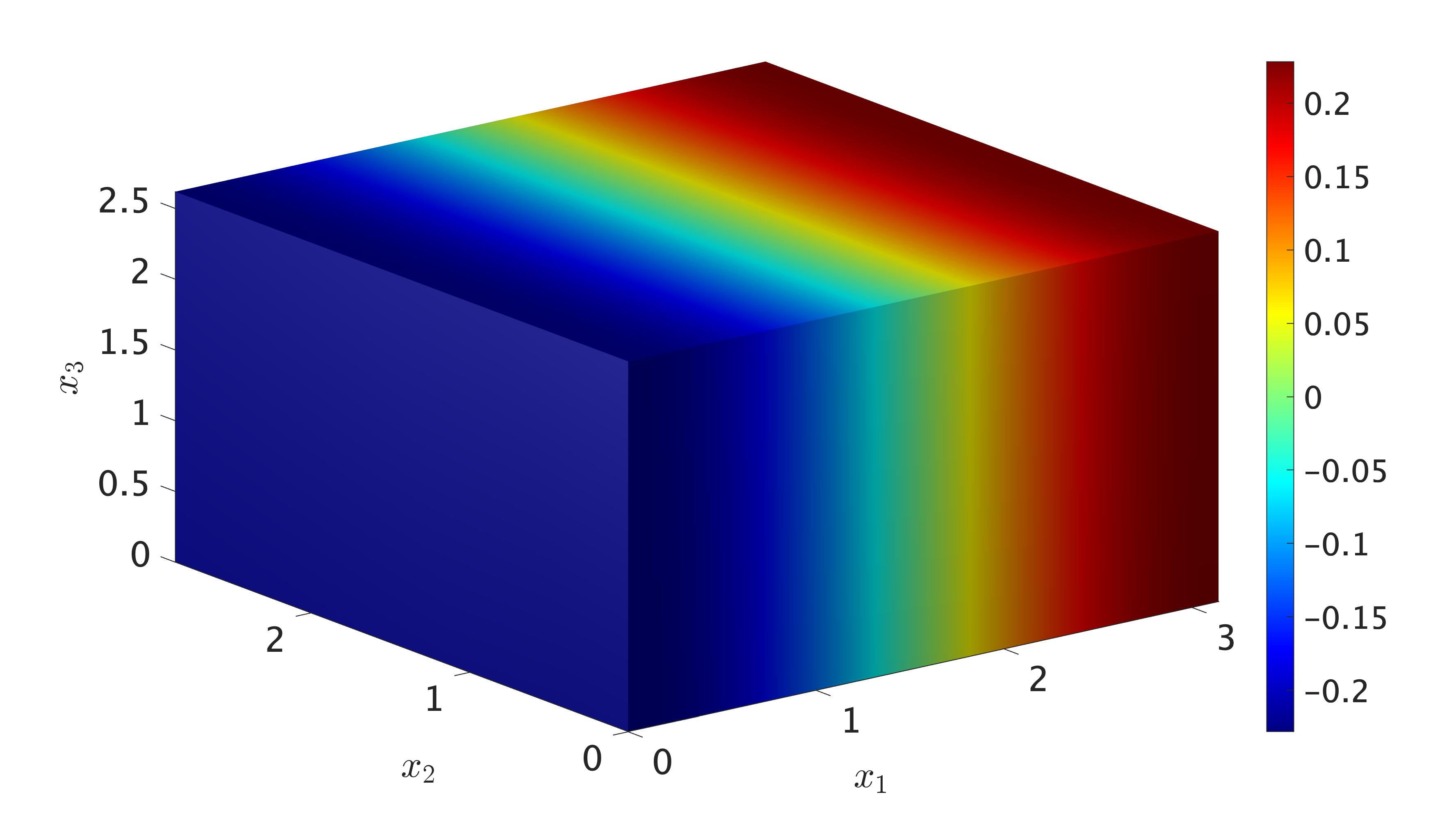}
		\subcaption{Profile of stable equilibrium (stripe pattern)}\label{fig:1_6}
	\end{minipage}
	\caption{The profiles of solution to the 3D Swift-Hohenberg: From the initial data (a), our integrator proves the existence of solution locally in time. After $i$ ($i=25$, $100$, $150$, $\dots$) steps, the time evolution of solution profile is fully changed in the 3D prism domain. Consequently, after $287$ steps, we verified that the exact solution enclosure is in the trapping region in the part (P7) of Procedure \ref{proc:global_existence}. Therefore, we proved that there exists a global in time solution from the initial data (a) to the stripe pattern equilibrium (f).}\label{fig:1}
\end{figure}

From now on, we fix $\sigma=0.04$ in the 3D case of \eqref{eq:SH}. We set $\bN=(4,4,4)$ for the size of Fourier and $\bbm=(2,1,1)$ for the Fourier projection to obtain the solution map shown in Section \ref{sec:solution_map}. We also set $\nu_{\jp{F}}=1$ for the $\ell_{\omega}^1$ norm.
As shown in Fig.\,\ref{fig:1}\,(f), we have a stable (stripe pattern) equilibrium of Swift-Hohenberg equation \eqref{eq:SH} via the tools of computer-assisted proofs.
The explicit construction of the constants $C$ and $\lambda$ introduced in Secion \ref{sec:Semigroup_es} yields that $C=1.0216$ and $\lambda=0.0799$ for the part (P3).
The largest positive $\rho$ defined in Section \ref{eq:explicit_trapping_region_radius_SH} is obtained as $\rho = 0.0988$ (the part (P6)). Therefore, our target neighborhood of the equilibrium is $\rho/C =  0.0967$.

The initial data is set as
\begin{align}
	u_0(x) = \sum_{\bk\ge 0}\alpha_{\bk}\varphi_{\bk}\cos(k_1L_1x_1)\cos(k_2L_2x_2)\cos(k_3L_3x_3)
\end{align}
with
\begin{align}
	\varphi_{\bk}=\begin{cases}
		-0.005, & \bk=(1,0,0),~(0,1,0),~(0,0,1)\\
		0, & \mbox{otherwise}
	\end{cases}.
\end{align}
Setting $\tau_i=0.25$, our rigorous integrator proves that the solution of the IVP of \eqref{eq:SH} exists at least to $t_K=71.75$ ($K=287$). The resulting error bound $\bm{\varrho}_0$ via Theorem \ref{thm:local_inclusion_multisteps} is given by
\begin{align}
	\sup_{t\in(0,71.75]}\|a(t)-\ba(t)\|_{\omega}\le  4.7241\cdot 10^{-6}=\bm{\varrho}_0.
\end{align}
In this case we also obtain $\bm{\varrho}_K=4.7241\cdot 10^{-6}$. The value of $\bm{\varrho}_K +  \|\ba^{J_K}(t_K) - \ta\|_{\omega}$ is bounded by $0.0962$, which is less than $\rho/C$. Hence, for the part (P8), the proof of global existence of the solution to the Swift-Hohenberg equation \eqref{eq:SH} is completed.

Several profiles of the solution are presented in Fig.\,\ref{fig:1}, demonstrating that the evolutionary behavior of the solution undergoes significant changes within the 3D prism domain. We believe that such a capability to visualize the solution's evolution in a higher domain, underpinned by mathematical proof, provides a new perspective of pattern dynamics.

\subsubsection{2D Swift-Hohenberg equation} \label{sec:2D-SH}

Next let us consider the 2D case of Swift-Hohenberg equation \eqref{eq:SH}. The parameter is fixed as $\sigma = 3$. Computational parameters are set by $\bN=(12,12)$, $\bbm=(3,3)$, $\nu_{\jp{F}}=1.0$.
Our tool of computer-assisted proofs obtains two stable equilibria of Swift-Hohenberg equation \eqref{eq:SH} shown in Fig.\,\ref{fig:2}. Both equilibria are asymptotically stable. We apply the provided approach of computer-assisted proofs for existence of global in time solutions.

\begin{figure}[htbp]
	\begin{minipage}[b]{0.49\linewidth}
		\centering
		\includegraphics[width=\textwidth]{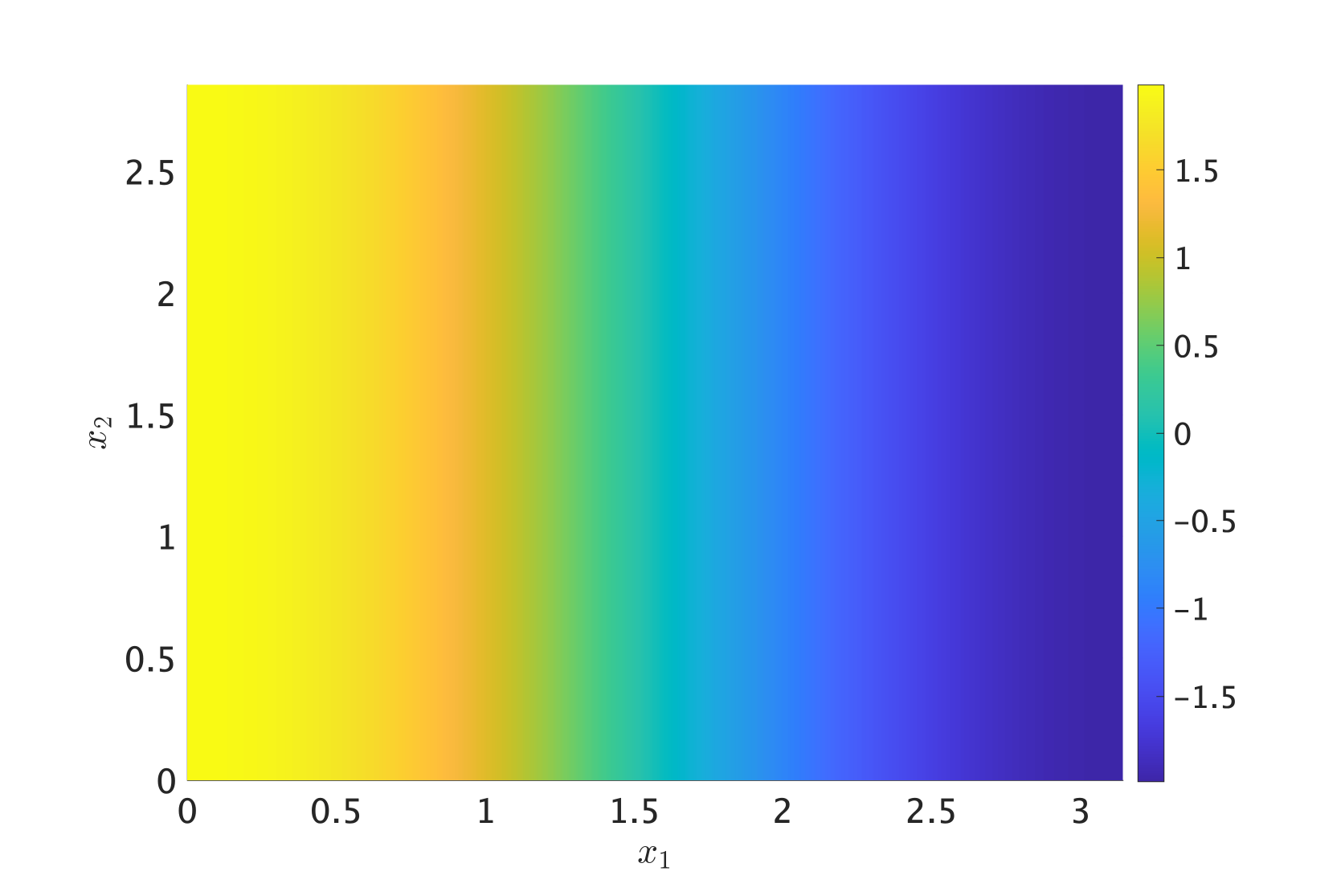}
		\subcaption{Stripe pattern equilibrium $\tilde{u}_1(x)$.}\label{fig:2_1}
	\end{minipage}
	\begin{minipage}[b]{0.49\linewidth}
		\centering
		\includegraphics[width=\textwidth]{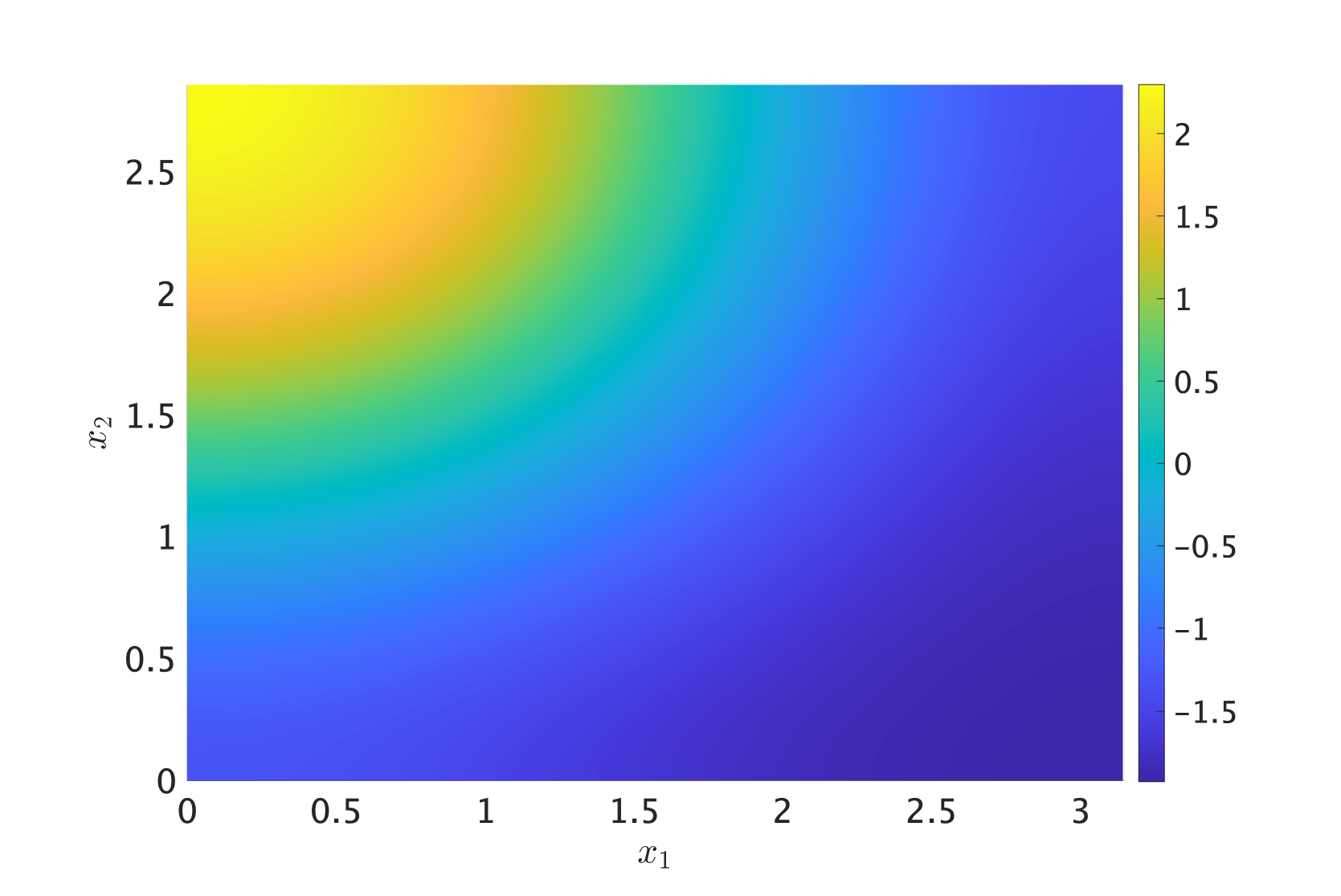}
		\subcaption{Spot pattern equilibrium $\tilde{u}_2(x)$.}\label{fig:2_2}
	\end{minipage}
	\caption{Profiles of two stable equilibria of the 2D Swift-Hohenberg equation: (a) Stripe pattern equilibrium solution $\tilde{u}_1(x)$ and (b) spot pattern equilibrium solution $\tilde{u}_2(x)$.}\label{fig:2}
\end{figure}

\paragraph{Stripe pattern equilibrium.}
The constants $C$ and $\lambda$ are determined by $C=2.7295$ and $\lambda=1.9406$, respectively.
The largest positive value of $\rho$ introduced in Section \ref{eq:explicit_trapping_region_radius_SH} is identified as $\rho = 0.1138$. 
This leads the radius of our target neighborhood near the equilibrium solution, given as $\rho/C = 0.0416$. 
The initial data $u_0(x)$ we used here is plotted in Fig.\,\ref{fig:3}\,(a). 
Using the provided integrator, we achieve a rigorous inclusion of the solution for the IVP up to $t_K=1.4141$ (where $K=181$), by setting $\tau_i=7.8125\cdot 10^{-3}$ ($i=1,2,\dots,K$) uniformly for each step.

The resulting error bound $\bm{\varrho}_0$ is given by
\begin{align}
	\sup_{t\in(0,1.4141]}\|a(t)-\ba(t)\|_{\omega}\le  2.9081\cdot 10^{-6}=\bm{\varrho}_0.
\end{align}
The $\bm{\varrho}_K$ bound is obtained by $\bm{\varrho}_K=2.6919\cdot 10^{-6}$. 
Recalling the part (P7) of Procedure \ref{proc:global_existence}, the value of $\bm{\varrho}_K +  \|\ba^{J_K}(t_K) - \ta\|_{\omega}$ is bounded by $0.0414$.
This is definitely less than $\rho/C$. 
Therefore, we have a computer-assisted proof of global existence of the solution asymptotically converging to the stripe pettern equilibrium solution plotted in Fig.\,\ref{fig:2}\,(a).

\begin{figure}[htbp]
	\begin{minipage}[b]{0.49\linewidth}
		\centering
		\includegraphics[width=\textwidth]{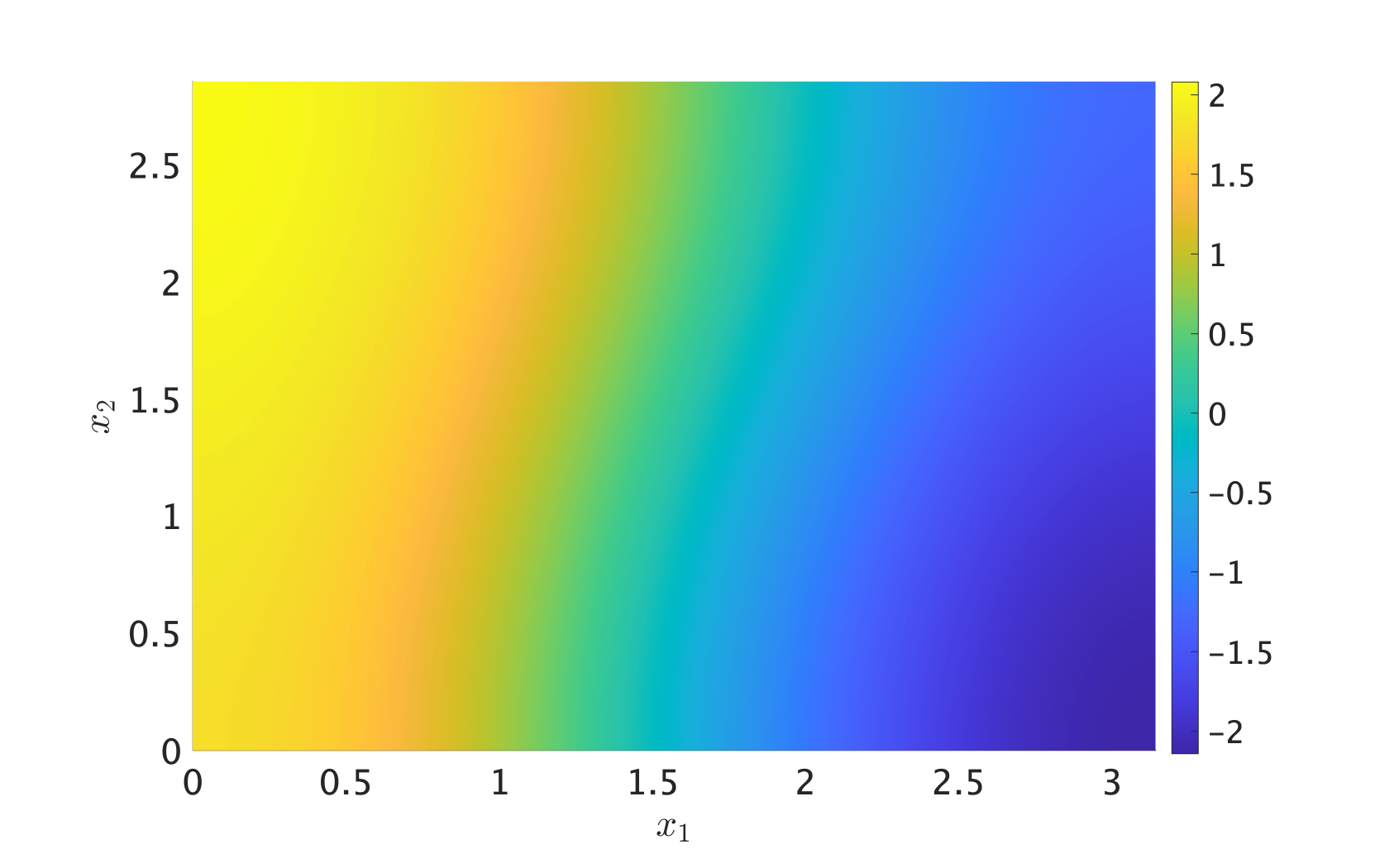}
		\subcaption{The initial data $u_0(x)$.}\label{fig:3_1}
	\end{minipage}
	\begin{minipage}[b]{0.49\linewidth}
		\centering
		\includegraphics[width=\textwidth]{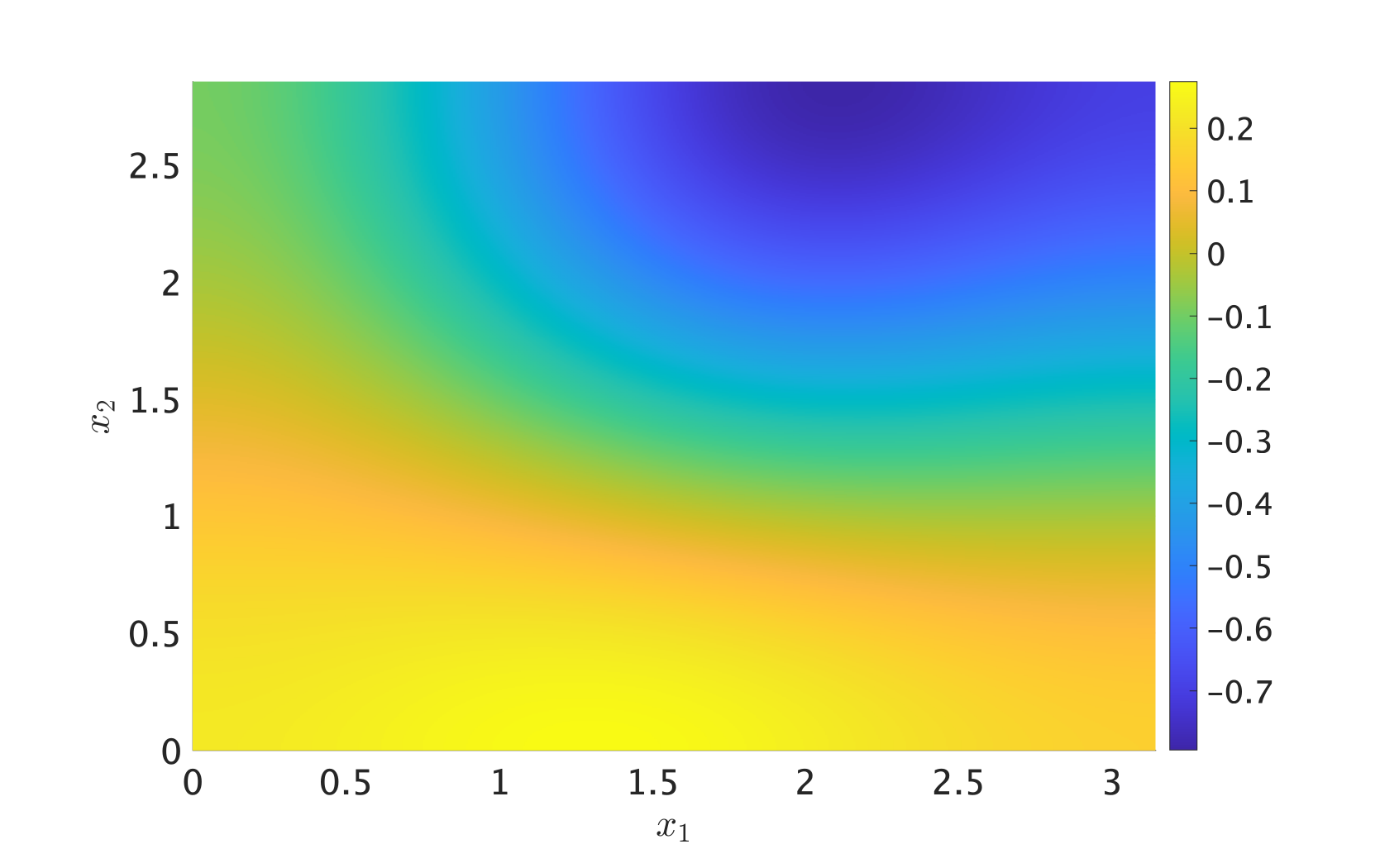}
		\subcaption{Plot of $\tilde{u}_1(x)-u_0(x)$.}\label{fig:3_2}
	\end{minipage}
	\caption{The initial data $u_0(x)$ and the difference between the stable stripe pattern equilibrium solution $\tilde{u}_1(x)$ and $u_0(x)$: To demonstrate computer-assisted proofs for global existence beyond the neighborhood of the equilibrium solution, our integrator is essential for capturing variations in the solution profile.}\label{fig:3}
\end{figure}

As shown in Fig.\,\ref{fig:3}\,(b), there is a substantial difference between the equilibrium solution $\tilde{u}_1(x)$ and the initial data $u_0(x)$ in this case. Consequently, our rigorous integrator plays a crucial role in providing computer-assisted proofs for the global existence of solutions originating from initial data, particularly those distant from the neighborhood of equilibria.

\paragraph{Spot pattern equilibrium.}
Let us consider another solution converging to the spot pattern equilibrium solution in Fig.\,\ref{fig:2}\,(b).
In this case, the values of $C$ and $\lambda$ are determined to be  $C=6.7813$ and $\lambda=0.3505$, respectively.
The largest positive $\rho$ is found to be $\rho = 0.0046$.
Consequently, these values lead to a radius of the target neighborhood around the equilibrium, calculated as $\rho/C = 6.746\cdot 10^{-4}$. 
The initial data, denoted by $u_0(x)$, is plotted in Fig.\,\ref{fig:4}\,(a), which is slightly close to the target equilibrium solution.
Nevertheless, the integrator helps us to propagate the rigorous inclusion of the solution into the trapping region.
In other words, without the integrator, the global existence of the solution cannot be proved.
The integration is carried out until $t_K=0.375$ (with $K=12$), setting $\tau_i=3.125\cdot 10^{-2}$ ($i=1,2,\dots,K$).

The  error bound $\bm{\varrho}_0$ in Theorem \ref{thm:local_inclusion_multisteps} is obtained by
\begin{align}
	\sup_{t\in(0,0.375]}\|a(t)-\ba(t)\|_{\omega}\le  1.2099\cdot 10^{-6}=\bm{\varrho}_0.
\end{align}
The $\bm{\varrho}_K$ bound is $\bm{\varrho}_K=9.6731\cdot 10^{-7}$. Then the value of $\bm{\varrho}_K +  \|\ba^{J_K}(t_K) - \ta\|_{\omega}$ is bounded by $6.669\cdot 10^{-4}$. This value is less than $\rho/C$.
From Theorem~\ref{thm:trapping_region}, there exists a solution globally in time, which converges to the spot pattern equilibrium solution.

\begin{figure}[htbp]
	\begin{minipage}[b]{0.49\linewidth}
		\centering
		\includegraphics[width=\textwidth]{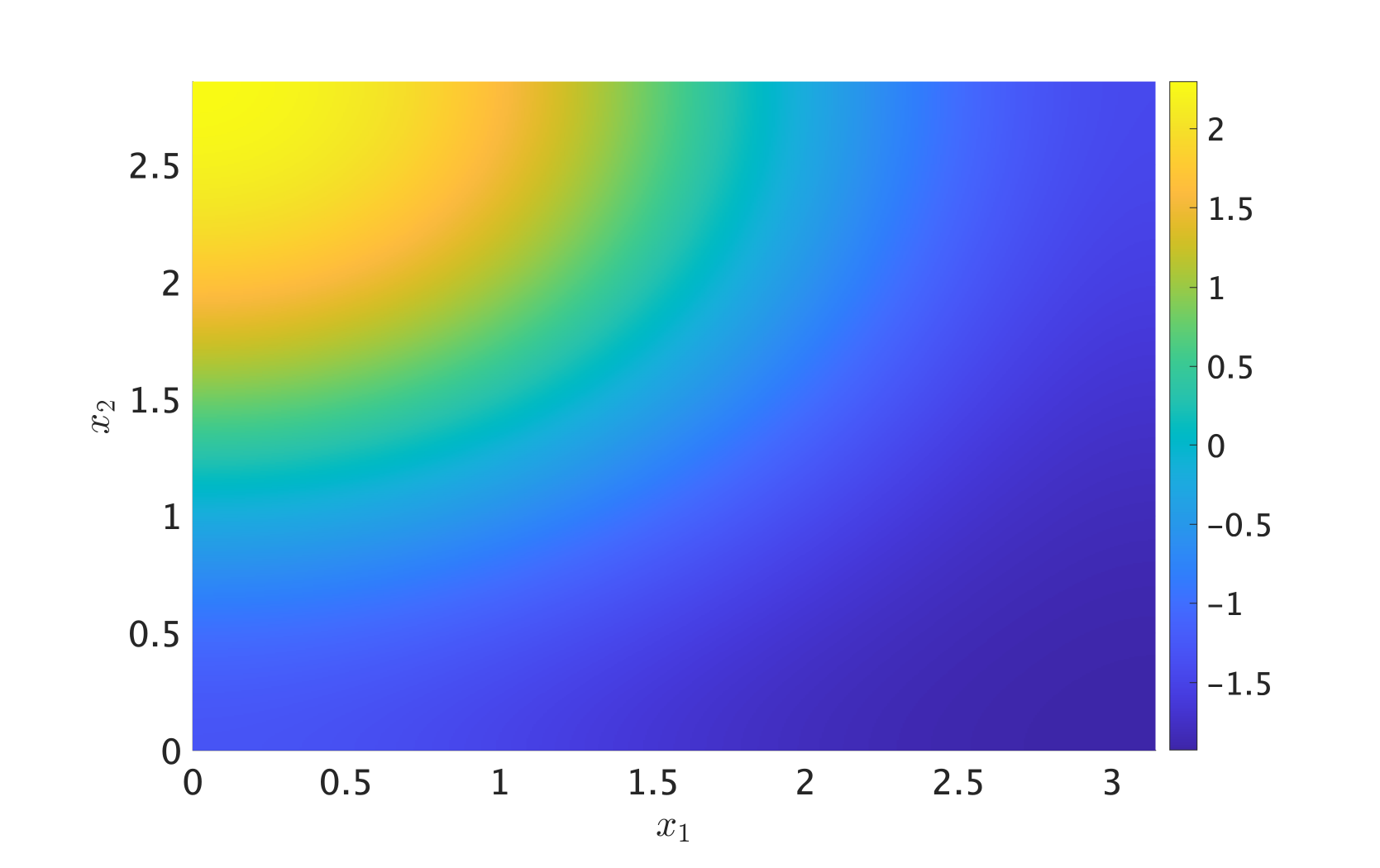}
		\subcaption{The initial data $u_0(x)$.}\label{fig:4_1}
	\end{minipage}
	\begin{minipage}[b]{0.49\linewidth}
		\centering
		\includegraphics[width=\textwidth]{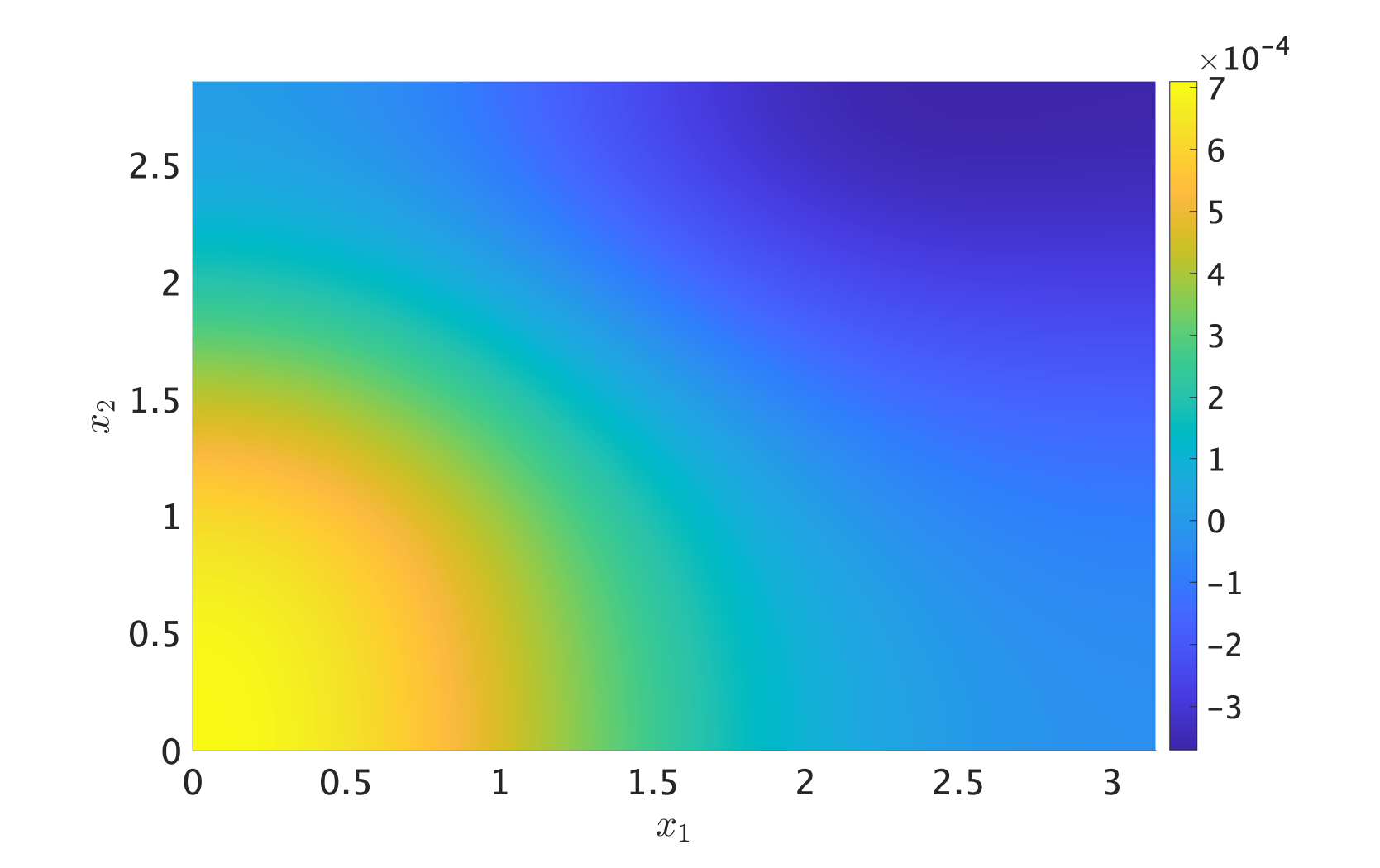}
		\subcaption{Plot of $\tilde{u}_2(x)-u_0(x)$.}\label{fig:4_2}
	\end{minipage}
	\caption{The initial data $u_0(x)$ and the difference between the stable spot pattern equilibrium solution $\tilde{u}_2(x)$ and $u_0(x)$: The initial data is close to the equilibrium solution but it is still out of the trapping region. Our integrator is neccessary to bring the rigorous inclusion of the solution into the trapping region.}\label{fig:4}
\end{figure}

A notable aspect of this result is that the global existence of the solution has only been proved for initial values close to the equilibrium solution. The reason is the exceptionally small trapping region required for proving global existence, which is related to the stability of this equilibrium solution. In other words, this equilibrium solution has a slow stable manifold, and on such a slow manifold the rigorous integrator is difficult to succeed for a long time period.
To successfully demonstrate computer-assisted proofs for global existence from initial values further from the equilibrium, it is necessary to progressively adjust the step size $\tau_i$ of each time step. Or the accuracy of the approximate solution needs to be improved, which leads to smaller bounds $\delta^{J_i}$ and $\varepsilon^{J_i}$ in Theorem \ref{thm:local_inclusion_multisteps}. These improvements should be positioned as future works.


\section{Application to the Ohta-Kawasaki equation} \label{sec:OK}

In this section, as the second application of the provided integrator, we consider the 2D Ohta-Kawasaki equation
\begin{equation}\label{eq:OK}
	u_t = -\Delta\left(\epsilon^2\Delta u + u - u^3\right) - \sigma(u-m)
\end{equation}
which is a nonlinear PDE used in the study of microphase separation in diblock copolymers, and which is pivotal in material science for predicting and understanding the self-organizing structures in soft condensed matter systems. Developed by Ohta and Kawasaki \cite{Ohta-Kawasaki}, it models the complex pattern formation in polymer blends, driven by the balance between entropic effects and chemical incompatibility. Note that the Ohta-Kawasaki equation \eqref{eq:OK} has been studied recently with the tools of rigorous numerics, including constructive computer-assisted proofs of existence of steady states \cite{MR3904424,MR4159300,MR3636312,MR4191556} and connecting orbits \cite{MR3773757}.

In this paper, we set $\epsilon=0.4$, $\sigma=1$, $m=0$ (the average of solution, i.e., $\frac{1}{|\Omega|}\int_\Omega u(x)dx$) and consider the equation \eqref{eq:OK} under the homogeneous Neumann boundary condition on the 2D rectangle domain $[0,\pi/L_1] \times [0,\pi/L_2]$ with $(L_1,L_2)=(1,1.1)$.
We note that the average of solution denoted by $m$ is conserved for any time. In other words, the zero mode of the Fourier coefficient is fixed as $m=0$ in this example.
Using the general notation \eqref{eq:general_PDE}, 
$\lambda_0=-\sigma$, $\lambda_1=-1$, $\lambda_2=-\epsilon^2$ and $\jp{\Delta^p N(u)} = \Delta (u^3)$ (that is $p=1$ and $N(u)=u^3$).
The corresponding ODEs \eqref{eq:the_ODEs} for the time-dependent Fourier coefficients is given by
\[
	\dot{a}_{\bk}(t) = \left(-\sigma+(\bk \bL)^2-\epsilon^2(\bk \bL)^4\right) a_{\bk}(t) - (\bk \bL)^2(a(t)^3)_\bk.
\]
Setting $\mu_{\bk}=-\sigma+(\bk \bL)^2-\epsilon^2(\bk \bL)^4$, $q=2p=2$, and $\cN_\bk(a)=(a^3)_\bk$ in the framework, we integrate \eqref{eq:the_ODEs} via the multi-step scheme in Section \ref{sec:time_stepping} until some $t_K \ge 0$.

\jp{
Unlike the earlier Swift-Hohenberg cases, achieving rigorous integration of the Ohta-Kawasaki equation for long time $t_K$ is challenging. This difficulty primarily arises from the wrapping effect encountered in rigorous forward-time integration, which stems from the overestimation of several bounds presented in Section~\ref{sec:solving_IVP}. Therefore, for successful integration up to a large $t_K$, it is essential to carefully choose the time steps $J_i$ as defined in Section~\ref{sec:time_stepping}.
In other words, a meticulous partitioning of the time interval $[0,t_K]$ allows for long-time integration.
Thus, we employed an ``adaptive'' procedure based on the error bound at the end point of each time step, as introduced in Section~\ref{sec:err_endpoint}.
The step sizes $\tau_i$ of time steps are automatically adjusted to ensure that the ratio of the error bound at $t=t_{i-1}$ to the one at $t=t_i$ ($i=1,2,\dots,K$) does not exceed a predetermined threshold, set in this case to $1.2$.

Note that while this ``adaptive'' procedure is similar to the \emph{time-stepping} method, the resulting validated error bound is obtained all at once over the entire time interval $[0,t_K]$. This approach corresponds to the multi-step method introduced in Section \ref{sec:time_stepping}, rather than the standard step-by-step procedure.
}

The Fr\'echet derivative of $\cN$ at $\ba$ is the same as that of Swift-Hohenberg, given by $D\cN(\ba)\phi=3\left(\ba^2*\phi\right)$ for $\phi\in\ell_{\omega}^1$.
Accordingly, the function $g$ satisfying \eqref{eq:N_assumption} is also the same. As a result, $\jp{\cE_{m,\infty}^{J}}$ and $\jp{\cE_{\infty,m}^{J}}$ bounds and $L_{\ba}(\varrho)=3\varrho\left(2\|\ba\|_X+\varrho\right)$ in Theorem \ref{thm:local_inclusion} are equal to those previously determined.
The main difference from the case of the Swift-Hohenberg equation in the previous section is the derivative at the nonlinear term.
We set $\nu_{\jp{F}}=1.001$, $\gamma=0.5$, and $\xi=0.8$ in Section \ref{sec:solution_map}. 
The other computational parameters are determined by $\bN=(12,12)$ and $\bbm=(3,3)$.

\subsection{Computer-assisted proofs}
We consider two initial data, whose time evolutionary solution profiles go to different stationary states as shown in Fig.\,\ref{fig:5}.
Such two initial datum are given by the form
\begin{align}
	u_0^i(x) = \sum_{\bk\ge 0}\alpha_{\bk}(\varphi^i)_{\bk}\cos(k_1L_1x_1)\cos(k_2L_2x_2),\quad i=1,2
\end{align}
with
\begin{align}\label{eq:initial_OK}
	(\varphi^1)_{\bk}=\begin{cases}
		0.002, & \bk=(2,0)\\
		0.02, & \bk=(0,1)\\
		0, & \mbox{otherwise}
	\end{cases},\quad
	(\varphi^2)_{\bk}=\begin{cases}
		-0.02, & \bk=(1,0)\\
		0.001, & \bk=(1,1)\\
		0, & \mbox{otherwise}
	\end{cases}.
\end{align}

\begin{figure}[htbp]
	\begin{minipage}[b]{0.49\linewidth}
		\centering
		\includegraphics[width=\textwidth]{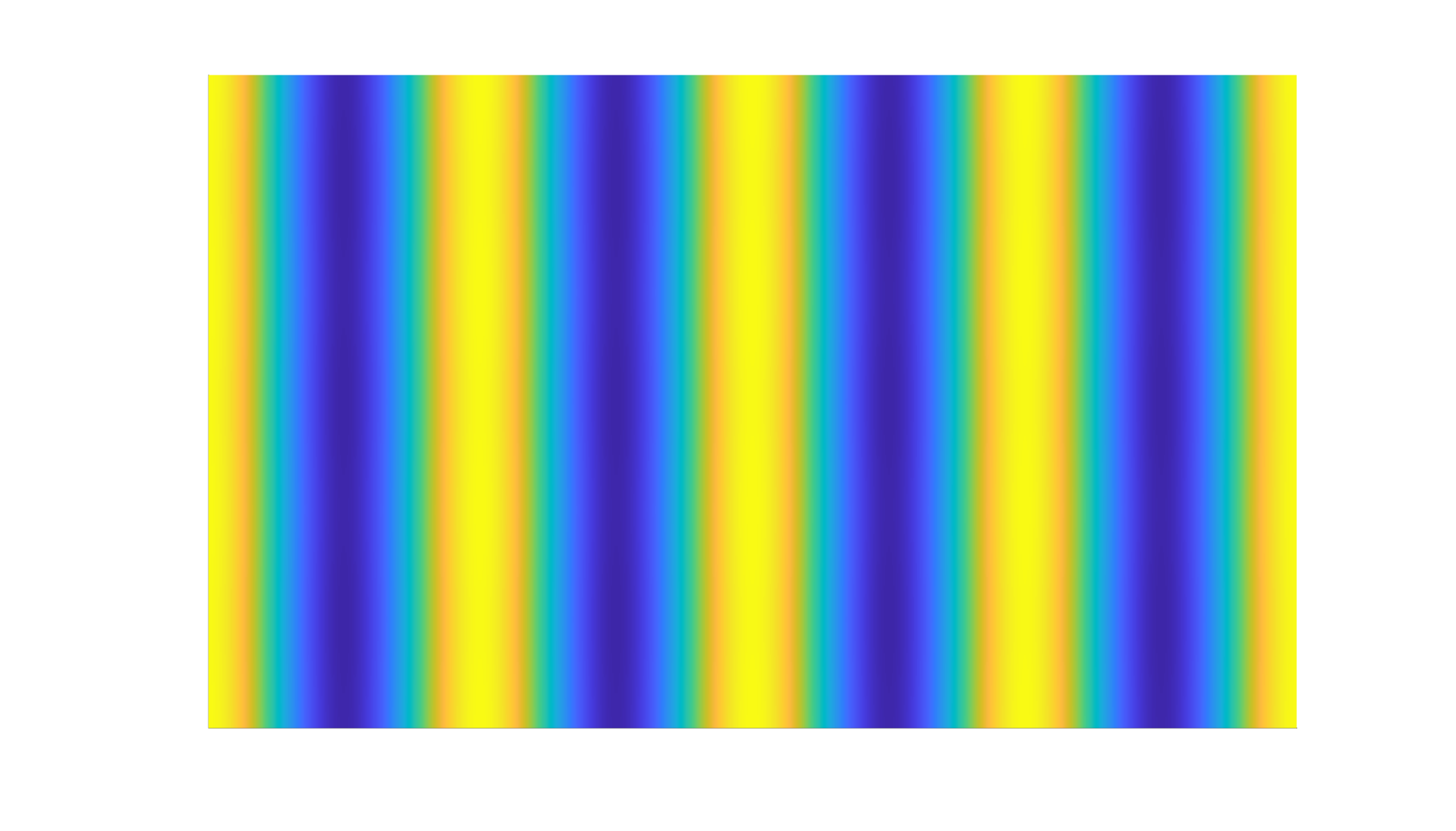}
		\subcaption{Stripe pattern stationary state.}\label{fig:5_1}
	\end{minipage}
	\begin{minipage}[b]{0.49\linewidth}
		\centering
		\includegraphics[width=\textwidth]{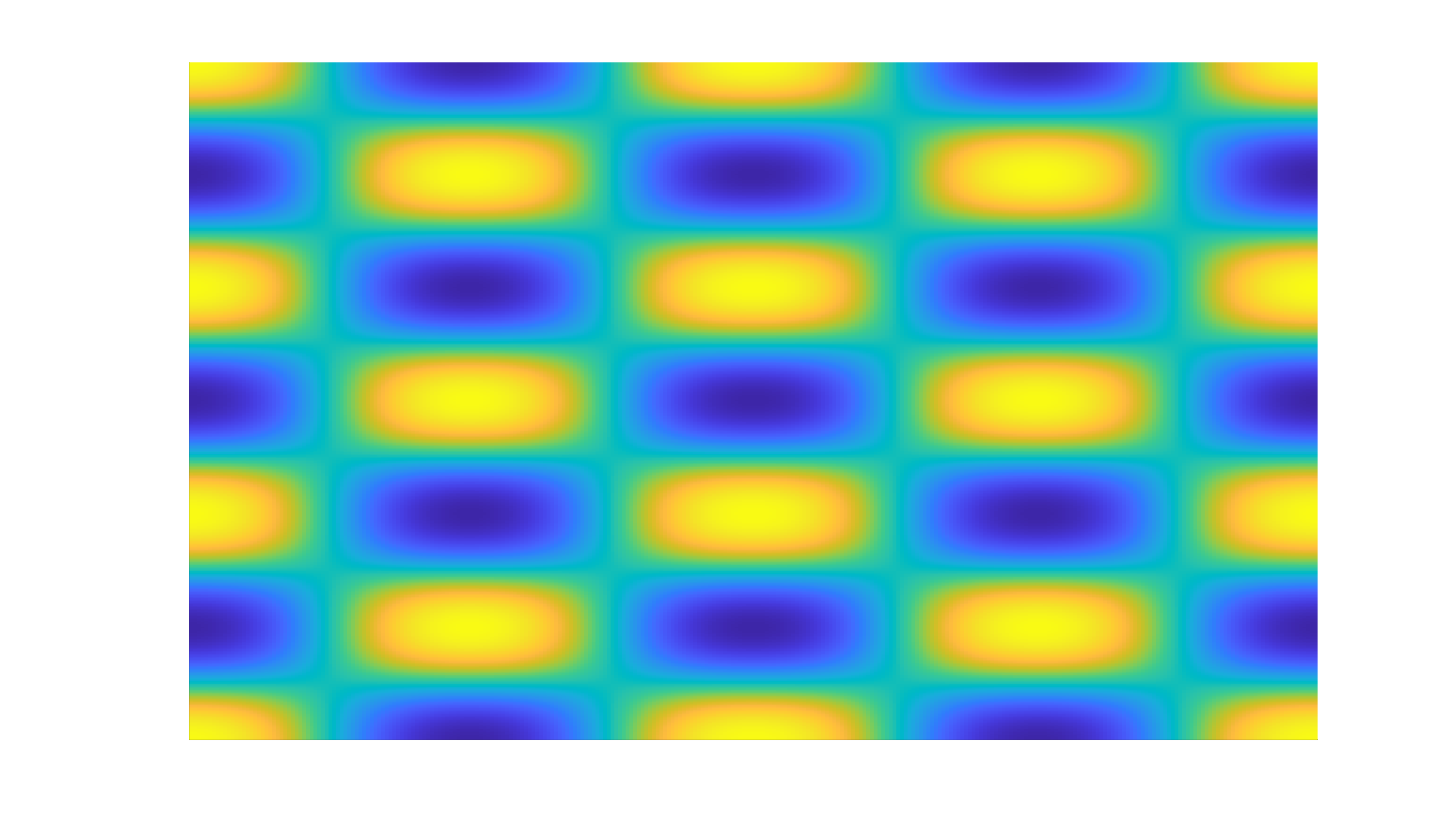}
		\subcaption{Spot pattern stationary state.}\label{fig:5_2}
	\end{minipage}
	\caption{Two stationary states of the 2D Ohta-Kawasaki equation \eqref{eq:OK}: From the different initial data $\varphi^i$ ($i=1,2$) in \eqref{eq:initial_OK}, the solution profile changes to each stationary state. Such an asymptotic dynamics of \eqref{eq:OK} is determined by the parameters $\epsilon$ and $m$ (see, e.g., \cite{MR3636312,MR4191556} for more details).}\label{fig:5}
\end{figure}

\begin{figure}[htbp]
	\begin{minipage}[b]{0.98\linewidth}
		\centering
		\includegraphics[width=\textwidth]{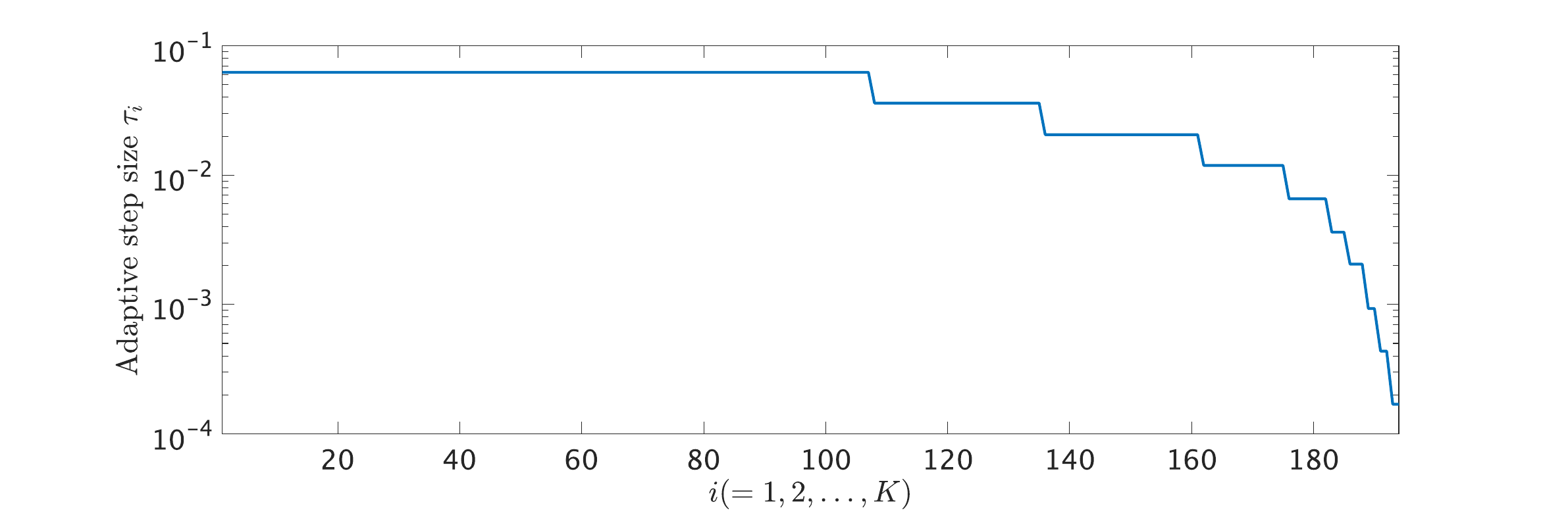}
		\subcaption{\jp{Resulting step size} $\tau_i$ (up to $K=194$) to the stripe pattern.}\label{fig:6_1}
	\end{minipage}\\
	\begin{minipage}[b]{0.98\linewidth}
		\centering
		\includegraphics[width=\textwidth]{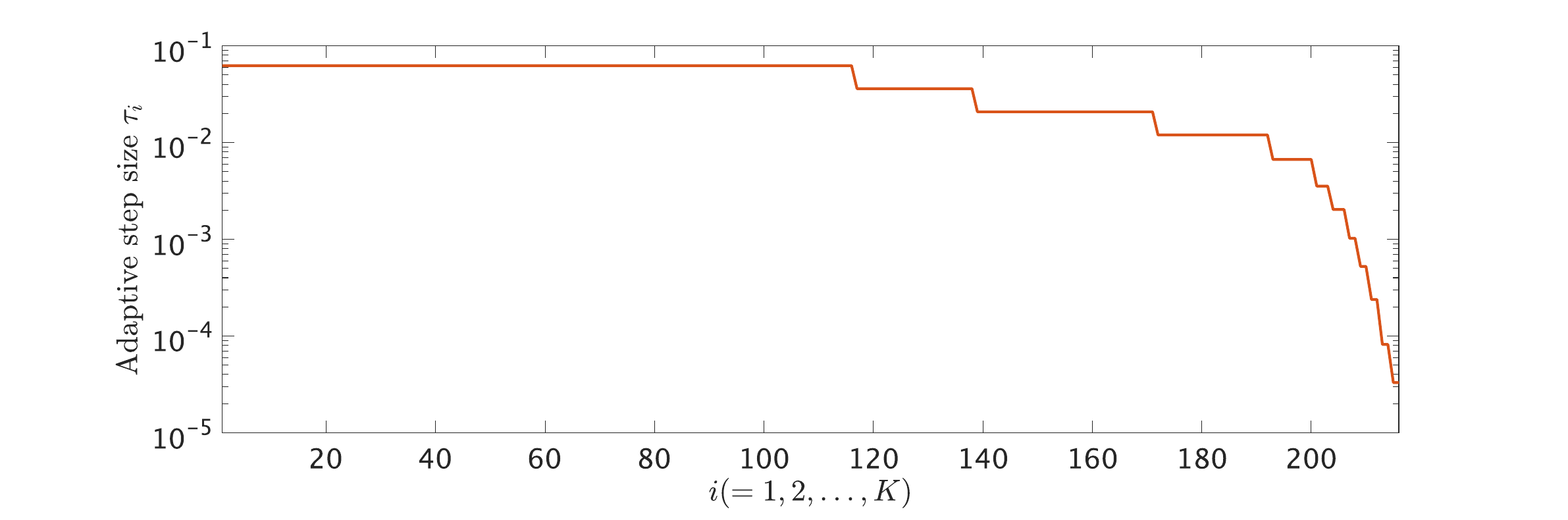}
		\subcaption{\jp{Resulting step size} $\tau_i$ (up to $K=216$) to the spot pattern.}\label{fig:6_2}
	\end{minipage}\\
	\begin{minipage}[b]{0.98\linewidth}
		\centering
		\includegraphics[width=\textwidth]{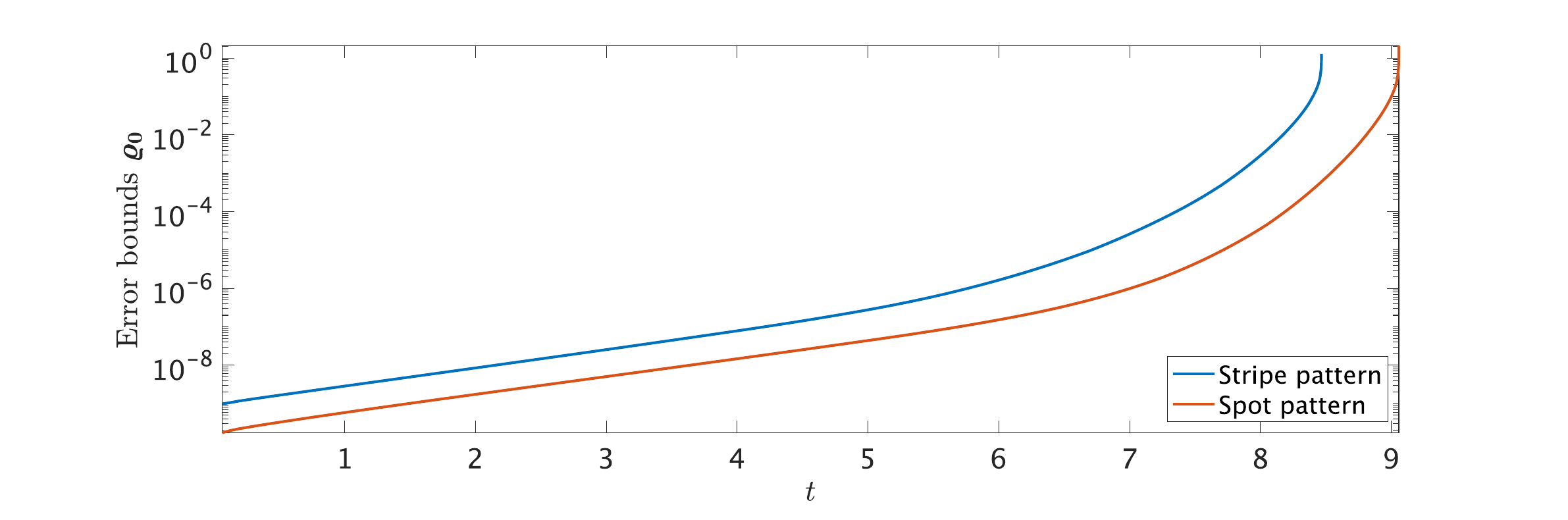}
		\subcaption{Error bounds $\bm{\varrho}_0$ \jp{on $[0,t]$}: Blue/red presents the case of stripe/spot pattern, respectively.}\label{fig:6_3}
	\end{minipage}
	\caption{Computational results of rigorous integration for the Ohta-Kawasaki equation \eqref{eq:OK}: \jp{Results of the adjusted step size are shown in (a) and (b). In both cases, the initial step size was set to $\tau_1=0.0625$.  After about 100 time steps, smaller step sizes were chosen to control the wrapping effect of the rigorous enclosure. This wrapping effect arises from an overestimation of the uniform bounds for the evolution operator presented in Section \ref{sec:solution_map}. By using Theorem \ref{thm:local_inclusion_multisteps}, the rigorous error bound displayed in (c) was obtained all at once over the entire interval $[0,t]$. Note that these results were not obtained using a step-by-step procedure.}}\label{fig:6}
\end{figure}

\paragraph{Stripe pattern state.}
Set the initial data as $\varphi^1$ in  \eqref{eq:initial_OK}. 
\jp{We chose the step size $\tau_i$ at each time step $J_i$, as illustrated in Fig\,\ref{fig:6}\,(a).}
Resulting computer-assisted proof assures that the solution of the IVP of \eqref{eq:OK} exists at least to $t_K=8.4666$ ($K=194$). The rigorous error bound $\bm{\varrho}_0$ via Theorem \ref{thm:local_inclusion_multisteps} satisfies
\begin{align}
	\sup_{t\in(0,8.4666]}\|a(t)-\ba(t)\|_{\omega}\le 1.2994=\bm{\varrho}_0.
\end{align}
We also have $\bm{\varrho}_K=1.2838$.

The results shown in Fig.\,\ref{fig:6}\,(a) and (c) indicate that after about 100 steps, \jp{while using a smaller step size helps to control the increasing error bound to some extent, such} error bound becomes larger and more difficult to handle. Eventually, this error escalation makes the hypothesis  of Theorem \ref{thm:local_inclusion_multisteps} unverifiable.
However, our result marks the first implementation of rigorous integration method for time-dependent PDEs in higher spatial dimensions, including nonlinear terms with derivatives.
We believe that this method constitutes a significant advance in the methodology of rigorous forward integration for PDEs and demonstrates satisfactory performance as a rigorous integrator.

\paragraph{Spot pattern state.}
Setting the initial data as $\varphi^2$ in \eqref{eq:initial_OK}, \jp{the choice of the step size is} displayed in Fig\,\ref{fig:6}\,(b). \jp{Computer-assisted proofs verifying the existence of a solution for the IVP of \eqref{eq:OK} succeeded} up to $t_K=9.0577$ (with $K=216$). The rigorous error bound $\bm{\varrho}_0$, as in Theorem \ref{thm:local_inclusion_multisteps}, is confirmed to satisfy the following inequality:
\begin{align}
	\sup_{t\in(0,9.0577]}\|a(t)-\ba(t)\|_{\omega}\le 2.0778=\bm{\varrho}_0,
\end{align}
Notably, $\bm{\varrho}_K$ at the final step is obtained as $2.0647$.

Fig\,\ref{fig:6}\,(c) also shows the escalation of the error bounds, due to the wrapping effect. In this case, the error bound is slightly smaller than that of the previous (stripe pattern) case, indicating the success of the integrator over more multiple time steps. 
To achieve successful long-time rigorous integration, it is necessary either to make the step size more reasonable, as discussed at the end of Section \ref{sec:2D-SH}, or to reduce the defects at each time step.

We conclude this section by mentioning that while it would be natural to consider a computer-assisted proof for global existence in the Ohta-Kawasaki equation, as was done in Section~\ref{sec:global_existence} for Swift-Hohenberg, this would require introducing a nontrivial construction for the trapping region (as the one used in the paper \cite{MR4444839,MR4388950}), and therefore we do not follow this path here. 

\section*{Conclusion}

In this paper, we have developed a general constructive method to compute solutions for IVPs of semilinear parabolic PDEs. Our approach combines the principles of semigroup theory with the practicality of computer-assisted proofs. 
A feature of our approach over the ordinary mathematical-analytic approach is the Fourier-Chebyshev series expansion, which derives a numerical candidate for the solution, thereby enhancing computational accessibility and feasibility.
We have constructed a two-component zero-finding problem as a direct derivation from the original PDEs.
To bring in the fixed-point argument, we have introduced a Newton-like operator at the numerically approximated solution. 
Central to our approach is the use of the evolution operator, which is the solution map for linearized equations at the numerical approximation. 
By inverting the linearized operator ``by hand'', the Newton-like operator is effectively derived in a more direct way.
The existence and local uniqueness of the fixed point of the Newton-like operator, together with its rigorously determined bounds, provide the rigorous inclusion of the solution to the IVPs.
Computer-assisted proof is done by numerically verifying the hypothesis of the contraction mapping via interval arithmetic.

We have further generalized our approach to a multi-step scheme that extends the existence time of solutions to IVPs.
This involves considering a coupled system of the zero-finding problem over multiple time steps, where the Newton-like operator at numerically approximated solution is derived by constructing the inverse of the linearized operator.
By numerically validating the hypothesis that the Newton-like operator becomes the contraction map, we obtain a rigorous inclusion of the solution over these multiple time steps. This multi-step approach should enhance the applicability of our approach not only for IVPs but also for boundary value problems in time, which is the subject of future studies.%

Constructing the trapping region based on the mechanism of convergence towards asymptotically stable equilibria, we have also presented proofs of global existence of solutions to three-dimensional PDEs (Swift-Hohenberg) converging to a nontrivial equilibrium. This is the first instance of CAPs being applied to global existence of solutions to 3D PDEs, offering a new perspective in the field. Moreover, our integrator has been applied to the 2D Ohta-Kawasaki equation, dealing with derivatives in the nonlinear terms. 
The multi-step scheme, together with the smoothing property of the evolution operator, has allowed for some effective control of the wrapping effect, and long-time rigorous forward integration has been successfully achieved. Our study not only contributes to the theoretical aspects of these equations but also provides practical tools and methods for computer-assisted proofs for their resolutions, opening new avenues for study and application in the field of computer-assisted proofs.

We conclude this paper by highlighting several potential extensions of our rigorous integrator.
First, our approach provides a cost-effective way to implement computer-assisted proofs for \jp{multi-steps}.
This aspect is particularly noteworthy because it allows the linearized problems to be solved independently and a priori at each time step, making the process naturally parallelizable and keeping the computational cost additive rather than multiplicative.
In other words, the manual inversion of the linearized operator, a critical part of our approach, significantly reduces the computational complexity by avoiding the inversion of large finite-dimensional matrices. 
\jp{We believe this feature could be beneficial in studying boundary value problems over time (such as periodic orbits and connecting orbits), which is a topic for exploration in future research.}

Second, exploring beyond the spatial domain assumption (rectangular domain) is an interesting direction of this study. For example, it is worth considering the potential of extending computer-assisted proofs to the time-dependent PDEs on more complicated domains, such as disks/spheres, etc.
This would require suitable basis functions to approximate solutions satisfying boundary conditions.
For effective computer-assisted proofs, the property of geometric decay for the coefficients of the series expansion (such as Fourier/Chebyshev) is desirable, and the property of Banach algebra  for discrete convolutions is useful for handling function products.

Finally, a challenging application of our approach lies in addressing the incompressible 3D Navier-Stokes equation on the domain $\mathbb{T}^3$.
Here, we could use the Fourier setting of the vector fields as described in \cite{BerBreLesVee21} to construct the associated infinite-dimensional ODEs.
One could also obtain a trapping region of the solution asymptotic to the stable equilibrium (zero) solution.
Achieving computer-assisted proofs for such complex equations is one of the ultimate goals in this line of research.

\appendix


\section{Chebyshev interpolation for the integrator}\label{appA:chebyshev}
In this appendix, we briefly explain how one constructs an approximate solution $\left(\ba^{(\bN)}_{\bk }(t)\right)_{\bk\in\FN}$ in \eqref{eq:appsol} using Chebyshev polynomials of the first kind defined by
\[
T_\ell(t)\bydef\cos(\ell\theta),\quad \theta = \arccos(t),\quad \ell\ge 0.
\]
The Chebyshev polynomial expansion is called \emph{Fourier cosine series in disguise} \cite{boyd}, that is, the change of variables by $t = \cos\theta$ makes both series expansions equivalent.
We recall the approximate solution $\bar{a}^{(\bm{N})}$ in \eqref{eq:appsol} denoted by
\[
\bar{a}^{(\bm{N})}_{\bk }(t)=\bar{a}_{0, \bk }+2 \sum_{\ell=1}^{n-1} \bar{a}_{\ell, \bk } T_{\ell}(t).
\]
Here the Chebyshev coefficients $(\ba_{\ell,\bk})_{\ell<n,\atop\bk\in\FN}$ are represented by the $d+1$ dimensional (finite) tensor.
To get such Chebyshev coefficients by numerics, we use the form
\begin{align*}
	\bar{a}^{(\bN)}_{\bk}(t)&=\bar{a}_{0, \bk}+2 \sum_{\ell=1}^{n-1} \bar{a}_{\ell, \bk} T_{\ell}(t)\\
	&=\bar{a}_{0, \bm{k}}+2 \sum_{\ell=1}^{n-1} \bar{a}_{\ell, \bm{k}} \cos(\ell\theta)\\
	&=\sum_{|\ell|<n} \bar{a}_{\ell, \bm{k}} e^{i\ell\theta}\quad (\mbox{with the cosine symmetry: }\bar{a}_{-\ell, \bm{k}}=\bar{a}_{\ell, \bm{k}}).
\end{align*}
The coefficients can be computed via the FFT algorithm using the form
\[
	\bar{a}_{\ell, \bm{k}}  = \frac1{2n-2}\sum_{j=0}^{2n-3} \bar{a}^{(\bm{N})}_{\bm{k}}(t_j) e^{-\pi i\frac{\ell j}{n-1}},\quad|\ell|<n,\quad\bk\in\FN,
\]
where $t_j$ is the Chebyshev points (of the second kind) of the $n$\,th order Chebyshev polynomial defined by $t_j\bydef \cos\left(\frac{\pi j}{n-1}\right)$.

Practically, we numerically compute the solution $\ba^{(\bN)}_{\bk}(t)$ of the truncated system \eqref{eq:ODEs_finite} by a certain numerical integrator (e.g., MATLAB's \texttt{ode113}/\texttt{ode15s} or exponential integrator etdrk4 \cite{trefethen}). Then we evaluate the function value of $\ba^{(\bN)}_{\bk}$ at $t_j$ and transform these values into the Chebyshev coefficients via the FFT. To fix the size $n$, we use a truncation method proposed in \cite{trefethen} to chop the Chebyshev series by an appropriate size.
Using the explicit construction of the Chebyshev coefficients described above, one can define the coefficients of the approximate solution $\ba$ in \eqref{eq:appsol} and rigorously compute the $\varepsilon$ and $\delta$ bounds described in Section \ref{sec:local_inclusion}.


\section{Functional analytic background}

Let us define another sequence space as
\begin{align}\label{eq:ell_inf}
	\ell_{\omega^{-1}}^\infty\bydef\left\{a=(a_{\bk })_{\bk\ge 0}:a_{\bk}\in \C,~\|a\|_{\infty,\omega^{-1}}\bydef\sup_{\bk\ge 0}|a_{\bk }|\omega_{\bk}^{-1}<+\infty\right\}.
\end{align}
It easily sees that this space is isometrically isomorphism to the \emph{dual} space of  $\ell_{\omega}^1$ defined in \eqref{eq:ell-one-space}, that is $\left(\ell_{\omega}^1\right)^\ast=\ell_{\omega^{-1}}^\infty$.
This fact is an analogy of the relation between the classical \emph{ell-one} space and \emph{ell-infinity} space.
\begin{lem}\label{lem:dual_es}
	Suppose that $c\in\ell_{\omega}^1$ and $a\in\ell_{\omega^{-1}}^\infty$. Then
	\begin{align}
		\left|\sum_{\bk\ge 0} a_{\bk}c_{\bk}\right|\le \|a\|_{\infty,\omega^{-1}}\|c\|_{\omega}.
	\end{align}
\end{lem}
\begin{proof}
	For $a\in\ell_{\omega^{-1}}^\infty$ and $\bk\ge 0$ it follows from the definition of the weighted supremum norm in \eqref{eq:ell_inf} that
	\begin{align}
		\frac{|a_{\bk}|}{\omega_{\bk}}\le  \|a\|_{\infty,\omega^{-1}}.
	\end{align}
	Therefore, we have
	\begin{align}
		\left|\sum_{\bk\ge 0} a_{\bk}c_{\bk}\right|&=\sum_{\bk\ge 0} |c_{\bk}||a_{\bk}|\\
		&\le \sum_{\bk\ge 0} \left(\|a\|_{\infty,\omega^{-1}}\omega_{\bk}\right)|c_{\bk}|\\
		&= \|a\|_{\infty,\omega^{-1}}\|c\|_{\omega}.\qedhere
	\end{align}
\end{proof}
Using Lemma \ref{lem:dual_es}, we have an estimate of the discrete convolution defined in \eqref{eq:discrete_convolution} for $c\in\ell_{\omega}^1$ and $a\in\ell_{\omega^{-1}}^\infty$.

\begin{lem}\label{lem:conv_es}
	Suppose that $c\in\ell_{\omega}^1$ and $a\in\ell_{\omega^{-1}}^\infty$. Then
	\begin{align}
		\left|\left(a*c\right)_\bk\right| 
		&\le \max\left\{|a_{\bk}|,\sup_{|\bk^\prime| \notin\bm{F_{1}}} \left|a_{\bk-\bk^\prime}\right|\omega_{|\bk^\prime|}^{-1}\right\}\|c\|_{\omega}\label{eq:conv_bounds1}
	\end{align}
	holds for $\bk\ge 0$.
\end{lem}
\begin{proof}
	For $\bk\ge 0$, it follows from \eqref{eq:discrete_convolution} and Lemma \ref{lem:dual_es} that
	\begin{align}
		\left|\left(a*c\right)_\bk\right| &= \left|\sum_{\bk^\prime \in \Z^d}a_{\bk-\bk^\prime}c_{\bk^\prime}\right|\\
		&\le |a_{\bk}||c_0|+\sum_{|\bk^\prime| \notin\bm{F_{1}}} \left|a_{\bk-\bk^\prime}\right|\left|c_{\bk^\prime}\right|\\
		&\le \max\left\{|a_{\bk}|,\sup_{|\bk^\prime| \notin\bm{F_{1}}} \left|a_{\bk-\bk_1}\right|\omega_{|\bk^\prime|}^{-1}\right\}\|c\|_{\omega}.\qedhere
	\end{align}
\end{proof}


\section{Proof of Corollary \ref{cor:solutionmap} and each bound}\label{sec:Wtau_bound}

\begin{lem}
	Under the same assumption of Theorem \ref{thm:ev_op}, the unique solution of 
	\begin{align}\label{eq:cinf_integral_alt}
		c^{(\infty)}(t)=e^{\cL_\infty (t-s)}\phi^{(\infty)} + \int_s^t e^{\cL_\infty (t-r)}({\rm Id} -\varPi^{(\bbm)})\cQ D\cN(\ba(r)) c^{(\infty)}(r) d r
	\end{align}
	exists for all $\phi\in\ell_{\omega}^1$. Then the following estimate of the evolution operator $\bU^{(\infty)}(t,s)$ holds
	\[
	\left\|\bU^{(\infty)}(t,s)\phi^{(\infty)}\right\|_{\omega}\le W^{(\infty)}(t,s)\left\|\phi ^{(\infty)}\right\|_{\omega},\quad\forall\phi\in\ell_{\omega}^1,
	\]
	where $W^{(\infty)}(t,s)$ is defined by
	\begin{align}\label{eq:W^inf}
		W^{(\infty)}(t,s)\bydef e^{-(1-\xi)|\mu_\ast |(t-s) + C_\infty g(\|\ba\|){(t-s)^{1-\gamma}}/{(1-\gamma)}}.
	\end{align}
\end{lem}

\begin{proof}
	From Theorem \ref{thm:ev_op}, the unique solution of \eqref{eq:cinf_integral_alt} exists and it yields that
	\begin{align}
		\|c^{(\infty)}(t)\|_\omega \le e^{-|\mu_\ast |(t-s)}\|\phi^{(\infty)}\|_\omega + \int_{s}^t C_\infty(t-r)^{-\gamma}e^{-(1-\xi)|\mu_\ast |(t-r)}g(r) \|c^{(\infty)}(r)\|_\omega dr,
	\end{align}
	where $g(r)$ is the same as that in the proof in Theorem \ref{thm:ev_op}. Letting $y(t)\bydef e^{(1-\xi)|\mu_\ast |(t-s)}\|c^{(\infty)}(t)\|_\omega$, it follows that
	\begin{align}
		y(t)\le e^{-\xi|\mu_\ast |(t-s)}\|\phi^{(\infty)}\|_\omega + C_\infty\int_{s}^t (t-r)^{-\gamma}g(r)y(r)dr.
	\end{align}
	From the Gronwall's inequality we have
	\begin{align}
		y(t)\le \|\phi^{(\infty)}\|_\omega \exp\left(C_\infty\int_{s}^t(t-r)^{-\gamma}g(r)dr\right).
	\end{align}
	This directly yields the bound \eqref{eq:W^inf}.
\end{proof}

\begin{lem}\label{lem:other_bounds_alt}
	Letting
	\begin{align}
		\iota = (1-\xi)|\mu_\ast |,\quad \tilde{\vartheta}=\frac{C_\infty g(\|\ba\|)}{1-\gamma},\quad \vartheta=C_\infty g(\|\ba\|) \mathrm{B}(1-\gamma,1-\gamma),
	\end{align}
	rewrite $W^{(\infty)}(t,s)$ in \eqref{eq:W^inf} and $W^{(\infty)}_q(t,s)$ in \eqref{eq:W^inf_q} as $e^{-\iota(t-s) + \tilde{\vartheta}(t-s)^{1-\gamma}}$ and $C_\infty (t-s)^{-\gamma}e^{-\iota(t-s) + \vartheta(t-s)^{1-\gamma}}$, respectively.	
	Define the constants $\jp{W_{\infty}^{\cS_{J}}}>0$, $\jp{\overline{W}_{\infty}^{\cS_{J}}}\ge0$, $\jp{{\overline{W}}_{\infty,q}^{\prime\cS_{J}}}\ge 0$, $\jp{{\doverline{W}}_{\infty,q}^{\prime\cS_{J}}}\ge0$ as
	\begin{align}
		\jp{W_{\infty}^{\cS_{J}}}&\bydef e^{\tilde{\vartheta}\tau^{1-\gamma}}\\[1mm]
		\jp{\overline{W}_{\infty}^{\cS_{J}}}&\bydef \left(\frac{1-e^{-\iota\tau}}{\iota}\right)e^{\tilde{\vartheta}\tau^{1-\gamma}}\\[1mm]
		\jp{{\overline{W}}_{\infty,q}^{\prime\cS_{J}}}&\bydef \left(\frac{\tau^{1-\gamma}}{1-\gamma}\right)C_\infty e^{{\vartheta}\tau^{1-\gamma}}\\[1mm]
		\jp{{\doverline{W}}_{\infty,q}^{\prime\cS_{J}}}&\bydef\frac{\tau^{2-\gamma}}{(1-\gamma)(2-\gamma)}C_\infty e^{{\vartheta}\tau^{1-\gamma}}\label{eq:Wbar_infty}
	\end{align}
	respectively. Then these bounds obey the inequalities \eqref{eq:ineqW_infty_sup_classic}, \eqref{eq:ineqW_infty_classic}, \eqref{eq:ineqhatW_infty_classic}, and \eqref{eq:ineqbarW_infty_classic}.
\end{lem}
\begin{proof}
	First, it follows from \eqref{eq:W^inf} that
	\begin{align}
		\sup_{(t,s)\in \jp{\cS_{J}}} W^{(\infty)}(t,s) = \sup_{(t,s)\in \jp{\cS_{J}}} e^{-\iota(t-s) + \tilde{\vartheta}(t-s)^{1-\gamma}}\le e^{\tilde{\vartheta} \tau^{1-\gamma}} =\jp{W_{\infty}^{\cS_{J}}}.
	\end{align}
	Second, we note that
	\begin{align}
		\sup _{(t,s)\in \jp{\cS_{J}}} \int_{s}^{t} W^{(\infty)}(r, s) d r 
		&= \sup _{(t,s)\in \jp{\cS_{J}}} \int_{s}^{t}  e^{-\iota(r-s) + \tilde{\vartheta}(r-s)^{1-\gamma}} d r\\
		&\le \sup _{(t,s)\in \jp{\cS_{J}}} e^{\tilde{\vartheta}(t-s)^{1-\gamma}} \int_{s}^{t}e^{-\iota(r-s)}dr\\
		&\le \sup _{(t,s)\in \jp{\cS_{J}}} e^{\tilde{\vartheta}(t-s)^{1-\gamma}}\left(\frac{1-e^{-\iota(t-s)}}{\iota}\right)\le \jp{\overline{W}_{\infty}^{\cS_{J}}}
	\end{align}
	and that
	\begin{align}
		\sup _{(t,s)\in \jp{\cS_{J}}} \int_{s}^{t} W^{(\infty)}_{q}(t, r) d r 
		&= \sup _{(t,s)\in \jp{\cS_{J}}}\int_{s}^{t}C_\infty (t-r)^{-\gamma}e^{-\iota(t-r) + \vartheta(t-r)^{1-\gamma}} d r\\
		&\le C_\infty \sup _{(t,s)\in \jp{\cS_{J}}} e^{\vartheta(t-s)^{1-\gamma}}\int_{s}^{t}  (t-r)^{-\gamma}e^{-\iota(t-r)} d r\\
		&\le C_\infty \sup _{(t,s)\in \jp{\cS_{J}}} e^{\vartheta(t-s)^{1-\gamma}}\frac{(t-s)^{1-\gamma}}{1-\gamma}\le \jp{{\overline{W}}_{\infty,q}^{\prime\cS_{J}}}.
	\end{align}
	Third, note that
	\begin{align}
		\sup _{(t,s)\in \jp{\cS_{J}}}\int_{s}^{t} \int_{s}^{r} W^{(\infty)}_{q}(r, \sigma)d \sigma d r
		&=\sup _{(t,s)\in \jp{\cS_{J}}}\int_{s}^{t} \int_{s}^{r} C_\infty (r-\sigma)^{-\gamma} e^{-\iota(r-\sigma) + \vartheta (r-\sigma)^{1-\gamma}} d \sigma d r\\
		&\le C_\infty \sup _{(t,s)\in \jp{\cS_{J}}}\int_{s}^{t} e^{ \vartheta (r-s)^{1-\gamma}} \int_{s}^{r} (r-\sigma)^{-\gamma} e^{-\iota(r-\sigma)}d \sigma d r\\
		&\le C_\infty \sup _{(t,s)\in \jp{\cS_{J}}}\int_{s}^{t} e^{ \vartheta (r-s)^{1-\gamma}} \frac{(r-s)^{1-\gamma}}{1-\gamma} d r\\
		&\le \frac{C_\infty}{1-\gamma} \sup _{(t,s)\in \jp{\cS_{J}}} e^{\vartheta (t-s)^{1-\gamma}} \frac{(t-s)^{2-\gamma}}{2-\gamma}\le \jp{{\doverline{W}}_{\infty,q}^{\prime\cS_{J}}}.
	\end{align}
	Finally, we have
	\begin{align}
		\sup _{(t,s)\in \jp{\cS_{J}}}\int_{s}^{t} W^{(\infty)}_{q}(t, r)(r -s) d r
		&= \sup _{(t,s)\in \jp{\cS_{J}}} \int_{s}^{t} C_\infty(t-r)^{-\gamma} e^{-\iota(t-r) + \vartheta(t-r)^{1-\gamma}}(r-s)dr\\
		&\le C_\infty \sup _{(t,s)\in \jp{\cS_{J}}} e^{\vartheta (t-s)^{1-\gamma}}\int_{s}^{t} (t-r)^{-\gamma} (r -s)d r\\
		&\le C_\infty \sup _{(t,s)\in \jp{\cS_{J}}} e^{\vartheta (t-s)^{1-\gamma}}\frac{(t-s)^{2-\gamma}}{(1-\gamma)(2-\gamma)}\le\jp{{\doverline{W}}_{\infty,q}^{\prime\cS_{J}}}.
	\end{align}
\end{proof}

\begin{proof}[Proof of Corollary \ref{cor:solutionmap}]
		Taking the $\ell_{\omega}^1$ norm of $b(t)$, it follows from \eqref{eq:Wm_bound}, \eqref{eq:Cm_bound}, \eqref{eq:Cinf_bound} and  \eqref{eq:Winfts_bound_classic} that
	\begin{align}
		\|b^{(\bm{m})}(t)\|_{\omega}&\le \jp{W_{m,0}^{\cS_{J}}}\|\phi^{(\bm{m})}\|_{\omega}+\jp{W_{m,q}^{\cS_{J}}}\jp{\cE_{m,\infty}^{J}}\int_s^t\|b^{(\infty)}(r)\|_{\omega} dr\label{eq:bm_omegaq_classic}\\
		\|b^{(\infty)}(t)\|_{\omega}&\le W^{(\infty)}(t,s)\|\phi^{(\infty)}\|_{\omega}+\jp{\cE_{\infty,m}^{J}}\int_s^tW^{(\infty)}_{{q}}(t,r)\|b^{(\bm{m})}(r)\|_{\omega} dr.\label{eq:binf_omegaq_classic}
	\end{align}
	Plugging each estimate in the other one, we obtain
	\begin{align}
		&\left\|b^{(\bm{m})}(t)\right\|_{\omega}\\
		&\le \jp{W_{m,0}^{\cS_{J}}}\|\phi^{(\bm{m})}\|_{\omega}+\jp{W_{m,q}^{\cS_{J}}}\jp{\cE_{m,\infty}^{J}}\int_s^t\left(W^{(\infty)}(r,s)\|\phi^{(\infty)}\|_{\omega}+\jp{\cE_{\infty,m}^{J}}\int_s^rW^{(\infty)}_{{q}}(r,\sigma)\|b^{(\bm{m})}(\sigma)\|_{\omega} d\sigma\right) dr\\
		&\le \jp{W_{m,0}^{\cS_{J}}}\|\phi^{(\bm{m})}\|_{\omega}+\jp{W_{m,q}^{\cS_{J}}}\jp{\overline{W}_{\infty}^{\cS_{J}}}\jp{\cE_{m,\infty}^{J}}\|\phi^{(\infty)}\|_{\omega}+\jp{W_{m,q}^{\cS_{J}}}\jp{{\doverline{W}}_{\infty,q}^{\prime\cS_{J}}} \jp{\cE_{m,\infty}^{J}}\jp{\cE_{\infty,m}^{J}}\|b^{(\bm{m})}\|
	\end{align}
	and
	\begin{align}
		&\|b^{(\infty)}(t)\|_{\omega}\\
		&\le W^{(\infty)}(t,s)\|\phi^{(\infty)}\|_{\omega}+\jp{\cE_{\infty,m}^{J}}\int_s^tW^{(\infty)}_{{q}}(t,r)\left(\jp{W_{m,0}^{\cS_{J}}}\|\phi^{(\bm{m})}\|_{\omega}+\jp{W_{m,q}^{\cS_{J}}}\jp{\cE_{m,\infty}^{J}}\int_s^r\|b^{(\infty)}(\sigma)\|_{\omega} d\sigma\right) dr\\
		&\le W^{(\infty)}(t,s)\|\phi^{(\infty)}\|_{\omega}+\jp{W_{m,0}^{\cS_{J}}}\jp{{\overline{W}}_{\infty,q}^{\prime\cS_{J}}} \jp{\cE_{\infty,m}^{J}}\|\phi^{(\bm{m})}\|_{\omega} + \jp{W_{m,q}^{\cS_{J}}}\jp{{\doverline{W}}_{\infty,q}^{\prime\cS_{J}}} \jp{\cE_{m,\infty}^{J}}\jp{\cE_{\infty,m}^{J}}\|b^{(\infty)}\|.
	\end{align}
	Since $\tilde{\kappa}=1-\jp{W_{m,q}^{\cS_{J}}}\jp{{\doverline{W}}_{\infty,q}^{\prime\cS_{J}}} \jp{\cE_{m,\infty}^{J}}C_\infty>0$ holds from the sufficient condition of Theorem \ref{thm:solutionmap}, we have
	\begin{align}
		\|b^{(\bm{m})}\| &\le \frac{\jp{W_{m,0}^{\cS_{J}}}\|\phi^{(\bm{m})}\|_{\omega}+\jp{W_{m,q}^{\cS_{J}}}\jp{\overline{W}_{\infty}^{\cS_{J}}} \jp{\cE_{m,\infty}^{J}}\|\phi^{(\infty)}\|_{\omega}}{\tilde{\kappa}}\label{eq:bm_bounds_classic}\\
		\|b^{(\infty)}\| &\le \frac{\jp{W_{\infty}^{\cS_{J}}}\|\phi^{(\infty)}\|_{\omega}+\jp{W_{m,0}^{\cS_{J}}}\jp{{\overline{W}}_{\infty,q}^{\prime\cS_{J}}} \jp{\cE_{\infty,m}^{J}}\|\phi^{(\bm{m})}\|_{\omega}}{\tilde{\kappa}}.\label{eq:binf_bounds_classic}
	\end{align}
	Therefore, it follows that
	\begin{align*}
		\|b\| &\le \|b^{(\bm{m})}\| + \|b^{(\infty)}\|\\
		&=\tilde{\kappa}^{-1}\left\|\begin{pmatrix}
			\jp{W_{m,0}^{\cS_{J}}} & \jp{W_{m,q}^{\cS_{J}}}\jp{\overline{W}_{\infty}^{\cS_{J}}} \jp{\cE_{m,\infty}^{J}}\\
			\jp{W_{m,0}^{\cS_{J}}}\jp{{\overline{W}}_{\infty,q}^{\prime\cS_{J}}} C_\infty & \jp{W_{\infty}^{\cS_{J}}}
		\end{pmatrix}\begin{pmatrix}
			\left\|\psi^{(\bm{m})}\right\|_{\omega}\\\left\|\psi^{(\infty)}\right\|_{\omega}
		\end{pmatrix}
		\right\|_1\\
		&\le \bm{\jp{W^{\cS_{J}}}}\|\psi\|_{\omega}. \qedhere
	\end{align*}
\end{proof}


\section{Computer-assisted proofs for the equilibria problems in Swift-Hohenberg} \label{sec:equilibria_SH}
From (\ref{eq:the_ODEs}), it is easy to see that the Chebyshev expansion of the equilibrium correspond to the solution of the zero finding problem:
\begin{align}
F_\bk (a)  &= \mu_\bk a_{ \bk} + N_{\bk}(a). 
\end{align}
To prove the existence of a solution, we will used a Newton-Kantorovich argument by using the following Theorem.
\begin{thm}\label{thm::RadiiNonLin}
Let $X$ and $Y$ be Banach spaces and $F: X \rightarrow Y$ be a Fr\'echet differentiable mapping. Suppose $\bx \in X$, $A^\dagger \in B(X,Y)$ and $A \in B(X,Y)$. Moreover assume that $A$ is injective. Let $Y_0,~Z_0$ and $ Z_1$ be positive constants and $Z_2:(0,\infty) \rightarrow [0,\infty)$ be a non-negative function satisfying 
\begin{align*}
\| AF(\bx) \|_X \leq Y_0, \\
\| Id - A A^\dagger \|_{B(X)} \leq Z_0,\\
\| A [DF(\bx) - A^\dagger] \|_{B(X)} \leq Z_1,
\end{align*}
and
\begin{align*}
\|A[DF(c) - DF(\bx)] \|_{B(X)} \leq Z_2(r)r, \quad \text{ for all } \quad c \in \overline{B_r(\bx)} \text{ and all } r>0.
\end{align*}
Define
\begin{align*}
p(r) = Z_2(r)r^2 - (1-Z_0-Z_1)r + Y_0.
\end{align*}
If there exists $r_0>0$ such that $p(r_0) < 0 $, then there exists a unique $\tilde{x} \in \overline{B_{r_0}(\bx)} $ satisfying $F(\tilde{x}) = 0$.
 \end{thm} 
We notice that Theorem \ref{thm::RadiiNonLin} is a variation of Theorem \ref{thm:radii_polynomial} for a non linear map. If $\tilde{a} \in \ell_\omega^1$ is the solution such that $F(\tilde{a} ) = 0$ and let $\ba$ be a numerical approximation such that $F(\ba) \approx 0$. Using Theorem \ref{thm::RadiiNonLin}, we will prove there exists a $r > 0$  such that $\tilde{a} \in \overline{B_r(\ba)}$ similarly as we did in section \ref{sec:fundamental_sol}. Let $\bm{m} = \{ m_1, \hdots , m_d \}$ be the size of the finite set $\Fm$, we define
 \begin{align*}
 F_\bk^{(\bm{m})}(a) \bydef  \left\{ 
 \begin{tabular}{cl}
 $ \mu_\bk a_{ \bk} + N_{\bk}(a)$, & if $\bk \in \Fm$,\\
 $0$, & if $\bk \notin \Fm$.
 \end{tabular}
 \right. 
 \end{align*}
Let $h \in \ell_\omega^1$, we define the operators $A^\dagger: \ell_\omega^1 \rightarrow \ell_\omega^1 $ and $A: \ell_\omega^1 \rightarrow \ell_\omega^1 $ by
\begin{align*}
(A^\dagger h)_\bk \bydef  \left\{ 
 \begin{tabular}{cl}
 $ (DF^{(\bm{m})}(\ba)h^{(m)})_\bk $, & if $\bk \in \Fm$,\\
 $\mu_\bk h_{ \bk}$, & if $\bk \notin \Fm$,
 \end{tabular}
 \right. 
\end{align*}
and 
\begin{align*}
(A h)_\bk \bydef  \left\{ 
 \begin{tabular}{cl}
 $ (A^{(\bm{m})}h^{(m)})_\bk $, & if $\bk \in \Fm$,\\
 $\frac{1}{\mu_\bk}  h_{ \bk}$, & if $\bk \notin \Fm$,
 \end{tabular}
 \right. 
\end{align*}
where $A^{(\bm{m})}$ is the numerical inverse of $DF^{(\bm{m})}(\ba)$. We can now compute the bounds $Y_0, Z_0, Z_1$ and $Z_2$ from Theorem \ref{thm::RadiiNonLin}. \\

\noindent{\bf The bound \boldmath$Y_0$\unboldmath.} 
The bounds $Y_0$ can be define by the inequality
\begin{align*}
 \| (AF(\bar{a})) \|_{{\omega}} &=  \| A^{(\bm{m})} F^{(\bm{m})}(\bar{a})\|_\omega  + \sum_{\bk \notin  \Fm} \left| \frac{1}{\mu_\bk} N_\bk(\ba) \right|  \omega_\bk \\
 & \leq \| A^{(\bm{m})} F^{(\bm{m})}(\bar{a})\|_\omega  + \left( \sup_{\bk \notin  \Fm}  \frac{1}{| \mu_\bk | } \right)  \sum_{\bk \notin  \Fm} \left|  N_\bk(\ba) \right|  \omega_\bk \bydef Y_0
\end{align*}
Since the nonlinear term $N_\bk(\ba)$ is a polynomial and $(\ba)_\bk = 0$ for all $\bk \notin \Fm$, the sum in the second term of the inequality is finite and can be rigorously computed using interval arithmetic.\\

\noindent{\bf The bound \boldmath$Z_0$\unboldmath.} Let $h \in \ell_\omega^1$ , the operator $B \bydef I - A A^\dagger$ is given component-wise by

\begin{align*}
(Bh)_\bk = 
\begin{cases}
\left(  Id^{(\bm{m})}   - A^{(\bm{m})}  DF^{(\bm{m})} (\ba) \right) h^{(\bm{m})}  &\bk \in F_{\bm{m}},
\\
0, & \bk \notin F_{\bm{m}}.
\end{cases}
\end{align*}
Then, the computation of $Z_0$ is finite and given by
\begin{align*}
\| B \|_{B(\ell_\omega^1)}  =  \| Id^{(\bm{m})} - A^{(\bm{m})} DF^{(\bm{m})}(\ba) \|_{B(\ell_\omega^1)}  \bydef Z_0.
\end{align*}
\noindent{\bf The bound \boldmath$Z_1$\unboldmath.} For any $h \in B_1(0)$, let
\[
z \bydef  [DF(\bar x)-A^\dagger]h 
\]
which is given component-wise by 
\begin{align*}
z_{\bk} & = \left\{
\begin{matrix}
\left(DN(\ba)*h^{(\infty)} \right)_\bk & \bk \in \Fm,\\
\left(DN(\ba)*h \right)_\bk & \bk \notin \Fm.
\end{matrix}
\right. 
\end{align*}
To simplify the notation of the bound $Z_1$, lets us first define component-wise uniform bounds $ \hat z_{\bk}$  for $\bk \in  \Fm$. We find
\begin{align*}
|z_\bk| &= \left|  \left(DN(\ba)*h^{(\infty)} \right)_\bk  \right| 
\leq     \max_{\bm{j} \in \N^d } \left\{ \frac{\left(DN(\ba) \right)_{\bk - \bm{j} }}{\omega_{\bm{j}}} \right\}   \bydef \hat z_\bk .
\end{align*}
Since $\ba$ is finite and $DN$ is polynomial, the computation of $\hat z_\bk $ is also finite and computed similarly as in section \ref{sec:fundamental_sol}. We also need to find bounds for the tails given by
\begin{align*}
\|A \pi^{(\infty)} z \|_\omega & = \sum_{\bk \notin \Fm} \left|  \frac{1}{\mu_\bk}  \left(DN(\ba)*h\right)_\bk \right| \omega_\bk \leq \sup_{\bk \notin \Fm} \left( \frac{1}{|\mu_\bk|} \right)  \left\| DN(\ba) \right\|_{B(\ell_\omega^1)}  .
\end{align*}
Let the linear operator $|A|$ represents the absolute value component-wise of $A$, then  
\begin{align*}
\| A [DF(\bar x)-A^\dagger] \|_{B(\ell_\omega^1)} &= \sup_{\|h\|_\omega \leq 1} \| Az \|_\omega ,\\
& = \sup_{\|h\|_\omega \leq 1} \left( \sum_{\bk \in \Fm } | (Az)_\bk| \omega_\bk +  \| A\pi^{(\infty)}z \|_\omega \right) ,\\
&\leq \| |A^{(\bm{m})} |\pi^{(\bm{m}) } \hat z \|_\omega + \sup_{\bk \notin \Fm} \left( \frac{1}{|\mu_\bk|} \right)  \left\| DN(\ba) \right\|_{B(\ell_\omega^1)} \\ 
&\leq \| |A^{(\bm{m})} |\pi^{(\bm{m}) } \hat z \|_\omega + \sup_{\bk \notin \Fm} \left( \frac{1}{|\mu_\bk|} \right) 3 \| \ba^2 \|_\omega \bydef Z_1.
\end{align*}
\noindent{\bf The bound \boldmath$Z_2$\unboldmath.}
Let $h \in \ell_\omega^1$ with  $\| h\|_\omega \le 1$ and $c  \in \overline{B_{r} (\ba)}$, we define 
$$y \bydef  (DF(c) - DF(\ba))h $$
given component-wise by
\begin{align*}
y_\bk = -3[ (c^2 - \ba^2) *h ]_\bk  = -3[(c + \ba)* (c - \ba) *h ]_\bk  .
\end{align*}
Since, $c \in B_r(\ba) $, there exists a $\hat h \in B_1(0)$ such that $c = \ba + r \hat h$ and we can bound $y$ by
\begin{align*}
\| y \|_\omega &=   3 \| (c + \ba)* (c - \ba) *h  \|_\omega, \\
&= 3r \| (2\ba + r \hat h )* \hat h  *h  \|_\omega, \\
&\leq  3r \left(  2 \| \ba \|_\omega  + r \| \hat h \|_\omega \right)    \| \hat h \|_\omega  \| h \|_\omega ,\\
&\leq  3r \left(  2 \| \ba \|_\omega  + r  \right) . 
\end{align*}
Then, we have
\begin{align*}
\| A[ DF(c) - DF(\ba) ]h \|_\omega & =   \| Ay \|_\omega ,\\
&\leq \| A \|_{B(\ell_\omega^1)} \| y \|_\omega ,\\
&\leq  3r \| A \|_{B(\ell_\omega^1)} \left(  2 \| \ba \|_\omega  + r  \right) \bydef Z_2(r)r.
\end{align*}
Using interval arithmetic in MATLAB, we can rigorously compute the bounds $Y_0,~Z_0,~Z_1$ and $Z_2(r)$ and by using Theorem \ref{thm::RadiiNonLin}, we can find a $r>0$ such that $\tilde{a} \in \overline{B_r(\ba)}$.

Finally, we list up the bounds for computer-assisted proofs of the equilibrium solution to the Swift-Hohenberg equation in the following table:
\begin{table}[h]\label{Bounds_Equilibrium_SH}
\centering
\caption{ Bounds of the Swift-Hohenberg equilibria from section \ref{3D Swift-Hohenberg equation} and  \ref{sec:2D-SH} .   } 
\begin{tabular}{|c|c|c|c|c|c|}
\hline
\rowcolor[HTML]{C0C0C0} 
Equilibrium & $Y_0$ & $Z_0$ & $Z_1$ & $Z_2(r)$ & $r$ \\ \hline
3D                          &  $3.0086\cdot 10^{-8}$      &  $1.1275  \cdot 10^{-11}$   &    $ 0.0025$    &   $ 0.0469r +   0.0217$    & $3.0161\cdot 10^{-8}$    \\ \hline
2D (Stripe)                 &     $2.4471 \cdot 10^{-11}$   &   $3.0407 \cdot 10^{-10}$    &     $0.0013$   &  $ 0.0003r +  0.0013$     &  $2.4502 \cdot 10^{-11}$   \\ \hline
2D (Spot)                   &    $1.3240 \cdot 10^{-10}$     &   $3.0407 \cdot 10^{-10}$      &  $ 0.0024 $      &    $0.0003   r+ 0.0024$   &   $1.3271 \cdot 10^{-10}$    \\ \hline
\end{tabular}
\end{table}


\section*{Acknowledgement}
AT is partially supported by Researcher+, JST Comprehensive Support Project for the Strategic Professional Development Program for Young Researchers and JSPS KAKENHI Grant Numbers JP22K03411, JP21H01001, and JP20H01820.

\section*{Declaration}
The authors declare that they have no known competing financial interests or personal relationships that could have appeared to influence the work reported in this paper.


\bibliographystyle{unsrt}
\bibliography{papers}

\end{document}